\documentclass[letterpaper, 11pt]{article}

\usepackage[dvipsnames]{xcolor}
\usepackage{
  amsmath, amsthm, amssymb, mathtools, dsfont, units, bbm,  
  graphicx, wrapfig, subfig, float,                         
  listings, color, inconsolata, pythonhighlight,            
  fancyhdr, hyperref, enumitem, framed                      
}

\usepackage[ruled,vlined]{algorithm2e}
\usepackage{tabularray}

\usepackage{newpxtext, newpxmath, inconsolata}
\usepackage[cal=dutchcal,scr=boondoxo,bfscr,frak=pxtx,bb=dsfontserif]{mathalpha}

\usepackage[
    backend=biber,
    style=numeric,
    giveninits=true,
    maxnames = 999
]{biblatex}
\DeclareFieldFormat[article]{volume}{\mkbibbold{#1}} 
\DeclareFieldFormat[article]{journaltitle}{#1} 
\DeclareFieldFormat[article]{title}{\mkbibquote{\mkbibemph{#1}}} 
\DeclareFieldFormat[inproceedings]{title}{\mkbibquote{\mkbibemph{#1}}} 
\DeclareFieldFormat[inproceedings]{booktitle}{#1} 
\DeclareFieldFormat{title}{\mkbibquote{\mkbibemph{#1}}} 
\renewbibmacro*{in:}{} 
\AtEveryBibitem{\clearfield{eprint}} 
\AtEveryBibitem{\clearfield{primaryclass}} 

\usepackage[left=1in, right=1in, top=1.0in, bottom=.9in, headsep=.2in, footskip=0.35in]{geometry}

\usepackage[bottom]{footmisc}

\allowdisplaybreaks

\usepackage[font={it,footnotesize}]{caption}

\hypersetup{colorlinks=true, linkcolor=RoyalBlue, citecolor=RedOrange, urlcolor=RoyalBlue}

\usepackage{titlesec}
\titleformat{\section}{\large\bfseries\selectfont}{\thesection\;\;\;}{0em}{}
\titleformat{\subsection}{\normalsize\bfseries\selectfont}{\thesubsection\;\;\;}{0em}{}

\setlist[itemize]{wide=0pt, leftmargin=16pt, labelwidth=10pt, align=left}

\usepackage[subfigure]{tocloft}

\theoremstyle{plain}
\newtheorem{theorem}{Theorem}[section]
\newtheorem{proposition}[theorem]{Proposition}
\newtheorem{lemma}[theorem]{Lemma}
\newtheorem{corollary}[theorem]{Corollary}

\theoremstyle{definition}
\newtheorem{definition}[theorem]{Definition}
\newtheorem{remark}[theorem]{Remark}
\newtheorem{assumption}{Assumption}
\newtheorem*{notations*}{Notations}
\newtheorem*{acknowledge*}{Acknowledgement}

\numberwithin{equation}{section}

\newcommand{\bbc}{\mathbb{C}}
\newcommand{\bbd}{\mathbb{D}}
\newcommand{\bbe}{\mathbb{E}}
\newcommand{\bbj}{\mathbb{J}}
\newcommand{\bbn}{\mathbb{N}}
\newcommand{\bbp}{\mathbb{P}}
\newcommand{\bbr}{\mathbb{R}}
\newcommand{\bbt}{\mathbb{T}}

\newcommand{\fka}{\mathfrak{a}} 
\newcommand{\fkx}{\mathfrak{x}}
\newcommand{\fks}{\mathfrak{s}} 
\newcommand{\fkm}{\mathfrak{m}}
\newcommand{\fkn}{\mathfrak{n}}
\newcommand{\fkt}{\mathfrak{t}}
\newcommand{\fkb}{\mathfrak{b}}
\newcommand{\fkw}{\mathfrak{w}}
\newcommand{\fkg}{\mathfrak{g}}

\newcommand{\bfe}{\mathbf{e}}
\newcommand{\bfu}{\mathbf{u}}
\newcommand{\bfv}{\mathbf{v}}

\newcommand{\bfy}{\mathbf{y}}
\newcommand{\fku}{\boldsymbol{\mathfrak{u}}}
\newcommand{\fkv}{\boldsymbol{\mathfrak{v}}}
\newcommand{\bfxi}{\boldsymbol{\xi}}
\newcommand{\bfId}{\mathbf{I}}
\newcommand{\bfSigma}{\mathbf{\Sigma}}

\newcommand{\paraAVE}{\psi^{\mathrm{ave}}}
\newcommand{\paraISO}{\psi^{\mathrm{iso}}}
\newcommand{\boundAVE}{\phi^{\mathrm{ave}}}
\newcommand{\boundISO}{\phi^{\mathrm{iso}}}
\newcommand{\ParaAVE}{\Psi^{\mathrm{ave}}}
\newcommand{\ParaISO}{\Psi^{\mathrm{iso}}}
\newcommand{\BoundAVE}{\Phi^{\mathrm{ave}}}
\newcommand{\BoundISO}{\Phi^{\mathrm{iso}}}
\newcommand{\rpeAVE}{\mathcal{E}^{\mathrm{ave}}}
\newcommand{\rpeISO}{\mathcal{E}^{\mathrm{iso}}}
\newcommand{\gaussAVE}{\mathcal{Y}^{\mathrm{ave}}}
\newcommand{\gaussISO}{\mathcal{Y}^{\mathrm{iso}}}
\newcommand{\highAVE}{\mathcal{Z}^{\mathrm{ave}}}
\newcommand{\highISO}{\mathcal{Z}^{\mathrm{iso}}}
\newcommand{\caQ}{\mathcal{Q}}

\newcommand{\rie}{\mathrm{\textup{\tiny RIE}}}
\newcommand{\fr}{\mathrm{\textup{\tiny F}}}
\newcommand{\finv}{\mathrm{\textup{\tiny FINV}}}

\newcommand{\coefOne}{\theta}
\newcommand{\coefTwo}{\vartheta}
\newcommand{\coefOnePre}{\varsigma}
\newcommand{\dimu}{\mathbf{\Delta}^{i \mu}}
\newcommand{\parimu}{\partial_{i \mu}}

\newcommand{\opB}{\mathscr{B}}
\newcommand{\opX}{\mathscr{X}}
\newcommand{\opS}{\mathscr{S}}
\newcommand{\opSd}{\opS_{\mathrm{d}}}
\newcommand{\opSo}{\opS_{\mathrm{o}}}

\newcommand{\oprec}{{O}_{\prec}}
\newcommand{\oprecBig}[1]{{O}_{\prec} \bigg ( {#1} \bigg )}

\newcommand{\qwhere}{\quad \text{ where } \quad}
\newcommand{\qand}{\quad \text{ and } \quad}

\newcommand{\abs}[1]{| {#1} |}
\newcommand{\norm}[1]{\| {#1} \|}
\newcommand{\angles}[1]{\langle {#1} \rangle}

\newcommand{\shortpara}[1]{\noindent\underline{\emph{#1}}}

\usepackage{tikz, pgfplots}
\usetikzlibrary{calc}
\pgfplotsset{compat=1.18,
    /pgf/declare function={
	density(\x) = (2/pi/\x)*sqrt((\x^2-1/4)*(9/4-\x^2));
    },
}
\newcommand{\drawPolygon}[2]{
    \draw[line width=1pt,color=SkyBlue,fill=Gray,fill opacity=0.08] (0.5, 0.02*#1) -- (1.5, 0.02*#1) -- ({1.5+(15-#1)/#1/50}, 0.3) -- ({1.5+(15-#1)/#1/50}, {0.6+#2*0.04}) -- ({0.5-(15-#1)/#1/50}, {0.6+#2*0.04}) -- ({0.5-(15-#1)/#1/50}, 0.3) -- cycle;
    \node[below] (Dm#1) at ({0.85+#2*0.15},{0.6+#2*0.04}) {};
    \coordinate (Br#1) at (1.5, 0.02*#1) {};
}

\begin{document}


\title{Eigenvector overlaps in large sample covariance matrices and nonlinear shrinkage estimators}


\author{ 
    \normalsize Zeqin Lin \\ 
    \small \href{mailto:zeqin001@e.ntu.edu.sg}{zeqin001@e.ntu.edu.sg} 
    \and 
    \normalsize Guangming Pan \\ 
    \small \href{mailto:gmpan@ntu.edu.sg}{gmpan@ntu.edu.sg}
}
\date{
    \small School of Physical and Mathematical Sciences \\
    \small Nanyang Technological University
}


\maketitle
\begin{abstract}
Consider a data matrix $Y = [\mathbf{y}_1, \cdots, \mathbf{y}_N]$ of size $M \times N$, where the columns are independent observations from a random vector $\mathbf{y}$ with zero mean and population covariance $\Sigma$. Let $\bfu_i$ and $\bfv_j$ denote the left and right singular vectors of $Y$, respectively. This study investigates the eigenvector/singular vector overlaps $\angles{\bfu_i, D_1 \bfu_j}$, $\angles{\bfv_i, D_2 \bfv_j}$ and $\angles{\bfu_i, D_3 \bfv_j}$, where $D_k$ are general deterministic matrices with bounded operator norms. In the high-dimensional regime, where the dimension $M$ scales proportionally with the sample size $N$, we establish the convergence in probability of these eigenvector overlaps towards their deterministic counterparts with explicit convergence rates. Building on these findings, we offer a more precise characterization of the loss associated with Ledoit and Wolf's nonlinear shrinkage estimators of the population covariance $\Sigma$.
\end{abstract}
\tableofcontents \label{sec:contents}




\section{Introduction} 
\label{sec:intro}

The estimation of covariance matrices or precision matrices is a fundamental problem in the field of multivariate statistics that continues to garner significant attention. This seemingly straightforward problem holds substantial importance across a multitude of domains such as finance, genomics, signal processing, and many others, due to its wide-ranging downstream applications. 

Consider a $M$-dimensional random vector $\bfy \in \bbr^M$ with a zero mean. The $M \times M$ nonnegative-definite \emph{population covariance matrix} of interest is defined as $\Sigma := \bbe \bfy \bfy^\top$. Let $Y = [\bfy_1, \cdots, \bfy_N]$ be the $M \times N$ data matrix where $(\bfy_\mu)_{\mu = 1}^N$ are independent observations from the random vector $\bfy$. It is well understood that when the number of variables $M$ is commensurate with the sample size $N$, 
the \emph{sample covariance matrix}
\begin{equation}
    \hat{\Sigma} = \frac{1}{N} Y Y^\top = \frac{1}{N} \sum_{\mu = 1}^N \bfy_\mu \bfy_\mu^\top \in \bbr^{M \times M},
    \label{def:sample-covariance}
\end{equation}
has limited efficiency when estimating $\Sigma$ and can lead to erroneous conclusions. This curse of dimensionality has spurred a plethora of efforts over the past two decades to develop robust, efficient, and scalable methods for covariance estimation in high-dimensional settings. For a comprehensive review of recent advances in this topic, we recommend  \cite{caiEstimatingStructuredHighdimensional2016,fanOverviewEstimationLarge2016,lamHighdimensionalCovarianceMatrix2020,pourahmadiHighDimensionalCovarianceEstimation2013} and the references therein.

In the absence of any prior knowledge about the structure of the population covariance, the class of shrinkage estimators developed by Ledoit and Wolf  \cite{ledoitNonlinearShrinkageEstimation2012,ledoitPowerNonLinear2022} is one of the most widely recognized and commonly employed method due to their proven capability to reduce estimation error. The key characteristic of shrinkage estimators is their rotational invariance, which essentially means that these estimators are derived by retaining the eigenvectors of the sample covariance matrix while manipulating the eigenvalues. Specifically, suppose that the singular value decomposition (SVD) of $Y$ and the spectral decomposition of $\hat{\Sigma}$ are respectively given by
\begin{equation}
    Y 
    = \sqrt{N} \sum_{i = 1}^{M \wedge N} s_i \bfu_i \bfv_i^\top
    \qand
    \hat{\Sigma} 
    = \sum_{i = 1}^{M} s_i^2 \bfu_i \bfu_i^\top
    =: \sum_{i = 1}^{M} \lambda_i \bfu_i \bfu_i^\top
    =: U \Lambda U^\top,
    \label{eqn:spectral-decomposition}
\end{equation}
where $\{ \bfu_i \}_{i=1}^M$ and $\{ \bfv_\mu \}_{\mu=1}^N$ denote the left and right singular vectors of $Y$, respectively. Here the singular values of $Y/\sqrt{N}$ and the eigenvalues of $\hat{\Sigma}$ are ordered as 
\begin{equation*}
    s_1 \geq s_2 \geq \cdots \geq s_{M \wedge N} \geq 0
    \qand
    \lambda_1 \geq \lambda_2 \geq \cdots \geq \lambda_{M} \geq 0.
\end{equation*}
Then, a \emph{rotationally invariant estimator} (RIE) of $\Sigma$ admits the form
\begin{equation*}
    \hat{\Sigma}^\rie = \sum_{i = 1}^{M} \hat{\lambda}^\rie_i \bfu_i \bfu_i^\top
    =: U \hat{\Lambda}^\rie U^\top,
\end{equation*}
with the shrunk eigenvalues $\hat{\lambda}^\rie_i$ typically dependent on $Y$. The concept of shrinking the eigenvalues of sample covariance matrices dates back to the pioneering works by Stein \cite{stein1975estimation,steinLecturesTheoryEstimation1986}. Within this framework, the task of covariance estimation boils down to the identification of an optimal $\hat{\Lambda}^\rie$, such that the resulting RIE is optimal with respect to (w.r.t.) some specified criterion. 

In the high-dimensional regime ($M \asymp N$), a remarkable breakthrough in this line of research was the introduction of the linear shrinkage estimator by Ledoit and Wolf \cite{ledoitWellconditionedEstimatorLargedimensional2004}. In this context, $\hat{\Lambda}^\rie = \operatorname{diag} (\hat{\lambda}^\rie_1, \cdots, \hat{\lambda}^\rie_M)$ is derived through a (data-dependent) linear transformation of the sample eigenvalues $\Lambda$. Subsequently, Ledoit and Wolf continued to excavate the potential of shrinkage methods by investigating the more flexible approach of nonlinear shrinkage of sample eigenvalues \cite{ledoitAnalyticalNonlinearShrinkage2020,ledoitNonlinearShrinkageEstimation2012,ledoitOptimalEstimationLargedimensional2018,ledoitQuadraticShrinkageLarge2022,ledoitShrinkageEstimationLarge2021}. The introduction of these nonlinear shrinkage estimators coincided with significant advancements in random matrix theory (RMT), resulting in an additional level of performance improvement. One can refer to the systematic review provided in \cite{ledoitPowerNonLinear2022} for more details about this trajectory and extensive pointers to the literature.

To motivate our research about eigenvector overlaps in large sample covariance matrices, let us quickly revisit Ledoit and Wolf's approach of constructing nonlinear shrinkage estimators. The first step involves selecting an appropriate loss function $\mathscr{L}$ to quantify the estimation error. One can derive the minimizer of the optimization problem
\begin{equation}
    \min_{\hat{\lambda}^\rie_1, \cdots, \hat{\lambda}^\rie_M} \mathscr{L} (\hat{\Sigma}^\rie, \Sigma)
    \quad \text{ s.t. } \quad
    \hat{\Sigma}^\rie = U \operatorname{diag} (\hat{\lambda}^\rie_1, \cdots, \hat{\lambda}^\rie_M) U^\top.
    \label{eqn:min-estimation-error}
\end{equation}
That is, we operate within the class of RIEs and target at the optimal estimator that minimizes estimation error. For illustration, we consider the Frobenius loss $\mathscr{L}^\fr (\hat{\Sigma}^\rie, \Sigma) := M^{-1} \norm{\hat{\Sigma}^\rie - \Sigma}_{\mathrm{F}}^2$, where $\norm{\cdot}_{\mathrm{F}}$ denotes the Frobenius norm. In this situation, by elementary matrix algebra, one can easily show that the minimizer of (\ref{eqn:min-estimation-error}) is given by the \emph{oracle} shrinkages
\begin{equation}
    \lambda^\fr_i := \angles{\bfu_i, \Sigma \bfu_i}
    \equiv \bfu_i^\top \Sigma \bfu_i,
    \label{def:oracle-overlap}
\end{equation}
i.e., the overlaps of sample eigenvectors w.r.t the population covariance. Consequently, the finite-sample optimal ``estimator'' is given by
\begin{equation*}
    {\Sigma}^\fr 
    = U \Lambda^\fr U^\top
    = U \operatorname{diag} ({\lambda}^\fr_1, \cdots, {\lambda}^\fr_M) U^\top
\end{equation*}
with $\lambda^\fr_i$ defined in (\ref{def:oracle-overlap}). However, the ``estimator'' ${\Sigma}^\fr$ is not practically feasible since the shrunk eigenvalues $\lambda^\fr_i$ involve the unknown population covariance $\Sigma$. 
This explains why these values are termed ``oracle''. Nevertheless, it signifies considerable progress as the problem has now been transformed from estimating $M(M+1)/2$ parameters in the full population covariance matrix to the more amenable task of estimating $M$ eigenvector overlaps. This paves the way for the application of high-dimensional asymptotic results from RMT.

In the high-dimensional regime where $M$ is proportional to $N$, Mar\v{c}enko and Pastur's groundbreaking work \cite{marcenkoDistributionEigenvaluesSets1967} established the weak convergence of $M^{-1} \sum_{i = 1}^M \delta_{\lambda_i}$, namely the \emph{empirical spectral distribution} (ESD) of the sample covariance $\hat{\Sigma}$, to a certain \emph{limiting spectral distribution} (LSD). Ledoit and P\'{e}ch\'{e} \cite{ledoitEigenvectorsLargeSample2011} extended this result by demonstrating the weak convergence of $M^{-1} \sum_{i = 1}^M \angles{\bfu_i, \Sigma \bfu_i} \delta_{\lambda_i}$, i.e., the empirical spectral measure weighted by eigenvector overlaps. Specifically, it has been established that the Radon-Nikodym derivative of the weak limit of this weighted measure w.r.t. the Mar\v{c}enko-Pastur law is given by $1/(\lambda \abs{\fkm (\lambda)}^2)$ for $\lambda > 0$. Here $\fkm$ represents the Stieltjes transform of the Mar\v{c}enko-Pastur law. Refer to (\ref{eqn:self-consistent-z}) below for its formal definition. This development gave rise to the introduction of the transitional shrinkages
\begin{equation}
    \lambda^{\fr \star}_i := \frac{1}{\lambda_i \abs{\fkm (\lambda_i)}^2}.
    \label{def:oracle-estimate-renewed}
\end{equation}

The nonlinear shrinkage estimator based on the transitional estimates in (\ref{def:oracle-estimate-renewed}) is still infeasible in practice due to its reliance on $\fkm$, the Stieltjes transform of the LSD. However, a practical estimator can be derived from these transitional quantities upon employing a consistent estimator of $\fkm$,
\begin{equation}
    \hat{\lambda}^{\fr \star}_i := \frac{1}{\lambda_i \abs{\hat{\fkm} (\lambda_i)}^2}.
    \label{def:practical-estimate}
\end{equation}
In the early implementations by Ledoit and Wolf \cite{ledoitNonlinearShrinkageEstimation2012,ledoitOptimalEstimationLargedimensional2018,ledoitShrinkageEstimationLarge2021}, the estimation of $\fkm (\lambda_i)$ was accomplished through an indirect approach that involves a large dimensional optimization problem \cite{ledoitSpectrumEstimationUnified2015} to recover the spectrum of $\Sigma$. However, due to the lack of an analytical solution for this optimization problem, implementing these numerical strategies is cumbersome and typically demands a considerable amount of computational time. In their follow-up work \cite{ledoitAnalyticalNonlinearShrinkage2020}, Ledoit and Wolf devised a computationally more tractable implementation by integrating the kernel estimator of LSD developed in \cite{jingNonparametricEstimateSpectral2010}. More recently, Benaych-Georges \cite{benaych-georgesShortProofLedoitPeche2023} suggested an even more direct approach, approximating $\fkm (\lambda_i)$ by $\fkg (\lambda_i + \mathrm{i} \eta)$ with some $\eta \ll 1$. Here, $\fkg$ represents the Stieltjes transform of the ESD, which can be readily computed from the sample eigenvalues $\lambda_i$. This strategy has also been adopted in the algorithm developed in \cite{benaych-georgesOptimalCleaningSingular2023} for estimating cross-covariance matrices.

Both approaches in \cite{ledoitAnalyticalNonlinearShrinkage2020} and \cite{benaych-georgesShortProofLedoitPeche2023} result in effective and robust implementations of nonlinear shrinkage estimators. In \cite{ledoitAnalyticalNonlinearShrinkage2020},  by drawing comparisons with the oracle estimator $\Sigma^\fr$, the practical performance of the algorithmic estimator
\begin{equation*}
    \hat{\Sigma}^{\fr \star} := U \hat{\Lambda}^{\fr \star} U^\top
    = U \operatorname{diag} (\hat{\lambda}^{\fr \star}_1, \cdots, \hat{\lambda}^{\fr \star}_M) U^\top.
\end{equation*}
has been extensively validated through Monte Carlo simulations. Regarding the theoretical efficacy of $\hat{\Sigma}^{\fr \star}$, Ledoit and Wolf \cite{ledoitAnalyticalNonlinearShrinkage2020,ledoitOptimalEstimationLargedimensional2018} demonstrated that, under some technical assumptions,
\begin{equation}
    \mathscr{L}^{\fr} ({\Sigma}^{\fr}, \Sigma)
    - \mathscr{L}^{\fr} (\hat{\Sigma}^{\fr \star}, \Sigma) 
    \stackrel{\bbp}{\longrightarrow} 0,
    \quad \text{ as } \quad
    N \to \infty.
    \label{eqn:convergence-loss}
\end{equation}
Essentially, this means that asymptotically, the algorithmic implementation $\hat{\Sigma}^{\fr \star}$ is equally effective as its oracle counterpart ${\Sigma}^{\fr}$ when minimizing the loss.

However, a quantitative evaluation of the loss incurred when estimating (\ref{def:oracle-overlap}) with (\ref{def:practical-estimate}),
\begin{equation}
    \mathscr{L}^{\fr} ({\Sigma}^{\fr}, \hat{\Sigma}^{\fr \star})
    = \frac{1}{M} \sum_{i=1}^M \abs{{\lambda}^{\fr}_i - \hat{\lambda}^{\fr \star}_i}^2
    = \frac{1}{M} \sum_{i=1}^M \bigg | { \angles{\bfu_i, \Sigma \bfu_i} - \frac{1}{\lambda_i \abs{\hat{\fkm} (\lambda_i)}^2}} \bigg |^2,
    \label{def:loss-oracle-pratical}
\end{equation}
is still absent in the literature. From the perspective of practitioners, a quantitative evaluation of the loss (\ref{def:loss-oracle-pratical}) offers compelling justifications to employ $\hat{\Sigma}^{\fr \star}$ as a feasible substitute for the oracle estimator ${\Sigma}^{\fr}$. Moreover, an explicit convergence rate in terms of $N$ for (\ref{def:loss-oracle-pratical}) could provide theoretical supports for the rapid convergence, in terms of loss, of $\hat{\Sigma}^{\fr \star}$ towards ${\Sigma}^{\fr}$ as depicted in \cite[Figure 4]{ledoitAnalyticalNonlinearShrinkage2020}.

In view of the above, this article focuses on the asymptotic behavior of the eigenvector overlaps ${\lambda}^{\fr}_i = \angles{\bfu_i, \Sigma \bfu_i}$ within the context of large sample covariance matrices. We highlight that the results we establish here are considerably more general than the convergence of the loss (\ref{def:loss-oracle-pratical}). Actually, we are able to quantify, in the high-dimensional regime, the convergence of each eigenvector overlap towards its corresponding deterministic counterpart,
\begin{equation}
    {\lambda}^{\fr}_i - \hat{\lambda}^{\fr \star}_i
    = \angles{\bfu_i, \Sigma \bfu_i} - \frac{1}{\lambda_i \abs{\hat{\fkm} (\lambda_i)}^2}.
\end{equation}
This elementwise convergence even yields a satisfactory control, in terms of the operator norm, over the difference between ${\Sigma}^{\fr}$ and $\hat{\Sigma}^{\fr \star}$,
\begin{equation}
    \norm{{\Sigma}^{\fr} - \hat{\Sigma}^{\fr \star}}
    = \max_{1 \leq i \leq M} \abs{{\lambda}^{\fr}_i - \hat{\lambda}^{\fr \star}_i}.
    \label{def:diff-operator-norm}
\end{equation}
More specifically, our main result covers the overlaps of the left/right singular vectors of the data matrix $Y$ w.r.t. arbitrary deterministic matrices with bounded operator norms,
\begin{equation}
    \angles{\bfu_i, D_1 \bfu_j},
    \quad
    \angles{\bfv_i, D_2 \bfv_j}
    \qand
    \angles{\bfu_i, D_3 \bfv_j},
    \label{def:singular-vec-overlaps}
\end{equation}
where the matrices $D_1, D_2, D_3$ are of compatible sizes. This level of flexibility in choosing the deterministic matrix $D_1$ enables us to manage various nonlinear shrinkage estimators induced by different loss functions \cite{ledoitOptimalEstimationLargedimensional2018,ledoitShrinkageEstimationLarge2021}, extending our analysis beyond the Frobenius loss $\mathscr{L}^\fr$.

The methodology of this article, from a theoretical perspective, is inspired by the work of \cite{cipolloniOptimalLowerBound2023}, which investigated eigenvector overlaps in the context of additively deformed non-Hermitian i.i.d. matrices. The research \cite{cipolloniOptimalLowerBound2023} is part of a broader effort led by Cipolloni, Erd\H{o}s, Schr\"{o}der, and their collaborators, focusing on the overlaps of singular vectors/eigenvectors in large random matrix ensembles. A pivotal aspect of this line of research is the development of the so-called \emph{multi-resolvent local laws}, which serves as a potent theoretical tool for inferring asymptotic properties of eigenvector overlaps. This ambitious project initially focused on resolving the eigenstate thermalization hypothesis (ETH) for the most fundamental prototype in RMT, the Wigner matrices \cite{cipolloniEigenstateThermalisationEdge2023,cipolloniEigenstateThermalizationHypothesis2021,cipolloniOptimalMultiresolventLocal2022,cipolloniRankuniformLocalLaw2022,cipolloniThermalisationWignerMatrices2022}. Subsequently, the scope is expanded to encompass random matrix ensembles involving full-rank deformation, including deformed non-Hermitian i.i.d. matrices \cite{cipolloniOptimalLowerBound2023}, deformed Wigner matrices \cite{cipolloniGaussianFluctuationsEquipartition2023}, and general Wigner-type matrices \cite{erdosEigenstateThermalizationHypothesis2024}.

The sample covariance matrices examined in this article are subjected to a multiplicative deformation induced by the population covariance $\Sigma$, which aligns our work with the studies \cite{cipolloniGaussianFluctuationsEquipartition2023,cipolloniOptimalLowerBound2023,erdosEigenstateThermalizationHypothesis2024}. This intrinsic anisotropy within the underlying random matrix ensembles forces the deterministic counterpart of $\angles{\bfu_i, D \bfu_i}$ to be energy-dependent, meaning it relies on both the observable $D$ and the location of $\lambda_i$. Consequently, these deformed ensembles pose a significantly more complex technical challenge compared to Wigner matrices, for which $\angles{\bfu_i, D \bfu_i}$ asymptotically concentrates around the normalized trace of $D$ regardless of the index $i$.

The situation is further complicated by the sparsity of eigenvalues near the edges of LSD. In fact, the studies \cite{cipolloniGaussianFluctuationsEquipartition2023,cipolloniOptimalLowerBound2023,erdosEigenstateThermalizationHypothesis2024} have predominantly concentrated on developing asymptotic results for the overlap of bulk eigenvectors, i.e., those corresponding to eigenvalues located in the bulk of the asymptotic spectrum. Essentially, this excludes eigenvectors associated with edge eigenvalues (e.g., $\bfu_i$ with $i \ll N$). In this article, we venture into this challenging edge regime by conducting a meticulous analysis on resolvent chains involving spectral parameters close to the edges. Importantly, this rigorous effort leads to a nontrivial convergence rate for the overlap of edge eigenvectors, and significantly enhances our analysis of the nonlinear shrinkage estimators.

Finally, let us mention that apart from \cite{benaych-georgesShortProofLedoitPeche2023,ledoitEigenvectorsLargeSample2011}, there are two additional research directions in the field of RMT concerning the asymptotic behavior of eigenvectors of large sample covariance matrices. Both directions focus on the generalized eigenvector components, i.e. quantities of the form $\angles{\mathbf{a}, \bfu_i}$, where $\mathbf{a} \in \bbr^M$ is deterministic. One direction, represented by \cite{baiAsymptoticsEigenvectorsLarge2007,silversteinEigenvectorsLargeDimensional1989,silversteinLimitTheoremsEigenvectors1984,silversteinWeakConvergenceRandom1990,xiConvergenceEigenvectorEmpirical2020,yangLinearSpectralStatistics2020}, investigates the global behavior of eigenvectors by examining the LSD weighted by $\abs{\angles{\mathbf{a}, \bfu_i}}^2$. In the presence of an additional finite-rank deformation, another direction, represented by \cite{baoStatisticalInferencePrincipal2022,bloemendalPrincipalComponentsSample2016,paulAsymptoticsSampleEigenstructure2007}, analyzes the almost sure limit and fluctuation of the spiked eigenvectors. Overall speaking, compared with the extensive literature on eigenvalues, there are relatively few works dedicated to eigenvectors. Furthermore, we highlight that for a general non-null $\Sigma$, the analysis of individual eigenvectors of sample covariance matrices is notably scarce in the literature, and this article makes a progress in alleviating this situation.

During the preparation of the first version of this article, a contemporaneous work \cite{dingEigenvectorDistributionsOptimal2024} by Ding, Li and Yang established, in the context of the sample covariance matrices, the joint asymptotic Gaussian fluctuation of the generalized eigenvector components
\begin{equation*}
    \sqrt{M} [\angles{\mathbf{a}, \bfu_{i_1}}, \cdots, \angles{\mathbf{a}, \bfu_{i_n}}]
\end{equation*}
for arbitrary deterministic $\mathbf{a} \in \bbr^M$ with bounded norms and an finite set of indices $\{ i_1, \cdots, i_n \}$ associated with non-spiked and non-zero eigenvalues. This insightful theoretical result was achieved by analyzing the eigenvector moment flow induced by the rectangular Dyson Brownian motion (DBM). Particularly, as a corollary, they concluded the $L^1$ convergence of the eigenvector overlaps $\angles{\bfu_i, D_1 \bfu_i}$. Compared with the methodology adopted by \cite{dingEigenvectorDistributionsOptimal2024}, the analysis of eigenvector overlaps in the current article is built on the multi-resolvent local laws. This approach allows us to explicitly characterize the convergence rate of these overlaps. In addition, as a byproduct, we also cover the overlaps $\angles{\bfv_i, D_2 \bfv_j}$ and $\angles{\bfu_i, D_3 \bfv_j}$, which involve the right singular vectors of the data matrix $Y$.

To summarize, this article makes two key contributions. Firstly, we give more precise characterizations of the loss functions for estimating the population covariance matrix including the one from the algorithmic approximation of the oracle estimates for Ledoit-Wolf's nonlinear shrinkage estimators in the high-dimensional regime. Secondly, we establish the convergence of overlaps for all individual left and right singular vectors of the data matrix, except for those in the null space, towards their respective deterministic counterparts with convergence rates. 

The structure of this article is organized as follows. Section \ref{sec:singular-vec-overlaps} presents the main result regarding the convergence of eigenvector overlaps. Section \ref{sec:nonlinear-shrinkage} characterizes the convergence rate of the loss functions of nonlinear shrinkage estimators. Section \ref{sec:multi-local-law} is dedicated to our primary technical result, namely the multi-resolvent local laws for sample covariance matrices. Here, a crucial concept, the regularity of observables in resolvent chains, is introduced. In Section \ref{subsec:proof-main-result} we demonstrate how this powerful technical result is utilized to establish the main results discussed in Sections \ref{sec:singular-vec-overlaps} and \ref{sec:nonlinear-shrinkage}. The subsequent sections delve into the proof of the multi-resolvent local laws, with a detailed organization provided at the beginning of Section \ref{sec:self-improving-inequalities}.

\begin{notations*}
Throughout this article, we regard $N$ as the fundamental parameter and take $M \equiv M^{(N)}$. To streamline notation, we frequently suppress the dependence on $N$ from the notations, bearing in mind that all quantities that are not explicitly constant may depend on the asymptotic parameter $N$. Given the $N$-dependent nonnegative quantities $a \equiv a^{(N)}$ and $b \equiv b^{(N)}$, we write $a \lesssim b$ or $a = O (b)$ if $a^{(N)} \leq C b^{(N)}$ for some $N$-independent constant $C > 0$. Here $C$ may implicitly depend on other quantities that are explicitly $N$-independent (e.g., $\tau, C_p$ used in our assumptions below). We also write $a \asymp b$ if both $a \lesssim b$ and $a \gtrsim b$ hold.

We adopt the following notion of high probability bounds from \cite{erdosLocalSemicircleLaw2013} to systematize statements of the form ``$\mathcal{Y}$ is bounded with high probability by $\mathcal{Z}$ up to small powers of $N$''. For two families of nonnegative random variables parameterized by $N \in \mathbb{N}$ and $t \in \bbt^{(N)}$,
\begin{equation*}
    \mathcal{Y} \equiv \mathcal{Y}^{(N)}(t)
    \qand
    \mathcal{Z} \equiv \mathcal{Z}^{(N)}(t),
\end{equation*}
we say $\mathcal{Y}$ is \emph{stochastically dominated} by $\mathcal{Z}$, uniformly in $t$, if for any (small) $\varepsilon > 0$ and (large) $L > 0$,
\begin{equation}
    \sup_{t \in \bbt^{(N)}} 
    \bbp \big \{ {\mathcal{Y}^{(N)}(t) > N^{\varepsilon} \mathcal{Z}^{(N)}(t)} \big \} \leq N^{-L},
    \quad \text{ for all } \quad N \geq N_0(\varepsilon, L).
    \label{def:stochastic-domination}
\end{equation}
In this case, we write $\mathcal{Y} \prec \mathcal{Z}$. If $\mathcal{Y}$ is complex with $\abs{\mathcal{Y}} \prec \mathcal{Z}$, we also write $\mathcal{Y} = \oprec (\mathcal{Z})$. We refer to \cite[Lemma 3.4]{benaych-georgesLecturesLocalSemicircle2018} for the basic properties of $\prec$.

Given a complex number $z \in \bbc$, we use $\Re z$ and $\Im z$ to denote its real and imaginary parts, respectively. The transpose, complex conjugate, Hermitian transpose of a matrix $A$ are denoted as $A^\top$, $\bar{A}$ and $A^*$, respectively. Let $A \equiv A^{(N)}$ be a sequence of square matrices indexed by $N \in \mathbb{N}$. We denote the \emph{normalized trace} of $A$ as 
\begin{equation}
    \angles{A} := N^{-1} \operatorname{tr} A^{(N)}.
\end{equation}
Let us emphasize that the normalization constant is determined by the index $N$ rather than the dimension of $A$. The standard inner product of two vectors of the same dimension is denoted as $\angles{\fku, \fkv} := \fku^* \fkv$. We consistently use $\norm{\cdot}$ to denote the Euclidean norm of a vector or the operator norm of a matrix w.r.t. the Euclidean norm. Finally, given positive integers $N_1 < N_2$, we denote $\llbracket N_1, N_2 \rrbracket := [N_1, N_2] \cap \bbn$ and abbreviate $\llbracket N_2 \rrbracket \equiv \llbracket 1, N_2 \rrbracket$. 

\end{notations*}

\begin{acknowledge*}
The authors would like to express their sincere gratitude to Oliver Ledoit for bringing to our attention the importance of quantifying the convergence of the loss (\ref{def:loss-oracle-pratical}). We are also thankful to Xiucai Ding, Yun Li and Fan Yang for sharing the details of their recent work \cite{dingEigenvectorDistributionsOptimal2024}.
\end{acknowledge*}

\section{Singular vector/eigenvector overlaps} 
\label{sec:singular-vec-overlaps}

This section is to present our main results concerning the singular vector overlaps for the data matrix $Y$. First, the three singular vector overlaps in (\ref{def:singular-vec-overlaps}) can be unified as follows. Let
\begin{equation}
    \bfxi_{\pm i} := \frac{1}{\sqrt{2}} \begin{bmatrix} \bfu_i \\ \pm \bfv_i \end{bmatrix},
    \quad
    i \in \llbracket M \wedge N \rrbracket
    \qand
    \bfxi_{i} := \begin{bmatrix} \mathbb{1} ({M > N}) \bfu_i \\ 
    \mathbb{1} ({M < N}) \bfv_i \end{bmatrix},
    \quad
    i \in \llbracket M \wedge N + 1, M \vee N \rrbracket.
\end{equation}
We can define $\bbj := \llbracket -M \wedge N, M \vee N \rrbracket \backslash \{ 0 \}$. Then, $\{ \bfxi_i \}_{i \in \bbj}$ constitute an orthonormal basis of $\bbr^{M+N}$ and correspond to the eigenvectors of the self-adjoint dilation of $Y/\sqrt{N}$,
\begin{equation}
    H := \frac{1}{\sqrt{N}} \begin{bmatrix} 0 & Y \\ Y^\top & 0 \end{bmatrix}
    = \sum_{i = 1}^{M \wedge N} s_i \begin{bmatrix} 0 & \bfu_i \bfv_i^\top \\ \bfv_i \bfu_i^\top & 0 \end{bmatrix} 
    = \sum_{i \in \bbj} s_i \bfxi_i \bfxi_i^\top.
    \label{def:linearization}
\end{equation}
Here we also make the convention that 
\begin{equation*}
    s_{-i} = -s_i, \quad i \in \llbracket 1, M \wedge N \rrbracket
    \qand
    s_{i} = 0, \quad i \in \llbracket M \wedge N + 1, M \vee N \rrbracket.
\end{equation*}
Now with the $\bfxi_i$'s, it suffices to target the overlaps $\angles{\bfxi_i, D \bfxi_j}$ for general deterministic $(M+N) \times (M+N)$ matrices $D$. Informally speaking, the main theoretical result of this article states that, in the high-dimensional regime, under certain regularity condition on the spectrum of $\Sigma$, 
\begin{equation*}
    \angles{\bfxi_i, D \bfxi_j} - \delta_{ij} \theta_{i} (D) \to 0,
    \quad \text{ as } \quad N \to \infty,
\end{equation*}
where the deterministic coefficients $\theta_{i} (D)$ can be computed from the LSD for the $s_i$'s (or the $\lambda_i$'s). The formal statement is presented in Theorem \ref{thm:eigenvector-overlaps}. Before that, we need to rigorously formulate the technical assumptions and recall some fundamental results from RMT that are necessary for a coherent presentation.

For the assumptions presented subsequently, we adopt the convention that $\tau > 0$ represents a constant (independent of $N$) that can be chosen arbitrarily small.

\begin{assumption}[High-dimensional regime] \label{assump:high-dimension}
Suppose $\tau \leq M / N \leq \tau^{-1}$ uniformly for all $N$.
\end{assumption}

\begin{assumption}[Data generating process] \label{assump:moments}
We assume that $Y$, the matrix of observations, is generated according to $Y = \sqrt{N} \Sigma^{1/2} X$, where the entries of $X = [x_{i \mu}]$ are independent real random variables with $\mathbb{E} x_{i \mu} = 0$ and $\mathbb{E} \abs{x_{i \mu}}^2 = 1/N$. Furthermore, we assume that there exist constants $C_p > 0$, for any $p \in \bbn$, such that $\mathbb{E} \abs{x_{i \mu}}^p \leq C_p N^{-p/2}$.
\end{assumption}

With Assumption \ref{assump:moments}, the sample covariance matrix can be expressed as $\hat{\Sigma} = \Sigma^{1/2} X X^\top \Sigma^{1/2}$. The remaining assumptions pertain to the LSD of $\hat{\Sigma}$. Assume that the eigenvalues of $\Sigma$ are given by
\begin{equation*}
    \sigma_1 \geq \sigma_2 \geq \cdots \geq \sigma_M \geq 0.
\end{equation*}
We make the following two basic assumptions to avoid problem degeneration.

\begin{assumption}[Boundedness of $\Sigma$] \label{assump:boundedness}
$\| \Sigma \| = \sigma_1 \leq \tau^{-1}$.
\end{assumption}

\begin{assumption}[Anti-concentration at $0$] \label{assump:anti-concentration}
$M^{-1} \sum_{i=1}^M \mathbb{1} (\sigma_i \leq \tau) \leq 1 - \tau$.
\end{assumption}

Our next two assumptions concern the regularity of LSD of $\hat{\Sigma}$, and therefore require a more technical discussion. Before detailing  these assumptions, let us gather several elementary results concerning the LSD of $\hat{\Sigma}$ and its Stieltjes transform. Most of the subsequent statements, as well as their proofs, can be founded in \cite{baoUniversalityLargestEigenvalue2015,knowlesAnisotropicLocalLaws2017,silversteinAnalysisLimitingSpectral1995}. 

For each $z \in \bbc \backslash \bbr$, there exists a unique $\fkm \equiv \fkm^{(N)} (z) \in \bbc$ satisfying the self-consistent equation
\begin{equation}
    z = f(\fkm)
    := - \frac{1}{\fkm} + \frac{1}{N} \operatorname{tr} \frac{\Sigma}{1 + \fkm \Sigma},
    \quad \text{ with } \quad
    \Im z \cdot \Im \fkm > 0.
    \label{eqn:self-consistent-z}
\end{equation}
Moreover, $\fkm (z)$ is the Stieltjes transform of a probability measure $\varrho \equiv \varrho^{(N)}$ supported in $[0,\infty)$,
\begin{equation*}
    \fkm (z) = \int_{\bbr} \frac{1}{\lambda - z} \varrho (\mathrm{d} \lambda).
\end{equation*}
By slight abuse of terminology, we refer to the measure $\varrho$ as the limiting spectral distribution (LSD) of $\hat{\Sigma}$. Rigorously speaking, $\varrho$ represents LSD of the Gram matrix $Y^\top Y/N = X^\top \Sigma X$ since almost surely, the L\'{e}vy-Prokhorov metric between $\varrho$ and $N^{-1} \sum_{\mu=1}^N \delta_{\lambda_\mu}$, the ESD of $Y^\top Y/N$, converges to zero as $N \to \infty$. The measure $\varrho$ possesses a continuous density on $(0, \infty)$. In fact, by \cite[Theorem 1.1]{silversteinAnalysisLimitingSpectral1995}, one can extend the definition of $\fkm(z)$ down to the real axis by setting
\begin{equation*}
    \fkm (E) := \lim_{\eta \downarrow 0} \fkm (E + \mathrm{i} \eta),
    \quad E \not= 0.
\end{equation*}
Then, the density of $\varrho$ is given by $(1/\pi) \Im \fkm (E)$. When there is no ambiguity, we also use $\varrho(E)$ to denote the density of $\varrho$. The following lemma characterizes the support of $\varrho$ and can be found in \cite[Lemma 2.6]{knowlesAnisotropicLocalLaws2017}. Note that $\bbr_+ = (0, \infty)$.

\begin{lemma}
Under Assumptions \ref{assump:high-dimension}-\ref{assump:anti-concentration}, there exist $\fka_1 \geq \fka_2 \geq \cdots \geq \fka_{2 K} \geq 0$ such that
\begin{equation*}
    \operatorname{supp} \varrho \cap \bbr_+
    = \left ( \bigcup_{k=1}^K [\fka_{2k}, \fka_{2k-1}] \right ) \cap \bbr_+.
\end{equation*}
Moreover, we have $\fka_1 \lesssim 1$, i.e., the rightmost edge of $\operatorname{supp} \varrho$ is uniformly bounded in $N$.
\end{lemma}

Here we use $K \in \bbn$ to denote the number of bulk components of $\operatorname{supp} \varrho$. Note that both $K$ and the edges $\fka_k$ may depend on $N$ via $\Sigma \equiv \Sigma^{(N)}$. We are now ready to state our assumptions on the regularity of the spectrum of $\Sigma$, which essentially preclude pathological behaviors of the LSD $\varrho$. These regularity conditions have been previously discussed in various works on high-dimensional sample covariance matrices. The version we present here takes the same form as in \cite{knowlesAnisotropicLocalLaws2017}. For a more detailed explanation of the rationale behind these assumptions, we refer readers to \cite{knowlesAnisotropicLocalLaws2017} and the references therein. 

We acknowledge a potential critique of Assumption \ref{assump:edge-regularity} regarding its exclusion of spiked eigenvalues in the spectrum of $\Sigma$. However, we anticipate that, through a straightforward perturbation argument, most estimates on eigenvector overlaps and the loss of shrinkage estimators presented in this article remain valid for non-spiked sample eigenvectors within a spiked population model. We have refrained from incorporating spiked eigenvalues due to the methodological differences in deriving the shrinkage estimator for such cases, which warrant an independent treatment. Interested readers can refer to the works by Donoho, Gavish and Johnstone \cite{donohoOptimalShrinkageEigenvalues2018} or Ding, Li and Yang \cite{dingEigenvectorDistributionsOptimal2024} for a detailed exploration of shrinkage methods in the context of spiked population models.

\begin{assumption}[Edge regularity] \label{assump:edge-regularity}
Assume that the following holds for all $k \in \llbracket 2 K \rrbracket$,
\begin{equation*}
    \fka_{k} \geq \tau,
    \quad \min_{\ell: \ell \not= k} |\fka_{k}-\fka_{\ell}| \geq \tau,
    \qand \min_{i} |1 + \fkm(\fka_{k}) \sigma_{i}| \geq \tau.
\end{equation*}
\end{assumption}

\begin{assumption}[Bulk regularity] \label{assump:bulk-regularity}
Given any $\tau^\prime > 0$, there exists a constant $c \equiv c(\tau, \tau^\prime) > 0$ such that
\begin{equation*}
    \varrho(E) \geq c,
    \quad \text{ whenever } \quad 
    E \in \bigcup_{k=1}^K [\fka_{2k}+\tau^\prime, \fka_{2k-1}-\tau^\prime].
\end{equation*}
\end{assumption}

We note that Assumption \ref{assump:edge-regularity} constrains $K$ to be bounded by some constant depending only on $\tau$. Here the two assumptions are formulated in terms of $\varrho$. We mention that one can readily rephrase them using the LSD of the singular values of $Y/\sqrt{N}$, or equivalently, the eigenvalues of $H$. Actually, we find it convenient to introduce 
\begin{equation*}
    m (w) := w \fkm (w^2) = \int_{\bbr} \frac{1}{s - w} \rho (\mathrm{d} s),
\end{equation*}
where the probability measure $\rho \equiv \rho^{(N)}$ can be related to $\varrho$ via
\begin{equation*}
    \rho (\{ 0 \}) = \varrho (\{ 0 \})
    \qand
    \rho ([a, b]) = \rho ([-b, -a]) = \frac{1}{2} \varrho ([a^2, b^2]),
    \quad \text{ for all } \quad 0 < a < b.
\end{equation*}
In this article, we interchangeably use the LSDs $\varrho$ and $\rho$ (and their Stieltjes transform $\fkm$ and $m$). The choice of which one to use depends on the context. 

As mentioned earlier, the behavior of eigenvector overlaps in our context is influenced by the positions of their corresponding eigenvalues. Hence, we introduce the concept of \emph{classical locations} $\gamma_1 \geq \gamma_2 \geq \cdots \geq \gamma_{M \wedge N}$, which serve as the deterministic counterpart of the singular values $s_i$. These classical locations can be defined as quantiles of $\rho$ as follows,
\begin{equation}
    \frac{i - 1/2}{N}
    = \varrho ([\gamma_i^2, \infty)) 
    = 2 \rho ([\gamma_i, \infty)),
    \quad \text{ for all } \quad 
    i \in \llbracket M \wedge N \rrbracket.
    \label{def:classical-locations}
\end{equation}
A profound result in RMT is the so-call singular value/eigenvalue rigidity. Essentially, with high probability, the singular values $s_i$ (reps. the eigenvalues $\lambda_i = s_i^2$) are closed to their classical locations $\gamma_i$ (resp. $\gamma_i^2$) with a precision depending on $N$. The details are presented in Proposition \ref{prop:rigidity}. To streamline the presentation, we also define the classical number of eigenvalues in the $k$-th bulk component through
\begin{equation}
    N_k := N \int_{\fka_{2k}}^{\fka_{2k-1}} \varrho ({\mathrm{d} \lambda}),
    \quad \text{ for all } \quad 
    k \in \llbracket K \rrbracket.
    \label{def:classical-num-eigen}
\end{equation}
We note that it has been proved in \cite[Lemma A.1]{knowlesAnisotropicLocalLaws2017} that $N_k \in \bbn$ for all $k$. Now we can relabel the spectral components as follows,
\begin{equation}
    s_{k,i} := s_{N_1 + \cdots + N_{k-1} + i},
    \quad
    \gamma_{k,i} := \gamma_{N_1 + \cdots + N_{k-1} + i}
    \qand
    \bfxi_{k,i} := \bfxi_{N_1 + \cdots + N_{k-1} + i}.
    \label{def:relabelling}
\end{equation}
The eigenvalues $\lambda_{k,i}$ and the singular vectors $\bfu_{k,i}, \bfv_{k,i}$ can be defined in exactly the same way. Finally, let us introduce the main theoretical tool of this article, the resolvent (Green function),
\begin{equation}
    G(w) := (H - w)^{-1}
    = \sum_{i \in \bbj} \frac{1}{s_i - w} \bfxi_i \bfxi_i^\top.
    \label{def:resolvent-G}
\end{equation}
The single-resolvent local law (see Proposition \ref{prop:single-resolvent-local} below) asserts that whenever $\abs{\Im w} \gg N^{-1}$, the resolvent $G(w)$ becomes asymptotically deterministic in the weak operator sense as $N \to \infty$. The deterministic surrogate $\Pi (w)$ is given by
\begin{equation}
    \Pi (w) := \begin{bmatrix} \Gamma(w) & 0 \\ 0 & m(w) \end{bmatrix} 
    = \begin{bmatrix} - (w + m(w) \Sigma)^{-1} & 0 \\ 0 & m(w) \end{bmatrix}
    \in \bbr^{(M+N) \times (M+N)}.
    \label{def:Pi}
\end{equation} 
Here we use $\Gamma(w)$ to denote the top-left $M \times M$ block of $\Pi(w)$. Similar to $\fkm(z)$, the function $m(w)$, and thus $\Gamma (w)$, also have a continuous extension to the real line from the upper half plane,
\begin{equation*}
    m (E) := \lim_{\eta \downarrow 0} m (E + \mathrm{i} \eta) = E \fkm (E^2)
    \qand
    \Gamma (E) := - (E + m(E) \Sigma)^{-1},
    \quad E \not= 0.
\end{equation*}

We are now prepared to rigorously present the main theoretical result of this article.

\begin{theorem}[Eigenvector overlaps] \label{thm:eigenvector-overlaps}
Under Assumptions \ref{assump:high-dimension}-\ref{assump:bulk-regularity}, we have the following uniformly for all deterministic matrices $D \in \bbc^{(M+N) \times (M+N)}$ with $\norm{D} \lesssim 1$,
\begin{equation}
    \bigg | {\angles{\bfxi_{k, i}, D \bfxi_{\ell, j}} 
    - \delta_{k \ell} \delta_{ij} \frac{\angles{\Im \Pi(\gamma_{k,i}) D}}{2 \Im m(\gamma_{k,i})} } \bigg |
    \prec (N \fkn_{k,i} \fkn_{\ell,j})^{-1/6},
    \label{bound:overlap-xi-linearization}
\end{equation}
where $\fkn_{k,i} \equiv i \wedge (N_k+1-i)$ and the bound is also uniform in $k,\ell \in \llbracket K \rrbracket$, $i \in \llbracket N_k \rrbracket$ and $j \in \llbracket N_\ell \rrbracket$.
\end{theorem}

By specifying
\begin{equation*}
    D = \begin{bmatrix} D_1 & 0 \\ 0 & 0 \end{bmatrix},
    \quad
    D = \begin{bmatrix} 0 & 0 \\ 0 & D_2 \end{bmatrix}
    \qand
    D = \begin{bmatrix} 0 & D_3 \\ 0 & 0 \end{bmatrix},
\end{equation*}
respectively, we immediately obtain the following corollary of Theorem \ref{thm:eigenvector-overlaps}.

\begin{corollary}[Singular vector overlaps]
Under Assumptions \ref{assump:high-dimension}-\ref{assump:bulk-regularity}, we have the following uniformly for all deterministic matrices $D_1 \in \bbc^{M \times M}, D_2 \in \bbc^{N \times N}, D_3 \in \bbc^{M \times N}$ with $\norm{D_1}, \norm{D_2}, \norm{D_3} \lesssim 1$,
\begin{subequations} \label{bound:overlap-singular-all}
\begin{align}
    \bigg | {\angles{\bfu_{k, i}, D_1 \bfu_{\ell, j}} 
    - \delta_{k \ell} \delta_{ij} \frac{\angles{\Im \Gamma(\gamma_{k,i}) D_1}}{\Im m(\gamma_{k,i})} } \bigg |
    & \prec (N \fkn_{k,i} \fkn_{\ell,j})^{-1/6}, 
    \label{bound:overlap-singular-u-u} \\
    \big | { \angles{\bfv_{k, i}, D_2 \bfv_{\ell, j}} 
    - \delta_{k \ell} \delta_{ij} \angles{D_2} } \big |
    & \prec (N \fkn_{k,i} \fkn_{\ell,j})^{-1/6}, 
    \label{bound:overlap-singular-v-v} \\
    \big | { \angles{\bfu_{k, i}, D_3 \bfv_{\ell, j}} } \big |
    & \prec (N \fkn_{k,i} \fkn_{\ell,j})^{-1/6},
    \label{bound:overlap-singular-u-v}
\end{align}
\end{subequations}
where $\fkn_{k,i} \equiv i \wedge (N_k+1-i)$ and the bounds are also uniform in $k,\ell \in \llbracket K \rrbracket$, $i \in \llbracket N_k \rrbracket$ and $j \in \llbracket N_\ell \rrbracket$.
\end{corollary}

\begin{remark}[Convergence rate]
To gain insight into the convergence rate provided on the r.h.s. of (\ref{bound:overlap-xi-linearization}) and (\ref{bound:overlap-singular-all}), let us consider the simplest case where $K = 1$, meaning that $\operatorname{supp} \varrho$ contains only one bulk component. In this scenario, for $i = j < N_1/2$, (\ref{bound:overlap-singular-u-u}) simplifies to 
\begin{equation*}
    \bigg | {\angles{\bfu_{i}, D_1 \bfu_{i}} 
    - \frac{\angles{\Im \Gamma(\gamma_{i}) D_1}}{\Im m(\gamma_{i})} } \bigg |
    \prec N^{-1/6} i^{-1/3},
\end{equation*}
In the bulk regime where $i \asymp N$, this aligns with the convergence rate established in \cite{cipolloniOptimalLowerBound2023} for overlap of bulk singular vectors (in the context of a different random matrix ensemble). We anticipate this rate to be optimal, up to the arbitrary small prefactor $N^\varepsilon$ introduced by the definition of $\prec$. 

This convergence rate gradually degrades when transitioning to the edge regime $i \ll N$. Particularly, for eigenvectors corresponding to the largest eigenvalues, i.e., $\bfu_i$ with $i \lesssim 1$, the convergence rate we obtain reduces to $N^{-1/6}$. While this still provides a nontrivial quantification of the convergence of edge eigenvector overlaps towards their deterministic counterparts, which is sufficient for our statistical motivation, we acknowledge the possibility of improving upon this rate to $N^{-1/2}$, as demonstrated for bulk eigenvectors. Considering the scope and length of this article, we defer this exploration to future work.
\end{remark}

\begin{remark}[Relationship with Ledoit and P\'{e}ch\'{e}'s result]
The quantiles $\gamma_{k,i}$ strictly locate within the interior of $\operatorname{supp} \rho$, ensuring the strict positivity of the density $\Im m(\gamma_{k,i})$ in (\ref{bound:overlap-xi-linearization}) and (\ref{bound:overlap-singular-u-u}). The maps $E \mapsto \Im m(E)$ and $E \mapsto \angles{\Im \Gamma(E) D_1}$, up to a common factor, correspond to the asymptotic density of the measures $N^{-1} \sum_{i} (\delta_{s_i} + \delta_{-s_i})$ and $N^{-1} \sum_{i} \angles{\bfu_{i}, D_1 \bfu_{i}} (\delta_{s_i} + \delta_{-s_i})$, respectively. In essence, $\angles{\Im \Gamma(E) D_1} / \Im m(E)$ precisely represents the Radon-Nikodym derivative as described in \cite{ledoitEigenvectorsLargeSample2011}.

The concentration bound (\ref{bound:overlap-singular-u-u}) can be interpreted as a local refinement of Ledoit and P\'{e}ch\'{e}'s conclusion \cite{ledoitEigenvectorsLargeSample2011} regarding the weak convergence of $N^{-1} \sum_{i} \angles{\bfu_{i}, D_1 \bfu_{i}} (\delta_{s_i} + \delta_{-s_i})$. This is somewhat reminiscent of how the eigenvalue rigidity estimates refine the weak convergence of the ESD. Both the rigidity estimates and the concentration of eigenvector overlaps (first established in \cite{erdosRigidityEigenvaluesGeneralized2012} and \cite{cipolloniEigenstateThermalizationHypothesis2021} respectively, in the context of Wigner matrices) suggest that, in the realm of RMT, we can generally expect more than mere global convergence; to some extent, finer control over each individual spectral component is achievable. 
\end{remark}

\begin{remark}[Energy dependency]
The energy dependency of the deterministic counterparts of the eigenvector overlaps only manifests in the estimate (\ref{bound:overlap-singular-u-u}) and is absent from (\ref{bound:overlap-singular-v-v}) and (\ref{bound:overlap-singular-u-v}). This is consistent with the structure of $\Pi$ as outlined in (\ref{def:Pi}): the anisotropy is confined to its upper-left $M \times M$ block. In fact, in the null case where $\Sigma = I$, this anisotropy vanishes, and (\ref{bound:overlap-singular-u-u}) becomes energy independent,
\begin{equation*}
    \big | { \angles{\bfu_{k, i}, D_1 \bfu_{\ell, j}} 
    - \delta_{k \ell} \delta_{ij} (M^{-1} \operatorname{tr} D_1) } \big |
    \prec (N \fkn_{k,i} \fkn_{\ell,j})^{-1/6}.
\end{equation*}

We remark that the estimate (\ref{bound:overlap-singular-v-v}) on right singular vectors $\bfv_i$ closely resembles the corresponding estimates regarding eigenvectors of Wigner matrices \cite{cipolloniEigenstateThermalisationEdge2023,cipolloniEigenstateThermalizationHypothesis2021}. Moreover, we observe that the overlaps of right singular vectors of the matrix $\Sigma^{1/2} X$ share the same asymptotic limit as those of the undeformed matrix $X$. Consequently, we can conclude that, up to the first-order behavior, the left multiplicative deformation $\Sigma^{1/2}$ has no impact on the overlaps of right singular vectors. However, whether this deformation affects higher-order behaviors of $\angles{\bfv_{k, i}, D_2 \bfv_{\ell, j}}$, such as fluctuation, requires further investigation.

Finally, let us underline the distinctive pattern exhibited by (\ref{bound:overlap-singular-u-v}) compared to its counterpart \cite[Equation (2.8c)]{cipolloniOptimalLowerBound2023} for additively deformed i.i.d. square matrices. For the undeformed matrix $X$ with Gaussian entries, the estimate (\ref{bound:overlap-singular-u-v}) can be easily derived from the independence between the left and right eigenvectors. This result is anticipated to hold for $X$ with a general entry distribution due to universality. Our estimate (\ref{bound:overlap-singular-u-v}) suggests that the left multiplicative deformation maintains this weak correlation between the $\bfu_i$'s and the $\bfv_i$'s. Contrastingly, an additive deformation as in \cite{cipolloniOptimalLowerBound2023}, would disrupt this weak correlation, causing the overlap $\angles{\bfu_i, D_3 \bfv_i}$ to concentrate around a nonzero deterministic quantity.
\end{remark}

\section{Ledoit-Wolf's nonlinear shrinkage estimators} 
\label{sec:nonlinear-shrinkage}

In this section, we utilize the estimate (\ref{bound:overlap-singular-u-u}) to conduct a finite-sample analysis of Ledoit-Wolf's nonlinear shrinkage estimators. We mention that our results regarding eigenvector overlaps, as presented in Section \ref{sec:singular-vec-overlaps}, do not cover those eigenvectors in the null space of $\hat{\Sigma}$. Therefore, to facilitate discussion, we introduce the following additional assumption to ensure the invertibility of $\hat{\Sigma}$. The investigation of eigenvectors in the null space of $\hat{\Sigma}$ is deferred to future work.

\begin{assumption}[Invertibility of $\hat{\Sigma}$] \label{assump:invertibility}
Suppose $\tau \leq M / N \leq 1 - \tau$ and $\sigma_M \geq \tau$ uniformly for all $N$.
\end{assumption}

Note that Assumption \ref{assump:invertibility} automatically implies Assumptions \ref{assump:high-dimension} and \ref{assump:anti-concentration}. However, for simplicity, we still state the subsequent result under Assumptions \ref{assump:high-dimension}-\ref{assump:invertibility}. With Assumption \ref{assump:invertibility}, it is not difficult to verify that the relabelling introduced in (\ref{def:relabelling}) satisfies $\sum_{k = 1}^K N_k = M$. As per the regularity of the leftmost edge of $\varrho$ ensured by Assumption \ref{assump:edge-regularity}, as well as the eigenvalue rigidity estimate (\ref{bound:rigidity}), this implies that $\lambda_M$, the smallest eigenvalue of $\hat{\Sigma}$, is bounded below by some positive constant with high probability. Also note that now the estimate (\ref{bound:overlap-singular-u-u}) encompasses all the eigenvectors of $\hat{\Sigma}$.

The derivation of Ledoit-Wolf's nonlinear shrinkage estimators has been previously outlined. The transition from the finite-sample optimal yet infeasible shrinkage (\ref{def:oracle-overlap}) to the practical algorithmic one (\ref{def:practical-estimate}) hinges on the deterministic counterparts of eigenvector overlaps as described in (\ref{bound:overlap-singular-u-u}). Actually, by setting $D_1 = \Sigma$ and utilizing (\ref{eqn:self-consistent-z}), it is straightforward to verify that the formula of the transitional shrinkage introduced in (\ref{def:oracle-estimate-renewed}) aligns with this Radon-Nikodym derivative in (\ref{bound:overlap-singular-u-u}). The only difference is that, in (\ref{def:oracle-estimate-renewed}) the classical eigenvalue locations $\gamma_i^2$ are replaced by their empirical counterparts $\lambda_i$, whose impact is amenable due to eigenvalue rigidity (\ref{bound:rigidity}).

It is worth noting that this strategy extends beyond the Frobenius loss $\mathscr{L}^\fr$ and can be applied to various loss functions with statistical backgrounds, as explored by Ledoit and Wolf \cite{ledoitOptimalEstimationLargedimensional2018,ledoitShrinkageEstimationLarge2021}. Hence, besides the Frobenius loss, we also discuss another representative loss function, the so-called inverse Frobenius loss $\mathscr{L}^\finv$, which essentially represents the Frobenius loss for estimating the precision matrix $\Sigma^{-1}$. The corresponding formulas for the finite-sample optimal, transitional and algorithmic shrinkages under $\mathscr{L}^\finv$ are detailed in Table \ref{table:shrinkage}. In this case, the transitional shrinkages $\{\lambda^{\finv \star}_i\}$ are derived by setting $D_1 = \Sigma^{-1}$ in (\ref{bound:overlap-singular-u-u}).

\begin{table}[htbp]
    \centering
    \begin{tblr}{colspec={ccc}, row{1} = {font=\bfseries}, hline{1,2,6} = {solid,1pt},
        hline{3-5} = {dashed,0.5pt}, colsep = 12pt, rowsep = 3pt}
        Loss function & Frobenius & Inverse Frobenius \\
        Formula of loss 
        & $\mathscr{L}^\fr (\hat{\Sigma}^\rie, \Sigma) = M^{-1} \norm{\hat{\Sigma}^\rie - \Sigma}_{\mathrm{F}}^2$
        & $\mathscr{L}^\finv (\hat{\Sigma}^\rie, \Sigma) = M^{-1} \norm{(\hat{\Sigma}^\rie)^{-1} - \Sigma^{-1}}_{\mathrm{F}}^2$ \\ 
        Oracle
        & $\lambda^\fr_i = \angles{\bfu_i, \Sigma \bfu_i}$ 
        & $\lambda^\finv_i = 1/{\angles{\bfu_i, \Sigma^{-1} \bfu_i}}$ \\
        Transitional
        & $\lambda^{\fr \star}_i = \dfrac{1}{\lambda_i \abs{\fkm (\lambda_i)}^2}$ 
        & $\lambda^{\finv \star}_i = - \dfrac{\lambda_i}{1 - {M}/{N} + 2 \lambda_i \Re [\fkm (\lambda_i)]}$ \\
        Algorithmic  
        & $\hat{\lambda}^{\fr \star}_i = \dfrac{1}{\lambda_i \abs{\hat{\fkm} (\lambda_i)}^2}$ 
        & $\hat{\lambda}^{\finv \star}_i = - \dfrac{\lambda_i}{1 - {M}/{N} + 2 \lambda_i \Re [\hat{\fkm} (\lambda_i)]}$ \\
    \end{tblr}    
    \caption{Nonlinear shrinkage under Frobenius/inverse Frobenius loss}
    \label{table:shrinkage}
\end{table}

For a rigorous investigation of these algorithmic shrinkage estimators, it is necessary to clarify how $\fkm$, the Stieltjes transform of the LSD, is estimated. A natural estimator of $\fkm$ is its empirical counterpart, the Stieltjes transform of the ESD of $Y^\top Y / N$,
\begin{equation}
    \fkg (z) := \frac{1}{N} \sum_{j=1}^N \frac{1}{\lambda_j - z}
    = \frac{1}{N} \sum_{j=1}^M \frac{1}{\lambda_j - z} - \bigg ( {1 - \frac{M}{N}} \bigg ) \frac{1}{z}.
\end{equation}
Note that we have let $\lambda_{M+1} = \cdots = \lambda_N = 0$. However, directly setting $\hat{\fkm} (\lambda_i) = \fkg (\lambda_i)$ is not viable due to the singularity of $\fkg$ at $\lambda_i$. As previously discussed, currently there exist two computationally efficient solutions for this issue. One strategy, proposed by \cite{ledoitAnalyticalNonlinearShrinkage2020}, involves the use of a kernel to smooth the ESD, thereby enabling the evaluation of its Stieltjes transform at $\lambda_i$. Another approach, suggested by \cite{benaych-georgesShortProofLedoitPeche2023}, incorporates an additional imaginary part for the spectral parameter $0 < \eta \ll 1$ to mitigate this singularity,
\begin{equation}
    \hat{\fkm} (\lambda_i) = \fkg (\lambda_i + \mathrm{i} \eta).
    \label{def:emprical-estimator-Stieltjes}
\end{equation}
In this article we adopt the latter approach for analysis, as it is more transparent from the perspective of RMT. Another motivation for this selection is to circumvent the introduction of additional notations caused by kernel smoothing, which could potentially complicate the work. Nevertheless, we emphasis that this choice is not particularly significant, as the finite-sample analysis in this article can be straightforwardly adapted to analyze the shrinkage estimators implemented in \cite{ledoitAnalyticalNonlinearShrinkage2020}. In fact, the most formidable aspect of the analysis lies in quantifying the loss from the finite-sample optimal shrinkages to the transitional ones, which is addressed by our main estimate (\ref{bound:overlap-singular-u-u}). Conversely, the loss from the transitional shrinkages to the algorithmic ones is relatively manageable in viewing of existing tools \cite{jingNonparametricEstimateSpectral2010,knowlesAnisotropicLocalLaws2017}.

Algorithm \ref{algo:shrinkage} below outlines the entire estimation process, where we operate under the Frobenius loss. Its adaption to the inverse Frobenius Loss is straightforward using Table \ref{table:shrinkage} and (\ref{def:emprical-estimator-Stieltjes}). Note that in Algorithm \ref{algo:shrinkage} we have locally adapted the scale parameters $\eta_i$ to ensure that the resulting estimator fulfills the scale-equivariant property discussed in \cite{ledoitAnalyticalNonlinearShrinkage2020}.

\begin{algorithm}[htbp]
    \SetAlgoLined
    \KwIn{Data matrix $Y \in \mathbb{R}^{M \times N}$ with $M \leq N$; scale parameter $\eta > 0$ ($N$-dependent).}
    \KwOut{Nonlinear shrinkage estimator $\hat{\Sigma}^{\fr \star}$.}
    
    \textbf{Step 1:} Calculate the sample covariance matrix $\hat{\Sigma}$\;
    \textbf{Step 2:} Compute spectral decomposition of $\hat{\Sigma}$ as $\hat{\Sigma} = \sum_{i = 1}^{M} \lambda_i \mathbf{u}_i \mathbf{u}_i^\top$\;
    \textbf{Step 3:} \ForEach{$i \in \llbracket M \rrbracket$}{
        \begin{enumerate}[label=(\arabic*),noitemsep,nolistsep,leftmargin=*]
            \item Set $z_i = \lambda_i + \mathrm{i} \eta_i$ with $\eta_i = \lambda_i \eta$\;
            \item Compute the Stieltjes transform $\fkg (z_i) = N^{-1} \sum_{j=1}^M 1/(\lambda_j - z_i) - ( 1 - {M}/{N} ) / z_i$\;
            \item Calculate the shrunk eigenvalue $\hat{\lambda}^{\fr \star}_i = 1/(\lambda_i \abs{\fkg (z_i)}^2)$\;
        \end{enumerate}
    }
    \textbf{Step 4:} Return the nonlinear shrinkage estimator $\hat{\Sigma}^{\fr \star} = \sum_{i = 1}^{M} \hat{\lambda}^{\fr \star}_i \bfu_i \bfu_i^\top$\;
    
    \caption{Nonlinear Shrinkage (Frobenius Loss)}
    \label{algo:shrinkage}
\end{algorithm}

We are now in a position to present our main statistical finding concerning the finite-sample performance of nonlinear shrinkage estimators.

\begin{theorem}[Nonlinear shrinkage estimators] \label{thm:shrinkage}
Let $\eta \equiv \eta^{(N)}$ be some $N$-dependent scale parameter satisfying $\eta \in [N^{-2/3+c}, N^{-c}]$, where $c > 0$ is a small constant. Suppose that $\hat{\Sigma}^{\fr \star}$ is the nonlinear shrinkage estimator obtained from Algorithm \ref{algo:shrinkage}. Recall the relabelling introduced in (\ref{def:relabelling}). Then, under Assumptions \ref{assump:high-dimension}-\ref{assump:invertibility}, we have the entrywise bound
\begin{equation}
    \abs{\hat{\lambda}_{k,i}^{\fr \star} - \lambda_{k,i}^{\fr}}
    \prec \frac{1}{N^{1/6} (\fkn_{k,i})^{1/3}} 
    + \frac{\eta}{(\fkn_{k,i}/N)^{1/3} + \eta^{1/2}} 
    + \frac{1}{N \eta},
    \label{bound:entrywise-optimal-algo}
\end{equation}
where $k \in \llbracket K \rrbracket$, $i \in \llbracket N_k \rrbracket$ and $\fkn_{k,i} \equiv i \wedge (N_k+1-i)$. Consequently, in terms of the loss function, we have
\begin{equation}
    \mathscr{L}^\fr (\hat{\Sigma}^{\fr \star}, \Sigma^\fr) 
    \prec \frac{1}{N} + \eta^2 + \frac{1}{N^2 \eta^2}.
    \label{bound:loss-optimal-algo}
\end{equation}
In addition, if $\hat{\Sigma}^{\finv \star}$ is the estimator derived by adapting Algorithm \ref{algo:shrinkage} to the inverse Frobenius loss (defined in Table \ref{table:shrinkage}), then the estimates (\ref{bound:entrywise-optimal-algo}) and (\ref{bound:loss-optimal-algo}) also hold for $\hat{\Sigma}^{\finv \star}$.
\end{theorem}

Considering the r.h.s. of the bound (\ref{bound:loss-optimal-algo}) on the loss, a favorable choice for the scale parameter is $\eta = N^{-1/2}$. With this choice, we have the following corollary of Theorem \ref{thm:shrinkage}.

\begin{corollary}[Nonlinear shrinkage estimators with $\eta = N^{-1/2}$] \label{coro:shrinkage-half}
In the same context as described in Theorem \ref{thm:shrinkage}, assuming $\eta = N^{-1/2}$, the estimate (\ref{bound:entrywise-optimal-algo}) reduces to
\begin{equation}
    \abs{\hat{\lambda}_{k,i}^{\fr \star} - \lambda_{k,i}^{\fr}} \prec N^{-1/6} (\fkn_{k,i})^{-1/3}
    \label{bound:optimal-algo-eta-half-entrywise}
\end{equation}
Moreover, this entrywise estimate implies that
\begin{equation}
    \mathscr{L}^\fr (\hat{\Sigma}^{\fr \star}, \Sigma^\fr) \prec N^{-1}
    \qand
    \norm{\hat{\Sigma}^{\fr \star} - \Sigma^\fr} \prec N^{-1/6}.
    \label{bound:optimal-algo-eta-half-norm}
\end{equation}
These bounds also apply to the estimator $\hat{\Sigma}^{\finv \star}$.
\end{corollary}

It should be emphasized that Theorem \ref{thm:shrinkage} and Corollary \ref{coro:shrinkage-half} are not intended to justify the convergence of $\hat{\Sigma}^{\fr \star}$ towards the population covariance $\Sigma$. Actually, according to \cite[Theorem 4.2]{ledoitOptimalEstimationLargedimensional2018}, the loss $\mathscr{L}^\fr ({\Sigma}^{\fr}, \Sigma)$ is typically of a constant order. In other words, within the class of RIEs, even the finite-sample optimal estimator $\Sigma^\fr$ is incapable of achieving consistency without additional structural assumptions on $\Sigma$. The primary contribution of Theorem \ref{thm:shrinkage} and Corollary \ref{coro:shrinkage-half} lies in quantifying the loss incurred when employing $\hat{\Sigma}^{\fr \star}$ as a feasible surrogate for ${\Sigma}^{\fr}$. Fortunately, with an appropriate choice of the scale parameter $\eta = N^{-1/2}$, it is found that this loss is dominated by the rate of $N^{-1}$, which is sufficiently rapid to convince the practitioners to adopt the algorithmic estimator $\hat{\Sigma}^{\fr \star}$ as an approximation of ${\Sigma}^{\fr}$.

As highlighted, our approach actually yields a more general entrywise estimate (\ref{bound:optimal-algo-eta-half-entrywise}) that goes beyond just assessing the loss. This entrywise estimate even implies the convergence of $\hat{\Sigma}^{\fr \star}$ towards $\Sigma^\fr$ in terms of the operator norm. However, it should be noted that, akin to Theorem \ref{thm:eigenvector-overlaps}, we do not seek for the optimal convergence rate for those indices $i$ corresponding to the eigenvalues near the spectral edges. Consequently, there is potential to enhance this rate for the edge regime, thus refining the resulting convergence rate of $\norm{\hat{\Sigma}^{\fr \star} - \Sigma^\fr}$.

The entrywise convergence (\ref{bound:optimal-algo-eta-half-entrywise}) is depited in Figure \ref{fig:entrywise}, where the finite-sample optimal shrinkages $\lambda^\fr$ (resp. $\lambda^\finv$) and their algorithmic counterparts $\hat{\lambda}^{\fr \star}$ (resp. $\hat{\lambda}^{\finv \star}$) are juxtaposed. These values are computed from a $1000 \times 1000$ sample covariance matrix. In this Monte Carlo simulation, we allocate $20\%$ of the population eigenvalues as $1$, $40\%$ as $3$, and $40\%$ as $10$, a configuration notably discussed and analyzed in detail by \cite{baiNoEigenvaluesOutside1998}. Refer to \cite{ledoitAnalyticalNonlinearShrinkage2020} for a comprehensive simulation study on the loss of nonlinear shrinkage estimators under this population covariance setting. For the simulation corresponding to Figure \ref{fig:entrywise}, a sample size of $N = 2000$ is used, the scale parameter is set as $\eta = N^{-1/2}$, and $X$ is generated from i.i.d. normal variables of zero mean and variance $1/N$. The simulation results confirm our theoretical findings: an entrywise alignment of the algorithmic shrinkages with the finite-sample optimal ones is observed.

\begin{figure}[htbp]
    \centering
    \includegraphics[width=\textwidth]{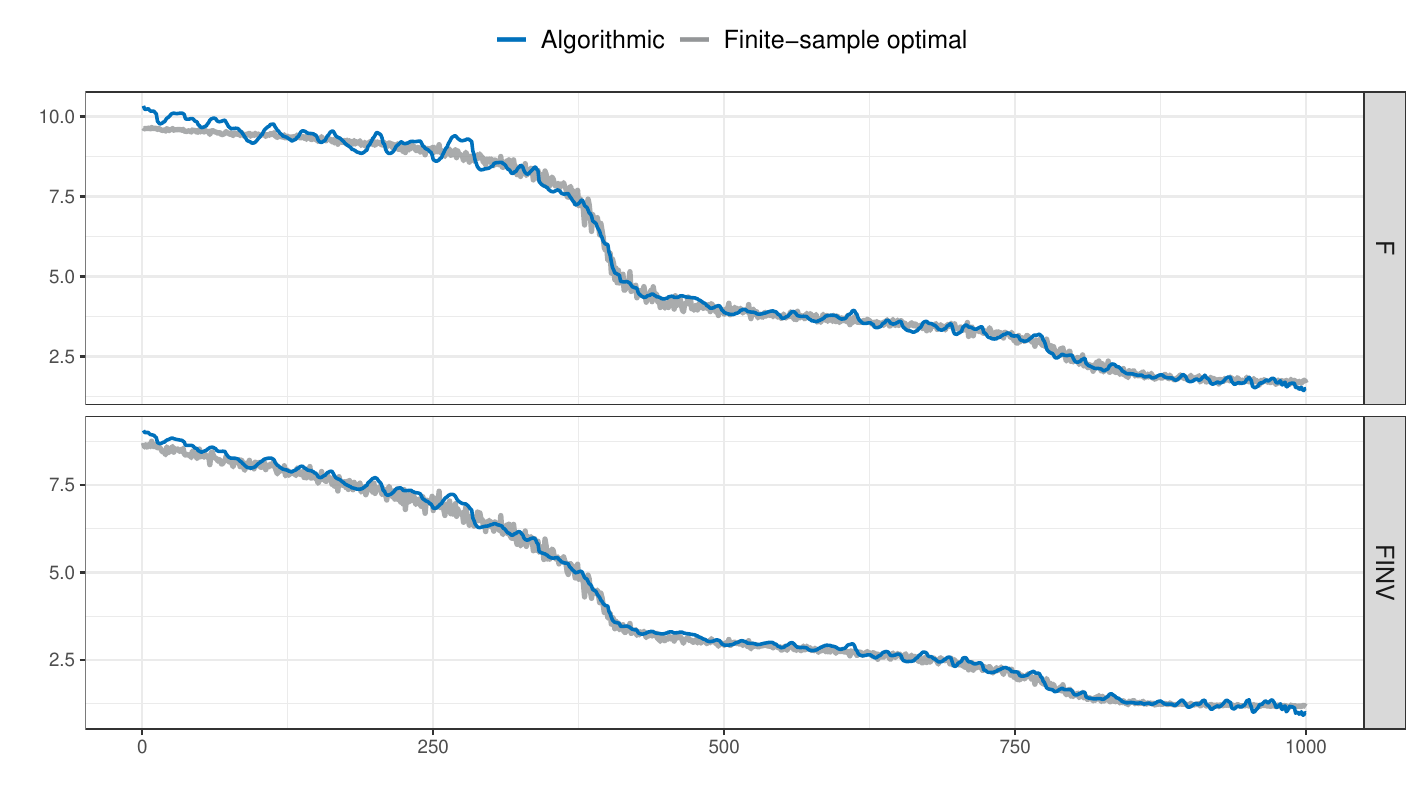}
    \caption{Finite-sample optimal shrinkages v.s. their algorithmic counterparts}
    \label{fig:entrywise}
\end{figure}

We conclude this section by two additional remarks on Theorem \ref{thm:shrinkage} and Corollary \ref{coro:shrinkage-half}.

\begin{remark}[Probabilistic bounds]
As previously clarified, the notation $\mathcal{Y} \prec \mathcal{Z}$ is used to signify that ``$\mathcal{Y}$ is bounded with high probability by $\mathcal{Z}$ up to small powers of $N$''. Nevertheless, it would be beneficial to further elaborate on this notation in (\ref{bound:optimal-algo-eta-half-entrywise}) and (\ref{bound:optimal-algo-eta-half-norm}), and to rephrase these bounds in a way that statisticians are more accustomed to. 

Firstly, we remark that the entrywise bound in (\ref{bound:optimal-algo-eta-half-entrywise}) is uniform across indices $k \in \llbracket K \rrbracket$ and $i \in \llbracket N_k \rrbracket$. As the cardinality of the corresponding index set is bounded by powers of $N$, a simple union bound allows the interchange of $\bbp$ with the supremum over indices in (\ref{def:stochastic-domination}), without degrading the probability tail. Specifically, for arbitrary (small) $\varepsilon > 0$ and (large) $L > 0$, we have
\begin{equation*}
    \bbp \Big \{ \max_{1 \leq k \leq K} \max_{1 \leq i \leq N_k} {(\fkn_{k,i})^{1/3} \abs{\hat{\lambda}_{k,i}^{\fr \star} - \lambda_{k,i}^{\fr}} \leq N^{-1/6 + \varepsilon}} \Big \}
    \geq 1 - N^{-L}
\end{equation*}
for all sufficiently large $N \geq N_0(\varepsilon, L)$. 

Similarly, one can use the definition (\ref{def:stochastic-domination}) to explicitly formulate the probabilistic bounds on $\mathscr{L}^\fr (\hat{\Sigma}^{\fr \star}, \Sigma^\fr)$ and $\norm{\hat{\Sigma}^{\fr \star} - \Sigma^\fr}$. Here we would like to mention the compatibility of $\prec$ with $\bbe$, as introduced in Lemma \ref{lemma:prec-expectation} below. Consequently, the first bound in (\ref{bound:optimal-algo-eta-half-norm}) implies that, for arbitrary (small) $\varepsilon > 0$, we have
\begin{equation*}
    \bbe \mathscr{L}^\fr (\hat{\Sigma}^{\fr \star}, \Sigma^\fr) 
    = M^{-1} \bbe \norm{\hat{\Sigma}^{\fr \star} - \Sigma^\fr}_{\mathrm{F}}^2
    \leq N^{-1+\varepsilon}
\end{equation*}
for all sufficiently large $N \geq N_0(\varepsilon)$.
\end{remark}

\begin{remark}[Loss functions]
As previously noted, Ledoit and Wolf \cite{ledoitOptimalEstimationLargedimensional2018,ledoitShrinkageEstimationLarge2021} have explored various loss functions with statistical backgrounds. Notably, both the inverse Stein loss and the minimum variance loss yield the same finite-sample optimal shrinkages as the Frobenius loss. Consequently, the analysis of estimators derived from these loss functions is fundamentally similar, with the exception of using a different formula when evaluating the loss. Overall speaking, this evaluation is not difficult to achieve by leveraging the entrywise bound (\ref{bound:optimal-algo-eta-half-entrywise}).
\end{remark}

\section{Multi-resolvent local laws} \label{sec:multi-local-law}

Unless otherwise specified, the subsequent theorems are all stated under Assumptions \ref{assump:high-dimension}-\ref{assump:bulk-regularity}.

\subsection{Single resolvent local laws and matrix Dyson equation} 
\label{subsec:single-resolvent-laws}

This section provides a concise overview of the single resolvent local laws for sample covariance matrices \cite{baoUniversalityLargestEigenvalue2015,knowlesAnisotropicLocalLaws2017} and the matrix Dyson equation \cite{erdosMatrixDysonEquation2019}. For a more comprehensive introduction to these topics, readers are encouraged to explore the works cited and the references therein. 

The spectral parameter $w$ we consider always falls into the fundamental domain
\begin{equation}
    \bbd \equiv \bbd (\tau, \tau^\prime)
    := \{ {w \in \bbc: \Re w > \tau^\prime, \
    \operatorname{dist} (\Re w, \operatorname{supp} \rho) \leq \tau^\prime 
    \text{ and }
    0 < \abs{\Im w} \leq \tau^{-1}} \}.
    \label{def:bbd-all}
\end{equation}
where $\tau^\prime > 0$ is some sufficiently small constant (depending only on $\tau$). Throughout this article, we reserve $w$ for the spectral parameter of the functions $m$, $\Pi$ and $G$, and $z = w^2$ for the spectral parameter of $\fkm$. It is worth noting that for any $w \in \bbd$, one can uniquely revert $w$ from its square $z = w^2$ through $w = \sqrt{z}$, where the square root is selected to have a branch cut along the negative real line. Consequently, functions of $w$ can effortlessly be regarded as functions of $z$, and vice versa. Throughout the subsequent discussion, we will frequently utilize this one-to-one correspondence without further explanation. 

Thanks to Assumption \ref{assump:edge-regularity}, the following comparison holds uniformly for $w \in \bbd$ and $z = w^2$,
\begin{equation}
    \abs{w} \asymp \abs{z} \asymp \Re w \asymp \Re z \asymp 1,
    \quad \abs{\Im w} \asymp \abs{\Im z}
    \qand \operatorname{dist} (w, \partial \operatorname{supp} \rho)
    \asymp \operatorname{dist} (z, \partial \operatorname{supp} \varrho).
\end{equation}
We denote the distance of $w \in \bbd$ to the real line and the closest spectral edge as
\begin{equation}
    \eta \equiv \eta (w)
    := \abs{\Im w}
    \qand
    d \equiv d(z) 
    := \operatorname{dist} (z, \partial \operatorname{supp} \varrho).
\end{equation}

The following lemma summarizes some basic estimates regarding the Stieltjes transform $m$, specifically ensuring the boundedness of the deterministic approximation $\Pi$. 

\begin{lemma}[Basic properties of $m$] \label{lemma:m-order-estimate}
\begin{subequations}
For sufficiently small $\tau^\prime > 0$ (depending only on $\tau$), the following estimates hold uniformly for $w \in \bbd (\tau, \tau^\prime)$,
\begin{equation}
    \abs{m(w)} \asymp \abs{\fkm (z)} \asymp 1
    \qand \abs{\Im m (w)} \asymp \abs{\Im \fkm (z)}
    \asymp \begin{cases}
        d^{1/2}, & \text{ if } \Re w \in \operatorname{supp} \rho, \\
        \eta / d^{1/2}, & \text{ if } \Re w \notin \operatorname{supp} \rho.
    \end{cases} 
    \label{eqn:m-order-estimate}
\end{equation}
In addition, we have $\| \Pi(w) \|, \| \Gamma(w) \| \lesssim 1$ that follows from
\begin{equation}
    \min_{i \in \llbracket M \rrbracket} \abs{w + m(w) \sigma_{i}} 
    \asymp \min_{i \in \llbracket M \rrbracket} \abs{1 + \fkm(z) \sigma_{i}} \gtrsim 1.
    \label{eqn:Pi-stability}
\end{equation}    
\end{subequations}
\end{lemma}

Apart from the apparent comparison between $m(w)$ and $\fkm(z)$, the estimates in Lemma \ref{lemma:m-order-estimate} have already been established in \cite{baoUniversalityLargestEigenvalue2015,knowlesAnisotropicLocalLaws2017} in terms of $\fkm$. Specifically, the estimate for $\abs{\fkm(z)}$ can be proved as in \cite[Theorem 3.1]{baoUniversalityLargestEigenvalue2015}. The estimates for $\abs{\Im \fkm(z)}$ and $\min_i \abs{1 + \fkm(z) \sigma_{i}}$ have been established in \cite[Lemma A.4]{knowlesAnisotropicLocalLaws2017} near the edge and in \cite[Lemma A.6]{knowlesAnisotropicLocalLaws2017} in the bulk.

According to (\ref{eqn:m-order-estimate}), it is possible for the ``density'' $\abs{\Im m (w)}$ to be much smaller than $\eta^{1/2}$ if $\Re w \notin \operatorname{supp} \rho$ and $d \gg \eta$. This low-density regime presents challenges for theoretical analysis, as many parts of our proof require a lower bound on $\abs{\Im m (w)}$. Consequently, for most of the results presented in this article, we need to exclude this regime and focus on the spectral parameters from
\begin{equation}
    \bbd_0 \equiv \bbd_0 (\tau, \tau^\prime)
    := \{ {w \in \bbd: \operatorname{dist} (\Re w, \operatorname{supp} \rho) \leq \abs{\Im w}} \}.
    \label{def:bbd-0}
\end{equation}
Note that this definition guarantees $\abs{\Im m (w)} \asymp d^{1/2}$ uniformly for $w \in \bbd_0 (\tau, \tau^\prime)$.

\begin{notations*}
Let us make some notational conventions before proceeding. We use $\bfId_M, \bfId_N$ and $\bfSigma_M$ to denote the natural embeddings of the matrices $I_M, I_N$ and $\Sigma$ in $\bbr^{(M + N) \times (M \times N)}$, respectively, i.e.
\begin{equation*}
    \bfId_M \equiv \begin{bmatrix} I_M & 0 \\ 0 & 0 \end{bmatrix},
    \quad
    \bfId_N \equiv \begin{bmatrix} 0 & 0 \\ 0 & I_N \end{bmatrix},
    \qand
    \bfSigma_M \equiv \begin{bmatrix} \Sigma & 0 \\ 0 & 0 \end{bmatrix}.
\end{equation*}
We also introduce the matrices
\begin{equation}
    \bfId^{\pm} := \bfId_M \pm \bfId_N = \begin{bmatrix} I_M & 0 \\ 0 & \pm I_N \end{bmatrix}
    \qand
    \bfSigma^{\pm} := \bfSigma_M \pm \bfId_N = \begin{bmatrix} \Sigma & 0 \\ 0 & \pm I_N \end{bmatrix}.
    \label{def:Sigma-pm}
\end{equation}
When there is no ambiguity, we use the abbreviations $\bfId \equiv \bfId^+$ and $\bfSigma \equiv \bfSigma^+$. 

We often use the letter $i$ to represent indices within $\llbracket M \rrbracket$, while reserve $\mu$ for indices within $\llbracket N \rrbracket$ (or $\llbracket M+1, M+N \rrbracket$, depending on the specific context). Hence, summation over $i$ (resp. $\mu$) without further specification should be interpreted as $\sum_{i \in \llbracket M \rrbracket}$ (resp. $\sum_{\mu \in \llbracket N \rrbracket}$). Continuing with this convention, we denote the canonical bases of $\bbr^M$ and $\bbr^N$ as $\{ \mathbf{e}_i \}_{i \in \llbracket M \rrbracket}$ and $\{ \mathbf{e}_\mu \}_{\mu \in \llbracket N \rrbracket}$, respectively. These two set of vectors can also be identified with their natural embeddings in $\bbr^{M + N}$, so that their union constitutes the canonical bases of $\bbr^{M + N}$.
\end{notations*}

We now present the single resolvent local laws for sample covariance matrices, both in \emph{averaged} and \emph{isotropic form}, established in \cite[Theorem 3.6]{knowlesAnisotropicLocalLaws2017}.

\begin{proposition}[Single resolvent local laws] \label{prop:single-resolvent-local}
Fix arbitrary (small) $\varepsilon \in (0, 1)$. Under assumptions \ref{assump:high-dimension}-\ref{assump:bulk-regularity}, we have the following average law and isotropic law uniformly in $w \in \bbd (\tau, \tau^\prime)$ with $\eta = \abs{\Im w} \geq N^{-1 + \varepsilon}$ and deterministic vectors $\fku, \fkv \in \bbc^{M+N}$ with $\norm{\fku}, \norm{\fkv} \lesssim 1$,
\begin{equation}
    \abs{\angles{G(w) \bfSigma^\pm} - \angles{\Pi(w) \bfSigma^\pm}} \prec \frac{1}{N \eta}
    \qand
    \abs{\angles{\fku, G(w) \fkv} - \angles{\fku, \Pi(w) \fkv}}
    \prec \frac{1}{\sqrt{N \eta}}.
    \label{eqn:single-resolvent-local-law}
\end{equation}
\end{proposition} 

As a standard corollary of the single resolvent local laws, we have the following rigidity result, obtained in \cite[Theorem 3.12]{knowlesAnisotropicLocalLaws2017}, which offers large deviation bounds on the locations of individual singular values/eigenvalues.

\begin{proposition}[Singular value/eigenvalue rigidity] \label{prop:rigidity}
Under assumptions \ref{assump:high-dimension}-\ref{assump:bulk-regularity}, we have
\begin{equation}
    \abs{s_{k,i} - \gamma_{k,i}} + \abs{\lambda_{k,i} - \gamma_{k,i}^2} 
    \prec N^{-2/3} (\fkn_{k,i})^{-1/3}
    \equiv N^{-2/3} [i \wedge (N_k+1-i)]^{-1/3},
    \label{bound:rigidity}
\end{equation}
uniformly in $k \in \llbracket K \rrbracket$ and $i \in \llbracket N_k \rrbracket$.
\end{proposition} 

Although Proposition \ref{prop:single-resolvent-local} is stated for spectral parameters with positive real parts, its extension to parameters with negative real parts is straightforward thanks to the chiral symmetry
\begin{equation}
    \bfId^- G (w) \bfId^- = - G (- w).
        \label{eqn:chiral-symmetry}
\end{equation}    

The approach employed in \cite{knowlesAnisotropicLocalLaws2017} to establish Proposition \ref{prop:single-resolvent-local} is essentially grounded in the Schur complement formula. Here, we opt for the approach of cumulant expansion (reduced to Stein's lemma when dealing with Gaussian variables), which offers greater adaptability for deriving multi-resolvent local laws. The cumulant expansion formula presented here can be derived through a slight modification of \cite[Proposition 3.1]{lytovaCentralLimitTheorem2009}.

\begin{lemma}[Cumulant expansion formula] \label{lemma:cumulant-expansion}
Let $n \in \mathbb{N}$ be fixed and $g \in C^{n+1}(\bbr)$. Supposed $\xi$ is a centered real random variable with finite moments to order $n + 2$. Denote the $r$-th cumulant of $\xi$ by $\kappa_r (\xi)$. Then, we have the expansion
\begin{equation}
    \mathbb{E} [ \xi g(\xi) ]
    = \sum_{r=0}^{n} \frac{\kappa_{r+1}(\xi)}{r !} 
    \mathbb{E} [g^{(r)}(\xi)] + \mathcal{R}_n,
    \label{eqn:cumulant-expansion}
\end{equation}
assuming that all expectations on the r.h.s. exist. The remainder term $\mathcal{R}_n$ satisfies, for any $s > 0$,
\begin{equation*}
    \abs{\mathcal{R}_n}
    \leq \frac{(Cn)^n}{n!} \bigg [ \mathbb{E} \abs{\xi}^{n+2} 
    \cdot \sup_{\abs{t} \leq s} \big | { g^{(n+1)}(t) } \big |
    + \mathbb{E} \big [ \abs{\xi}^{n+2} \mathbb{1}(\abs{\xi}>s) \big ]
    \cdot \sup_{t \in \bbr} \big | { g^{(n+1)}(t) } \big | \bigg ].
\end{equation*}
\end{lemma}

To apply the cumulant expansion formula (\ref{eqn:cumulant-expansion}) to the resolvent $G$, we start by representing the self-adjoint dilation $H$ as a sum of independent rank-$2$ matrices,
\begin{equation}
    H = \sum_{i, \mu} x_{i \mu} \dimu,
    \qwhere
    \dimu = \bfSigma^{1/2} (\mathbf{e}_i \mathbf{e}_\mu^\top 
    + \mathbf{e}_\mu \mathbf{e}_i^\top) \bfSigma^{1/2}.
    \label{eqn:H-sum-of-rank2}
\end{equation}
Given any differentiable $g(H)$, we define its derivative w.r.t. $x_{i \mu}$ as
\begin{equation*}
    \parimu g(H)
    \equiv \frac{\partial}{\partial x_{i \mu}} g(H)
    := \frac{\mathrm{d}}{\mathrm{d} t} g( H + t \dimu) \Big |_{t=0}.
\end{equation*}
In particular, we have
\begin{equation}
    \parimu H = \dimu
    \qand
    \parimu G = - G (\parimu H) G = - G \dimu G.
    \label{eqn:derivative-G}
\end{equation}

\begin{definition}[Second-order renormalization]
Let $g_1, g_2: \bbr \to \bbc$. Following \cite{cipolloniEigenstateThermalizationHypothesis2021}, we define the \emph{second-order renormalization} of $g_1(H) H g_2(H)$, denoted with an underline, as
\begin{equation}
    \underline{g_1(H) H g_2(H)}
    := g_1(H) H g_2(H) - \frac{1}{N} \sum_{i, \mu} \parimu [ g_1(H) \dimu g_2(H) ].
    \label{def:second-order-renormalization}
\end{equation}    
\end{definition}

The summation in (\ref{def:second-order-renormalization}) essentially corresponds to the second-order term in the cumulant expansion formula (\ref{eqn:cumulant-expansion}). Particularly, when the $x_{i \mu}$'s are Gaussian, we have $\bbe \underline{g_1(H) H g_2(H)} = 0$. Therefore, by concentration and universality, it is reasonable to expect that the summation in (\ref{def:second-order-renormalization}) captures the leading term of $g_1(H) H g_2(H)$, and in a certain sense, $\underline{g_1(H) H g_2(H)}$ collects the terms that are subleading to this term.

Utilizing (\ref{eqn:derivative-G}) and $HG = I + w G$, we have by definition
\begin{equation}
    \underline{HG}
    = H G + \frac{1}{N} \sum_{i, \mu} \dimu G \dimu G
    = I + w G + \opS [G] G,
    \label{eqn:underline-HG}
\end{equation}
where we introduced the linear operator $\opS$ on $(M+N) \times (M+N)$ matrices,
\begin{equation}
    A \quad \mapsto \quad \opS [A]
    := \frac{1}{N} \sum_{i, \mu} \dimu A \dimu.
\end{equation}
Note that we can decompose $\opS = \opSd + \opSo$, with
\begin{subequations}
\begin{align}
    \opSd [A] & := \angles{A \bfId_N} \bfSigma_M + \angles{A \bfSigma_M} \bfId_N
    = \frac{1}{2} ( {\angles{A \bfSigma^+} \bfSigma^+ 
    - \angles{A \bfSigma^-} \bfSigma^-} ), 
    \label{def:opS-diagonal} \\
    \opSo [A] & := \frac{1}{N} ({\bfSigma_M A^\top \bfId_N 
    + \bfId_N A^\top \bfSigma_M}) 
    = \frac{1}{2N} ({\bfSigma^+ A^\top \bfSigma^+ 
    - \bfSigma^-  A^\top \bfSigma^-})
    \label{def:opS-off-diagonal}.
\end{align}    
\end{subequations}
The motivation behind this decomposition is that, given any $A$, the image $\opSd [A]$ is block-diagonal, whereas $\opSo [A]$ is block-off-diagonal. In the context of this article, for terms involving $\opS$, the primary contribution is typically captured by the part involving $\opSd$, whereas the contribution from the part involving $\opSo$ is negligible up to the leading order. This is not surprising since the prefactor $N^{-1}$ in (\ref{def:opS-off-diagonal}) inherently diminishes the importance of $\opSo$.

Now, referring to (\ref{eqn:underline-HG}) and the preceding discussion, it is reasonable to expect that $\Pi(w)$, the deterministic approximation of $G(w)$, satisfies the following so-called \emph{matrix Dyson equation} (MDE)
\begin{equation}
    I + w \Pi + \opSd [\Pi] \Pi 
    = I + w \Pi + \Pi \opSd [\Pi] = 0.
    \label{eqn:single-MDE}
\end{equation}
This is exactly the case for $\Pi(w)$ defined in (\ref{def:Pi}). Actually, utilizing the self-consistent equation (\ref{eqn:self-consistent-z}), we know that $m(w)$ can also be characterized as the unique solution to
\begin{equation}
    w = - \frac{1}{m (w)} + \frac{1}{N} \operatorname{tr} \frac{\Sigma}{w + m(w) \Sigma},
    \quad \text{ with } \quad 
    \Im w \cdot \Im m > 0.
    \label{eqn:self-consistent-w}
\end{equation}

\subsection{Multi-resolvent local laws with regular matrices} \label{subsec:multi-resolvent-laws}

As previously mentioned, the key technical result of this article, leading to Theorem \ref{thm:eigenvector-overlaps}, is the multi-resolvent local laws, which extend Proposition \ref{prop:single-resolvent-local} to chains of resolvents and deterministic matrices. For our purposes, we primarily focus on resolvent chains of length $2$ or $3$,
\begin{equation}
    G_1 A_1 G_2 
    \qand
    G_1 A_1 G_2 A_2 G_3,
    \label{eqn:resolvent-chains}
\end{equation}
where $A_1, A_2$ are deterministic matrices. Here we also abbreviate $G_k \equiv G(w_k)$. For what follows, we frequently use such abbreviations to streamline notation when there are multiple spectral parameters. In particular, we also write $\Pi_k \equiv \Pi(w_k)$. Similar to the single resolvent local laws, we can identify certain deterministic matrices, dependent on the spectral parameters $w_k$ as well as the matrices $A_k$, that asymptotically serve as the deterministic approximation (in weak operator sense) of these resolvent chains. Let us denote the deterministic approximations for the two resolvent chains in (\ref{eqn:resolvent-chains}) respectively by
\begin{equation*}
    \Pi_{12} (A_1)
    \equiv \Pi (w_1, A_1, w_2)
    \qand
    \Pi_{123} (A_1, A_2)
    \equiv \Pi (w_1, A_1, w_2, A_2, w_3),
\end{equation*}
whose definitions are to be clear in (\ref{def:chain-deter-app}) below. As highlighted in \cite{cipolloniOptimalLowerBound2023,cipolloniThermalisationWignerMatrices2022}, these deterministic approximations cannot be directly obtained by substituting the $G_k$'s with the $\Pi_k$'s in (\ref{eqn:resolvent-chains}). Instead, they can be derived from the so-called \emph{recursive Dyson equations}, for which we refer to the derivations in Section \ref{sec:representation} as well as \cite[Appendix D]{cipolloniOptimalLowerBound2023}.

To this end, we introduce the \emph{two-body stability operator}
\begin{equation}
    \opB_{12} [A] \equiv \opB_{w_1, w_2} [A] := A - \Pi_1 \opSd [A] \Pi_2.
    \label{def:operator-B12}
\end{equation}
We also define $\opX_{12} \equiv \opX_{w_1, w_2}$ as the inverse of the operator $V \mapsto V - \opSd [\Pi_1 V \Pi_2]$. In other words,
\begin{equation}
    V = \opX_{12} [A]
    \quad \Longleftrightarrow \quad
    V - \opSd [\Pi_1 V \Pi_2] = A.
    \label{def:operator-X12}
\end{equation}
Note that notations $\opB_{13}, \opB_{23}, \opX_{13}, \opX_{23}$ can be interpreted analogously. Using (\ref{def:opS-diagonal}), one can easily verify that $\angles{\Pi_1 \opSd[A] \Pi_2 V} = \angles{A \opSd[\Pi_2 V \Pi_1]}$ for generic matrices $A$ and $V$. Consequently, the two operators $\opB$ and $\opX$ can be related through
\begin{equation}
    \angles{\opB_{12} [A_1] V_2} = \angles{A_1 \opX_{21}^{-1} [V_2]}
    \qand
    \angles{\opB_{12} [A_1] \opX_{21} [A_2]} = \angles{A_1 A_2},
    \label{eqn:opB-opX-inverse}
\end{equation}
where $A_k$ and $V_k$ are general matrices with compatible dimensions.

\begin{definition}[Deterministic approximations] \label{Def:chain-deter-app}
Given $(M+N) \times (M+N)$ deterministic matrices $A_1, A_2$ and spectral parameters $w_1, w_2, w_3 \in \bbd_0$, we set $V_1 \equiv \opX_{12} [A_1]$, $V_2 \equiv \opX_{23} [A_2]$ and define
\begin{subequations} \label{def:chain-deter-app}
\begin{align}
    \Pi_{12} (A_1)
    & := \opB_{12}^{-1} [\Pi_1 A_1 \Pi_2]
    = \Pi_1 V_1 \Pi_2, 
    \label{def:chain-deter-app-length2} \\
    \Pi_{123} (A_1, A_2)
    & := \opB_{13}^{-1} [ {\Pi_1 V_1 \Pi_2 V_2 \Pi_3} ]
    = \Pi_{13} ({V_1 \Pi_2 V_2}).
    \label{def:chain-deter-app-length3}
\end{align}
\end{subequations}
\end{definition}

Here the identities in (\ref{def:chain-deter-app}) directly follow from (\ref{def:operator-B12}) and (\ref{def:operator-X12}). In fact, we have
\begin{equation*}
    \opB_{12} [\Pi_1 V_1 \Pi_2]
    = \Pi_1 V_1 \Pi_2 - \Pi_1 \opSd [\Pi_1 V_1 \Pi_2] \Pi_2
    = \Pi_1 A_1 \Pi_2.
\end{equation*}
Alternatively, the equivalence between the definitions in (\ref{def:chain-deter-app}) can be obtained by expanding different resolvents in the chain. For more details, we refer to \cite[Appendix D]{cipolloniOptimalLowerBound2023} , where the readers can also find how to generalize Definition \ref{Def:chain-deter-app} to accommodate chains involving more than three resolvents.

We proceed to derive an explicit formula for the operator $\opX$ defined in (\ref{def:operator-X12}), in the context of sample covariance matrices. Given spectral parameters $w_1, w_2 \in \bbd$, we define the \emph{divided difference}
\begin{equation}
    \fkm[z_1, z_2] := \begin{cases}
        \fkm^\prime (z) \equiv \partial_z \fkm (z), & \text{ if } z_1 = z_2 = z, \\
        (\fkm_1 - \fkm_2)/(z_1 - z_2), & \text{ if } z_1 \not= z_2,
    \end{cases}
    \label{def:divided-diff}
\end{equation}
Note that we can unify two cases in the definition (\ref{def:divided-diff}) using the LSD $\varrho$,
\begin{equation*}
    \fkm[z_1, z_2] 
    = \int_{\bbr} \frac{1}{(\lambda - z_1) (\lambda - z_2)} \varrho (\mathrm{d} \lambda).
\end{equation*}
We introduce the deterministic quantities (with $\fkt$ indicating ``top'' and $\fkb$ referring to ``bottom'')
\begin{subequations}
\begin{align}
    \fkt_{12} \equiv \fkt(w_1, w_2) & = \angles{\Gamma_1 \Sigma \Gamma_2 \Sigma}
    = \angles{\Pi_1 \bfSigma_M \Pi_2 \bfSigma_M}
    = \frac{1}{\sqrt{z_1 z_2}} \left ( \frac{1}{\fkm_1 \fkm_2}
    - \frac{1}{\fkm[z_1, z_2]} \right ), \\
    \fkb_{12} \equiv \fkb(w_1, w_2) & = m_1 m_2
    = \angles{\Pi_1 \bfId_N \Pi_2 \bfId_N}
    = \sqrt{z_1 z_2} \fkm_1 \fkm_2.
\end{align} \label{def:fkt-fkb}
\end{subequations}
Now let $A$ be a deterministic $(M+N) \times (M+N)$ matrix. We denote its top-left $M \times M$ diagonal block by $A_M$ and bottom-right $N \times N$ diagonal block by $A_N$, i.e.
\begin{equation*}
    A = \begin{bmatrix} A_M & * \\ * & A_N \end{bmatrix}.
\end{equation*}
Set $V := \opX_{12} [A]$. Utilizing (\ref{def:opS-diagonal}) and (\ref{def:operator-X12}), we have
\begin{equation}
    A = V - \opSd [\Pi_1 V \Pi_2]
    = V - \angles{\Pi_1 V \Pi_2 \bfId_N} \bfSigma_M
    - \angles{\Pi_1 V \Pi_2 \bfSigma_M} \bfId_N.
    \label{tmp:derivation-X12}
\end{equation}
Multiplying both sides of (\ref{tmp:derivation-X12}) by $\Pi_2 \bfSigma_M \Pi_1$ and $\Pi_2 \bfId_N \Pi_1$, respectively, and taking the normalize trace, we obtain
\begin{align*}
    \angles{\Pi_1 A \Pi_2 \bfSigma_M}
    & = \angles{\Pi_1 V \Pi_2 \bfSigma_M}
    - \fkt_{12} \angles{\Pi_1 V \Pi_2 \bfId_N}, \\
    \angles{\Pi_1 A \Pi_2 \bfId_N} 
    & = \angles{\Pi_1 V \Pi_2 \bfId_N} 
    - \fkb_{12} \angles{\Pi_1 V \Pi_2 \bfSigma_M}.
\end{align*}
Solving these equations leads to
\begin{equation*}
    \angles{\Pi_1 V \Pi_2 \bfSigma_M}
    = \frac{\angles{\Pi_1 A \Pi_2 (\bfSigma_M + \fkt_{12} \bfId_N)}}
    {1 - \fkt_{12} \fkb_{12}}
    \qand
    \angles{\Pi_1 V \Pi_2 \bfId_N}
    = \frac{\angles{\Pi_1 A \Pi_2 (\fkb_{12} \bfSigma_M + \bfId_N)}}
    {1 - \fkt_{12} \fkb_{12}}.
\end{equation*}
Plugging this back into (\ref{tmp:derivation-X12}) yields
\begin{align}
\begin{split}
    \opX_{12} [A]
    & = A + \frac{\angles{\Pi_1 A \Pi_2 (\fkb_{12} \bfSigma_M + \bfId_N)}}
    {1 - \fkt_{12} \fkb_{12}} \bfSigma_M
    + \frac{\angles{\Pi_1 A \Pi_2 (\bfSigma_M + \fkt_{12} \bfId_N)}}
    {1 - \fkt_{12} \fkb_{12}} \bfId_N \\
    & = A - \angles{A_N} \bfId_N 
    + \frac{\angles{\Gamma_1 A_M \Gamma_2 \Sigma} + \angles{A_N}}{1 - \fkt_{12} \fkb_{12}} (\fkb_{12} \bfSigma_M + \bfId_N).
\end{split} \label{eqn:explicit-X12}    
\end{align}

Now we would like to conduct a stability analysis of the operator $\opX$ in the local regime so as to motivate the introduction of regular observables in Definition \ref{Def:regular-obs} below. Specifically, our goal is to obtain an asymptotic estimate of $\norm{\opX_{12} [A]}$ for $\norm{A} \lesssim 1$ and $w_1, w_2 \in \bbd$ with $\eta_1 \wedge \eta_2 \ll 1$. In viewing of (\ref{eqn:explicit-X12}) and Lemma \ref{lemma:m-order-estimate}, the most trivial estimate is that
\begin{equation}
    \norm{ \opX_{12} [A] } \lesssim \abs{1 - \fkt_{12} \fkb_{12}}^{-1}
    \asymp \abs{\fkm [z_1, z_2]}.
    \label{tmp:X12-norm-estimate}
\end{equation}
The magnitude of $\fkm[z_1, z_2]$ is significantly influenced by whether the two parameters $z_1$ and $z_2$ (or equivalently, $w_1$ and $w_2$, by our convention) lie in the same half-plane. hence, for $w \in \bbd$, we let
\begin{equation*}
    \fks (w) := \operatorname{sgn} (\Im w) = \operatorname{sgn} (\Im z).
\end{equation*}

The next lemma offers several estimates of the divided difference that are crucial for subsequent analysis. Its proof is presented in Section \ref{subsec:proof-stieltjes-transform}.

\begin{lemma}[Divided difference] \label{lemma:diff-estimate}
\begin{subequations}
By letting $\tau^\prime > 0$ sufficiently small (depending only on $\tau$), we have the following estimates uniformly for $w_1, w_2 \in \bbd (\tau, \tau^\prime)$ with $\fks_1 = \fks_2$,
\begin{align}
    1 \lesssim \abs{\fkm [z_1, z_2]} 
    & \lesssim (d_1 + d_2)^{-1/2} \wedge \abs{z_1-z_2}^{-1/2},
    \label{eqn:diff-estimate-same} \\
    \abs{\fkm [z_1^*, z_2]} 
    & \lesssim ({\Im \fkm_1}/{\Im z_1}) \wedge ({\Im \fkm_2}/{\Im z_2}) \wedge \abs{z_1 - z_2}^{-1} \wedge \abs{\fkm_1 - \fkm_2}^{-1}.
    \label{eqn:diff-estimate-opp}
\end{align}
Moreover, for $w_1, w_2 \in \bbd_0 (\tau, \tau^\prime)$ with $\fks_1 = \fks_2$, we have
\begin{equation}
    \abs{\fkm [z_1^*, z_2]} \gtrsim \abs{\fkm [z_1, z_2]} \vee \abs{z_1^* - z_2}^{-1/2}.
    \label{eqn:diff-compare-same-opp}
\end{equation}
\end{subequations}
\end{lemma}

For what follows, we consistently assume that $\tau^\prime > 0$ is chosen to be sufficiently small, dependent only on $\tau > 0$, such that the estimates in Lemmas \ref{lemma:m-order-estimate} and \ref{lemma:diff-estimate} hold.

The upper bound $({\Im \fkm_1}/{\Im z_1}) \wedge ({\Im \fkm_2}/{\Im z_2})$ in (\ref{eqn:diff-estimate-opp}) is indeed tight, as evident by setting $w_1 = w_2$. Let $w_1, w_2 \in \bbd$ such that $\fks_1 = \fks_2$. Recall (\ref{def:chain-deter-app-length2}) and (\ref{def:operator-X12}). Let us abbreviate 
\begin{equation*}
    \Pi_{1^*2} (A_1) \equiv \Pi (w_1^*, A_1, w_2)
    \qand
    \opX_{1^*2} \equiv \opX_{w_1^*, w_2}.
\end{equation*}
Now, with Lemmas \ref{lemma:m-order-estimate} and \ref{lemma:diff-estimate}, it is apparent that the estimate (\ref{tmp:X12-norm-estimate}) yields
\begin{equation*}
    \norm{ \Pi_{1^*2} [A_1] }
    = \norm{ \Pi_1^* \opX_{1^*2} [A] \Pi_2 } 
    \lesssim \abs{\fkm [z_1^*, z_2]}
    \lesssim (d_1^{1/2} / \eta_1) \wedge (d_2^{1/2} / \eta_2).
\end{equation*}
To gain some intuition from this bound, consider the simplest case where $w_1 = w_2 = w$. In this scenario, this upper bound is essentially of order $\eta^{-1}$ in the bulk, gradually transitioning to $\eta^{-1/2}$ as we approach the edges.

We remark that the trivial bound (\ref{tmp:X12-norm-estimate}) applies to all deterministic matrices $A$ with bounded norms. However, this bound may not be tight for certain specific class of matrices. Actually, we can provide a much better bound than (\ref{tmp:X12-norm-estimate}) if $\angles{\Gamma_1 A_M \Gamma_2 \Sigma} + \angles{A_N}$, the numerator of the fraction in the second line of (\ref{eqn:explicit-X12}), is sufficiently small to counterbalance the singularity induced by the denominator. This prompts us to examine the regular observables defined as follows.

\begin{definition}[Regular observables] \label{Def:regular-obs}
Let $w_1, w_2 \in \bbd_0$. We introduce the control parameter
\begin{equation}
    \beta (w_1, w_2) := \abs{1 - \fkt_{12} \fkb_{12}} \asymp \abs{\fkm [z_1, z_2]}^{-1}.
    \label{def:beta}
\end{equation}
Let $A$ be a deterministic $(M+N) \times (M+N)$ matrix, with its top-left and bottom-right diagonal blocks given by $A_M \in \bbc^{M \times M}$ and $A_N \in \bbc^{N \times N}$, respectively. We say that $A$ is $\{ w_1, w_2 \}$-\emph{regular} if it satisfies $\| A \| \lesssim 1$, $\angles{A_N} = 0$ and
\begin{equation}
    \abs{\angles{\Gamma (\fkw_1) A_M \Gamma (\fkw_2) \Sigma}} \lesssim \beta (\fkw_1, \fkw_2),
    \quad \text{ for } \fkw_1 \in \{ w_1, w_1^* \}, \fkw_2 \in \{ w_2, w_2^*\}.
    \label{eqn:regular-obs-condition}
\end{equation}
\end{definition}

\begin{lemma}[Basic properties of regular observables] \label{lemma:regular-obs-properties}
We have the following statements regarding the regular observables as introduced in Definition \ref{Def:regular-obs}. For statements \ref{item:ppt-reflected-para}-\ref{item:ppt-opX-bound-norm}, we assume that $w_1, w_2 \in \bbd_0$ and $A$ is $\{ w_1, w_2 \}$-regular.
\begin{enumerate}[label=(\roman*)]
    \item \label{item:ppt-reflected-para} $A$ is $\{ \fkw_1, \fkw_2 \}$-regular for $\fkw_1 \in \{ w_1, w_1^* \}, \fkw_2 \in \{ w_2, w_2^*\}$.
    \item \label{item:ppt-A-conjugate} $A^\top, \bar{A}, A^*$ are $\{ w_1, w_2 \}$-regular.
    \item \label{item:ppt-A-negative} $A \bfId^-$ is $\{ w_1, w_2 \}$-regular.
    \item \label{item:ppt-opX-bound-norm} $\| {\opX_{\fkw_1, \fkw_2} (A)} \| \lesssim 1$ for $\fkw_1 \in \{ w_1, w_1^* \}, \fkw_2 \in \{ w_2, w_2^*\}$.    
    \item \label{item:ppt-constraint} Among the four constraints listed in (\ref{eqn:regular-obs-condition}), the two constraints with $\fks (\fkw_1) = \fks (\fkw_2)$ are implied by either of the two constraints with $\fks (\fkw_1) \not= \fks (\fkw_2)$. For example, if $\fks (w_1) = \fks (w_2)$, then 
    \begin{equation*}
        \abs{\angles{\Gamma (w_1^*) A_M \Gamma (w_2) \Sigma}} \lesssim \beta (w_1^*, w_2)
        \quad \Longrightarrow \quad
        \abs{\angles{\Gamma (w_1) A_M \Gamma (w_2) \Sigma}} \lesssim \beta (w_1, w_2).
    \end{equation*}
\end{enumerate}
\end{lemma} 

The proof of Lemma \ref{lemma:regular-obs-properties} is provided in Section \ref{subsec:proof-regularity}. Now, one can compare the bound obtained in Lemma \ref{lemma:regular-obs-properties} \ref{item:ppt-opX-bound-norm} with the trivial bound in (\ref{tmp:X12-norm-estimate}). Note that Definition \ref{Def:regular-obs} is actually formulated for a sequence of matrices $A \equiv A^{(N)}$. Therefore, there should be no ambiguity in the use of the asymptotic notation $\lesssim$. Moreover, in the global regime $\eta_1 \wedge \eta_2 \gtrsim 1$, the condition (\ref{eqn:regular-obs-condition}) plays no role, as in this case we have $\beta (\fkw_1, \fkw_2) \gtrsim 1$ according to Lemma \ref{lemma:diff-estimate}.

Compared with the regular observables given in \cite[Definition 3.1]{cipolloniOptimalLowerBound2023}, our definition is adapted to moderately tolerate the deviation of $\angles{\Gamma (\fkw_1) A_M \Gamma (\fkw_2) \Sigma}$ from zero. This adaption offers greater flexibility and accommodates the edge regime, thereby significantly enhancing our proof. 

The next lemma provides bounds on the deterministic equivalents (\ref{def:chain-deter-app}) when the matrices $A_k$ are regular. It is a straightforward corollary of (\ref{eqn:explicit-X12}) and Lemma \ref{lemma:diff-estimate}.

\begin{lemma}[Bounds on the deterministic approximations] \label{lemma:opD-bound-regular}
Let $w_1, w_2, w_3 \in \bbd_0$ and $A_1, A_2$ be deterministic matrices. If, for each $k=1,2$, the matrix $A_k$ is $\{ w_k, w_{k+1} \}$-regular as per Definition \ref{Def:regular-obs}, then
\begin{equation}
    \| \Pi_{12} (A_1) \| \lesssim 1
    \qand
    \| \Pi_{123} (A_1, A_2) \| \lesssim 1/\eta.
\end{equation}
\end{lemma}

The multi-resolvent local laws in this article are formulated for resolvent chains with regular matrices. For such resolvent chains, not only does the magnitude of the deterministic approximation decrease significantly (as depicted in Lemma \ref{lemma:opD-bound-regular}), but the corresponding approximation error does as well. Here, the regularity of a matrix in a resolvent chain is conventionally understood as its regularity w.r.t. its two surrounding parameters. For ease of future reference, let us clarify this convention in the following definition.

\begin{definition}[Regular observables in resolvent chains] \label{Def:regular-obs-in-chain}
Consider any of the two expressions
\begin{equation*}
    \angles{G_1 A_1 \cdots G_k A_k}
    \qand
    \angles{\fku, G_1 A_1 \cdots G_k A_k G_{k+1} \fkv}
\end{equation*}
for some fixed length $k \in \bbn$, where $w_1, \cdots, w_{k+1} \in \bbd_0$ and $A_1, \cdots, A_k$ are deterministic matrices. For any $\ell \in \llbracket k \rrbracket$, we call $A_\ell$ regular (w.r.t. its surrounding parameters) if it is $\{ w_{\ell}, w_{\ell+1} \}$-regular as per Definition \ref{Def:regular-obs}. In case of the first expression above (i.e. in the averaged form), the indices are understood cyclically modulo $k$ (i.e. $w_{k+1} \equiv w_1$).
\end{definition}

We are now ready to present our main technical result of this article. The multi-resolvent local laws developed here are applicable to spectral parameters within the domain
\begin{equation}
    \bbd_\infty \equiv \bbd_\infty (\varepsilon, \tau) 
    := \{ w \in \bbc : \Re w \in \operatorname{supp} \rho \cap \bbr_+ 
    \text{ and }
    N^{-1+\varepsilon} \leq \abs{\Im w} \leq (2 \tau)^{-1} \}.
    \label{def:bbd-infty}
\end{equation}

\begin{theorem}[Local laws with regular matrices] \label{thm:local-law-regular}
Fix arbitrary (small) $\varepsilon, \tau > 0$. Under Assumptions \ref{assump:high-dimension}-\ref{assump:bulk-regularity}, we have the following local laws uniformly in $w_1, w_2, w_3 \in \bbd_{\infty} (\varepsilon, \tau)$, deterministic vectors $\fku, \fkv \in \bbc^{M+N}$ with $\norm{\fku}, \norm{\fkv} \lesssim 1$, and deterministic matrices $A_1$, $A_2$ that are regular w.r.t. their surrounding parameters.
\begin{subequations} \label{bound:multi-local-law-regular}
\begin{enumerate}[label=(\roman*)]
    \item Averaged law with one regular matrix:
    \begin{equation}
        \abs{\angles{G_1 A_1} - \angles{\Pi_1 A_1}} 
        \prec \frac{1}{N \eta^{1/2}}.
        \label{bound:ave-law-one-regular}
    \end{equation}
    \item Isotropic law with one regular matrix:
    \begin{equation}
        \abs{\angles{\fku, G_1 A_1 G_2 \fkv} - \angles{\fku, \Pi (w_1, A_1, w_2) \fkv}} 
        \prec \frac{1}{\sqrt{N \eta^2}}.
        \label{bound:iso-law-one-regular}
    \end{equation}
    \item Averaged law with two regular matrices:
    \begin{equation}
        \abs{\angles{G_1 A_1 G_2 A_2} - \angles{\Pi (w_1, A_1, w_2) A_2}} 
        \prec (N \eta)^{1/2} \cdot \frac{1}{N \eta}.
        \label{bound:ave-law-two-regular}
    \end{equation}
    \item Isotropic law with two regular matrices:
    \begin{equation}
        \abs{\angles{\fku, G_1 A_1 G_2 A_2 G_3 \fkv} - \angles{\fku, \Pi (w_1, A_1, w_2, A_2, w_3) \fkv}} 
        \prec (N \eta)^{1/4} \cdot \frac{1}{\sqrt{N \eta^3}}.
        \label{bound:iso-law-two-regular}
    \end{equation}
\end{enumerate}
\end{subequations}
\end{theorem}

The key estimate driving our theoretical result on eigenvector overlaps is (\ref{bound:ave-law-two-regular}). The other three estimates in Theorem \ref{thm:local-law-regular} are byproducts obtained during the derivation of this primary estimate. It is noteworthy that all the estimates in (\ref{bound:multi-local-law-regular}) provide nontrivial error bounds, considering the typical order of the deterministic approximations outlined in Lemma \ref{lemma:opD-bound-regular}. By comparing (\ref{bound:ave-law-one-regular}) with the average estimate in (\ref{eqn:single-resolvent-local-law}), we observe that the fluctuation of $G - \Pi$ along a regular direction is significantly smaller than a generic direction that may not necessarily be regular. Finally, the prefactors $(N \eta)^{1/2}$ and $(N \eta)^{1/4}$ in the error bounds of (\ref{bound:ave-law-two-regular}) and (\ref{bound:iso-law-two-regular}) suggest that these rates may not be optimal. However, we refrain from optimizing over these rates as (\ref{bound:ave-law-two-regular}) is already sufficient for our purpose. 

\subsection{Proof of the main results} 
\label{subsec:proof-main-result}

In this section, we provide the proof of our main results, Theorems \ref{thm:eigenvector-overlaps} and \ref{thm:shrinkage}, utilizing the averaged law with two regular matrices (\ref{bound:ave-law-two-regular}). Prior to that, let us introduce the notion of \emph{one-point regularization}, which projects an arbitrary matrix with bounded norm into the class of regular matrices. It is worth noting that there are multiple ways to regularize a matrix, and the choice is not unique. Here, we utilize (\ref{def:one-point-regularization}) since it aligns well with our proof.

\begin{lemma}[One-point regularization] \label{lemma:one-point-regularization}
Let $D$ be a $(M+N) \times (M+N)$ deterministic matrix with $\norm{D} \lesssim 1$. Suppose its top-left $M \times M$ and bottom-right $N \times N$ diagonal blocks are given by $D_M$ and $D_N$, respectively. For $w \in \bbd_0$, we define the one-point regularization of $D$ w.r.t. $w$ as
\begin{equation}
    (D)_{w}^{\circ} 
    := D - \frac{\angles{\Im \Gamma (w) D_M}}{\angles{\Im \Gamma (w)}} \bfId_M
    - \angles{D_N} \bfId_N.
    \label{def:one-point-regularization}
\end{equation}
Equivalently, we can write
\begin{equation}
    (D)_{w}^{\circ} = D - \coefOne_w^+ (D) \bfId^+ - \coefOne_w^- (D) \bfId^-,
    \qwhere \coefOne_w^\pm (D) := \frac{1}{2} \left ( \frac{\angles{\Im \Gamma (w) D_M}}{\angles{\Im \Gamma (w)}} \pm \angles{D_N} \right ).
    \label{def:coefficients-one-point}
\end{equation}
Suppose $w_1, w_2 \in \bbd_0$ such that $w \in \{ w_1, w_2 \}$. Then, $(D)^\circ_w$ is $\{ w_1, w_2 \}$-regular.
\end{lemma}

Without loss of generality, in the proofs of Theorems \ref{thm:eigenvector-overlaps} and \ref{thm:shrinkage}, we assume that $K = 1$, meaning that $\operatorname{supp} \varrho$, the support of the LSD $\varrho$, contains only one bulk component. In this case, the classical number of eigenvalues in this component is given by $N_1 = M \wedge N$. The extension to the case of multiple bulk components is straightforward and thus omitted.

\begin{proof}[Proof of Theorem \ref{thm:eigenvector-overlaps}]
Fix arbitrary (small) $\varepsilon > 0$ and (large) $L > 0$. Let $i,j \in \llbracket M \wedge N \rrbracket$. Recalling the classical locations defined in (\ref{def:classical-locations}), we introduce the spectral parameters
\begin{equation*}
    w_{i} := \gamma_i + \mathrm{i} \eta_i
    \qand
    w_{j} := \gamma_j + \mathrm{i} \eta_j,
\end{equation*}
where the scales are specified as
\begin{equation*}
    \eta_i := N^{-2/3+\varepsilon} [i \wedge (N_1+1-i)]^{-1/3}
    \qand
    \eta_j := N^{-2/3+\varepsilon} [j \wedge (N_1+1-j)]^{-1/3}.
\end{equation*}
By the rigidity estimates in Proposition \ref{prop:rigidity}, we have with probability at least $1 - N^{-L}$,
\begin{equation}
    \abs{s_i - \gamma_i} \leq \eta_i
    \qand
    \abs{s_j - \gamma_j} \leq \eta_j.
    \label{tmp:deviation-singular-value}
\end{equation}
Given any deterministic $D$ with $\norm{D} \lesssim 1$, let $A := (D)^\circ_{w_i}$ be its one-point regularization w.r.t. $w_i$ as defined in (\ref{def:one-point-regularization}). Thanks to Lemma \ref{lemma:one-point-regularization} and Lemma \ref{lemma:regular-obs-properties} \ref{item:ppt-reflected-para}, \ref{item:ppt-A-conjugate}, we know that both $A$ and $A^*$ are $\{ \fkw_i, \fkw_j \}$-regular, where $\fkw_i \in \{ w_i, w_i^* \}$ and $\fkw_j \in \{ w_j, w_j^* \}$. Therefore, by utilizing the averaged law with two regular matrices (\ref{bound:ave-law-two-regular}) in Theorem \ref{thm:local-law-regular}, along with the bound $\norm{\Pi(\fkw_i) A \Pi(\fkw_j)} \lesssim 1$ from Lemma \ref{lemma:opD-bound-regular}, we have
\begin{equation*}
    \angles{G (\fkw_i) A G (\fkw_j) A^* } = \oprec (1),
    \quad \text{ for } \fkw_i \in \{ w_i, w_i^* \}, \fkw_j \in \{ w_j, w_j^*\}.
\end{equation*}
Since $G (w) - G (w^*) = 2 \mathrm{i} \Im G (w)$, the above bound readily implies, with probability at least $1 - N^{-L}$,
\begin{equation*}
    \abs{\angles{\Im G (w_i) A \Im G (w_j) A^* }} \leq N^\varepsilon.
\end{equation*}
On the other hand, whenever the rigidity estimate (\ref{tmp:deviation-singular-value}) holds, we have
\begin{equation*}
    \angles{\Im G (w_i) A \Im G (w_j) A^*}
    = \frac{1}{N} \sum_{i^\prime, j^\prime \in \bbj}
    \frac{\eta_j \eta_j \abs{\angles{\bfxi_{i^\prime}, A \bfxi_{j^\prime}}}^2}
    {[(s_{i^\prime} - \gamma_i)^2 + \eta_i^2] [(s_{j^\prime} - \gamma_j)^2 + \eta_j^2]} 
    \geq \frac{\abs{\angles{\bfxi_{i}, A \bfxi_{j}}}^2}{4 N \eta_i \eta_j}.
\end{equation*}
In other words, we have established, with probability at least $1 - 2 N^{-L}$,
\begin{align*}
    \abs{\angles{\bfxi_{i}, A \bfxi_{j}}}
    \leq (4 N^{1 + \varepsilon} \eta_i \eta_j)^{1/2}
    = 2 N^{-1/6+3\varepsilon/2} [i \wedge (N_1+1-i)]^{-1/6} [j \wedge (N_1+1-j)]^{-1/6}.
    \label{tmp:overlap-with-A}    
\end{align*}
Now, it remains to reduce the l.h.s. of the preceding inequality to the l.h.s. of (\ref{bound:overlap-xi-linearization}). Utilizing the definition of one-point regularization in (\ref{def:coefficients-one-point}), along with $\angles{\bfxi_{i}, \bfxi_{j}} = \delta_{ij}$ and $\angles{\bfxi_{i}, \bfId^- \bfxi_{j}} = 0$, we have
\begin{equation*}
    \angles{\bfxi_{i}, A \bfxi_{j}}
    = \angles{\bfxi_{i}, (D)^\circ_{w_i} \bfxi_{j}}
    = \angles{\bfxi_{i}, D \bfxi_{j}} - \delta_{ij} \coefOne_{w_i}^+ (D).
\end{equation*}
Here for the coefficient $\coefOne_{w_i}^+ (D)$ we can employ the equations (\ref{eqn:Gamma-perturbation}) and (\ref{eqn:Gamma-trace}) below to get
\begin{align*}
    \coefOne_{w_i}^+ (D) 
    & = \frac{1}{2} \left ( \frac{\angles{\Im \Gamma (w_i) D_M}}{\angles{\Im \Gamma (w_i)}} + \angles{D_N} \right ) 
    = \frac{1}{2} \left ( \frac{\angles{\Im \Gamma (w_i) D_M}}{\Im m (w_i)} + \angles{D_N} \right ) 
    + O \big ( {\Im w_i/\Im m(w_i)} \big ) \\
    & = \frac{1}{2} \left ( \frac{\angles{\Im \Gamma (\gamma_i) D_M}}{\Im m (\gamma_i)} + \angles{D_N} \right ) 
    + O \big ( {\abs{w_i - \gamma_i} + \abs{m(w_i) - m(\gamma_i)} + \Im w_i/\Im m(w_i)} \big ) \\ 
    & = \frac{\angles{\Im \Pi (\gamma_i) D}}{2 \Im m (\gamma_i)} + O (\eta_i^{1/2}).
\end{align*}
where we also used the estimates in Lemmas \ref{lemma:m-order-estimate} and \ref{lemma:diff-estimate} in the last step. We mention that although the equations/estimates for $m$ and $\Gamma$ employed here are originally stated for $w$ with $\eta \not= 0$, their validity remains intact as $\Im w \downarrow 0$. Summarizing the three equations above concludes the proof.
\end{proof}

\begin{proof}[Proof of Theorem \ref{thm:shrinkage}]
As mentioned, we assume $K = 1$ to simplify the presentation. Additionally, note that this proof is conducted under the additional Assumption \ref{assump:invertibility}, which implies $N_1 = M$ when $K = 1$. Algorithm \ref{algo:shrinkage} and Theorem \ref{thm:shrinkage} are formulated using the spectral parameters $z_i = \lambda_i + \mathrm{i} \eta_i$ and the Stieltjes transform $\fkg (z) = N^{-1} \operatorname{tr} (X^\top \Sigma X - z)^{-1}$. However, for the sake of notation simplicity, we opt to formulate this proof in terms of $w_i = s_i + \mathrm{i} \eta_i$ and the Stieltjes transform
\begin{equation*}
    g(w) = w \fkg (w^2) = \angles{G(w) \bfId_N}.
\end{equation*}
Importantly, this switch in spectral parameters does not alter the essence of the argument. Moreover, according to the eigenvalue rigidity estimate (\ref{bound:rigidity}), it suffices to work with the high-probability event where the singular values adhere closely to their classical locations,
\begin{equation*}
    \abs{s_i - \gamma_i}
    \leq N^{-2/3+c/2} (\fkn_i)^{-1/3},
    \quad \text{ where } \quad
    \fkn_i \equiv i \wedge (M+1-i).
\end{equation*}
In particular, for a sufficiently large constant $C > 0$, we have $\lambda_i \in [C^{-1}, C]$ uniformly across $i \in \llbracket M \rrbracket$ and $N \in \bbn$ on the aforementioned high-probability event. Consequently, we can, without loss of generality, suppose that $\eta_i = \eta \in [N^{-2/3+c}, N^{-c}]$. 

Now, regarding the entrywise estimate (\ref{bound:entrywise-optimal-algo}), we need to prove
\begin{equation}
    \bigg | {\angles{\bfu_{i}, \Sigma \bfu_{i}} - \frac{1}{\abs{g(s_i + \mathrm{i} \eta)}^2}} \bigg |
    \prec \frac{1}{N^{1/6} (\fkn_{i})^{1/3}} 
    + \frac{\eta}{(\fkn_{i}/N)^{1/3} + \eta^{1/2}} 
    + \frac{1}{N \eta}.
    \label{tmp:elementwise-to-prove}
\end{equation}
By setting $D_1 = \Sigma$ in (\ref{bound:overlap-singular-u-u}), we obtain
\begin{equation*}
    \bigg | {\angles{\bfu_{i}, \Sigma \bfu_{i}} - \frac{1}{\abs{m(\gamma_i)}^2}} \bigg |
    \prec \frac{1}{N^{1/6} (\fkn_{i})^{1/3}},
\end{equation*}
where we utilized $\angles{\Gamma (\gamma_i) \Sigma} = - \gamma_i - 1/m(\gamma_i)$ computed from (\ref{eqn:self-consistent-w}). This corresponds to the first fraction on the r.h.s. of (\ref{tmp:elementwise-to-prove}). Next, we leverage the estimates from Lemmas \ref{lemma:m-order-estimate} and \ref{lemma:diff-estimate} to obtain
\begin{equation*}
    \bigg | {\frac{1}{\abs{m(s_i + \mathrm{i} \eta)}^2} - \frac{1}{\abs{m(\gamma_i)}^2}} \bigg | 
    \lesssim \abs{m(s_i + \mathrm{i} \eta) - m(\gamma_i)}
    \lesssim \frac{\abs{s_i - \gamma_i} + \eta}{\operatorname{dist} (s_i, \partial \operatorname{supp} \rho)^{1/2} + \eta^{1/2}}.
\end{equation*}
Here, using the triangle inequality, we have
\begin{equation*}
    \operatorname{dist} (s_i, \partial \operatorname{supp} \rho) 
    \geq \operatorname{dist} (\gamma_i, \partial \operatorname{supp} \rho) - \abs{s_i - \gamma_i}.
\end{equation*}
On the one hand, by incorporating the definition of the classical locations $\gamma_i$ along with the square root behavior of the density of the LSD $\rho(E)$ around its edges\footnote{See e.g. \cite[Theorem 3.1]{baoUniversalityLargestEigenvalue2015}. Alternatively, we can derive this square root behavior by of $\rho (E)$ by letting $\eta \downarrow 0$ in the estimate of $\Im m (E + \mathrm{i} \eta)$ as given by (\ref{eqn:m-order-estimate}).}, it is not difficult to verify that
\begin{equation*}
    \operatorname{dist} (\gamma_i, \partial \operatorname{supp} \rho) 
    \asymp (\fkn_{i}/N)^{2/3}.
\end{equation*}
On the other hand, by eigenvalue rigidity and the range of $\eta$ we have
\begin{equation*}
    \abs{s_i - \gamma_i} \leq N^{-2/3+c/2} \ll N^{-2/3+c} \leq \eta.
\end{equation*}
Combining these estimates together yields the second fraction on the r.h.s. of (\ref{tmp:elementwise-to-prove}),
\begin{equation}
    \bigg | {\frac{1}{\abs{m(s_i + \mathrm{i} \eta)}^2} - \frac{1}{\abs{m(\gamma_i)}^2}} \bigg | 
    \prec \frac{\eta}{(\fkn_{i}/N)^{1/3} + \eta^{1/2}}.
\end{equation}
Finally, we employ the single resolvent local law (\ref{eqn:single-resolvent-local-law}) to obtain 
\begin{equation*}
    \frac{1}{\abs{g(s_i + \mathrm{i} \eta)}^2} - \frac{1}{\abs{m(s_i + \mathrm{i} \eta)}^2}
    \lesssim \sup_{E} \abs{g(E + \mathrm{i} \eta) - m(E + \mathrm{i} \eta)}
    \prec \frac{1}{N \eta},
\end{equation*}
where the supremum is taken over $E$ with $E + \mathrm{i} \eta \in \bbd$. The estimate above corresponds to the last fraction on the r.h.s. of (\ref{tmp:elementwise-to-prove}). By combining these estimates, we conclude the proof of (\ref{tmp:elementwise-to-prove}).

To prove (\ref{bound:loss-optimal-algo}), let us recall that the partial order $\prec$ is preserved under summation over a set whose cardinality bounded by powers of $N$,
\begin{equation*}
    \frac{1}{M} \sum_{i=1}^{M} \bigg | {\angles{\bfu_{i}, \Sigma \bfu_{i}} - \frac{1}{\abs{g(s_i + \mathrm{i} \eta)}^2}} \bigg |^2
    \prec \frac{1}{M} \sum_{i=1}^{M} \bigg [ { \frac{1}{N^{1/3} (\fkn_{i})^{2/3}} 
    + \frac{\eta^2}{(\fkn_{i}/N)^{2/3} + \eta} 
    + \frac{1}{N^2 \eta^2} } \bigg ].
\end{equation*}
Here, the summation over the first fraction can be dominated by
\begin{equation*}
    \frac{1}{M^2} \sum_{i=1}^{M} \frac{1}{(\fkn_{i}/M)^{2/3}} 
    \lesssim \frac{1}{M^2} \sum_{i=1}^{M} \Big ( { \frac{i}{M} \wedge \frac{M+1-i}{M} } \Big )^{-2/3}
    \lesssim \frac{1}{M} \int_{0}^1 [t \wedge (1-t)]^{-2/3} \mathrm{d} t
    \lesssim \frac{1}{N}.
\end{equation*}
Similarly, for the summation over the second fraction, we can control it by
\begin{equation*}
    \frac{\eta^2}{M} \sum_{i=1}^{M} \frac{1}{(\fkn_{i}/N)^{2/3}} 
    \lesssim \frac{\eta^2}{M} \sum_{i=1}^{M} \Big ( { \frac{i}{M} \wedge \frac{M+1-i}{M} } \Big )^{-2/3}
    \lesssim \eta^2.
\end{equation*}
This completes the proof of (\ref{bound:loss-optimal-algo}) for the Frobenius loss. Regarding the adaption to the inverse Frobenius loss, the argument remains essentially the same, except that we need to employ the estimate for eigenvector overlaps of the form $\angles{\bfu_i, \Sigma^{-1} \bfu_i}$. To derive the deterministic counterpart in (\ref{bound:overlap-singular-u-u}) for $D_1 = \Sigma^{-1}$, we utilize (\ref{eqn:self-consistent-w}) again to obtain
\begin{equation*}
    \angles{\Gamma (\gamma_i) \Sigma^{-1}} 
    = - \frac{\angles{\Sigma^{-1}}}{\gamma_i} + \frac{m(\gamma_i)^2}{\gamma_i} 
    + \bigg ( {1 - \frac{M}{N}} \bigg ) \frac{m(\gamma_i)}{\gamma_i^2},
\end{equation*}
which implies that
\begin{equation*}
    \frac{\angles{\Im \Gamma (\gamma_i) \Sigma^{-1}}}{\Im m(\gamma_i)}
    = \frac{2 \Re [m(\gamma_i)]}{\gamma_i} 
    + \bigg ( {1 - \frac{M}{N}} \bigg ) \frac{1}{\gamma_i^2}
    = 2 \Re [\fkm (\gamma_i^2)]
    + \bigg ( {1 - \frac{M}{N}} \bigg ) \frac{1}{\gamma_i^2}.
\end{equation*}
Note that the $\gamma_i^2$'s serve as the classical location for the eigenvalues $\lambda_i$. Hence, the formula above illustrates how the transitional shrinkages $\lambda_i^{\finv \star}$ are obtained.
\end{proof}

\section{Self-improving inequalities} \label{sec:self-improving-inequalities}

For the subsequent sections, we delve into the proof of Theorem \ref{thm:local-law-regular}. Our methodology for proving the multi-resolvent local laws draws inspiration from \cite{cipolloniOptimalLowerBound2023}. The overarching strategy hinges on deriving a system of self-improving inequalities (Proposition \ref{prop:self-improving}) that encapsulate the errors in approximating the resolvent chains with their deterministic counterparts. The term ``self-improving'' implies that these inequalities can be iterated a certain number of times (see Lemma \ref{lemma:iteration}) to attain the desired error bound. The primary distinction in the analysis of this article, compared to that in \cite{cipolloniOptimalLowerBound2023}, lies in the inclusion of the intricate edge regime, where the Stieltjes transform $m$ loses its Lipschitz continuity. Our analysis heavily rely on the $1/2$-H\"{o}lder continuity of $m$ near the spectral edges, as demonstrated by Lemma \ref{lemma:diff-estimate}. This reliance is evident in the proofs of Lemmas \ref{lemma:coef-V1-Pi2} and \ref{lemma:coef-Sigma-circ} concerning the coefficients introduced by regularization of certain matrices. It is worth noting that while the arguments in this article are complicated by the inclusion of the edge regime, the overall proof remains concise, thanks to our flexible definition of regular observables, which circumvents redundant steps of regularization that contribute less to the overall argument.

We organize the proof as follows. Before delving into details, Section \ref{subsec:prelim-proof} introduces several definitions and lemmas that are utilized throughout the proof. In Section \ref{subsec:self-improving}, we present Proposition \ref{prop:self-improving}, which establishes the system of self-improving inequalities as mentioned above. Leveraging this result, we deduce the multi-resolvent local laws with regular matrices through iteration. The subsequent sections are dedicated to deriving Proposition \ref{prop:self-improving}. A key element of this derivation is Proposition \ref{prop:repre-full-underlined}, which represents errors in approximating the resolvent chains with their deterministic counterparts as second-order renormalization, up to negligible terms. The proof of this result, which involves considerable effort, is addressed in Section \ref{sec:representation}. Section \ref{subsec:cumulant-expansion} presents the derivation of the self-improving inequalities by incorporating this representation with the cumulant expansion formula (\ref{eqn:cumulant-expansion}). The contributions from second and higher cumulants in this expansion are quantified in Section \ref{subsec:gaussian-contribution} and \ref{subsec:high-order-contribution}, respectively. Finally, most proofs of the technical lemmas are consolidated in Section \ref{sec:tech-lemmas} to maintain the coherence of the main argument.

\subsection{Preliminaries} \label{subsec:prelim-proof}

To streamline the notations, we employ $\Upsilon$ to represent the errors incurred in approximating the resolvent chains with their deterministic counterparts,
\begin{subequations}
\begin{align}
    \Upsilon_1 & \equiv \Upsilon(w_1) := G(w_1) - \Pi(w_1), \\
    \Upsilon_{12} (A_1) & \equiv \Upsilon (w_1, A_1, w_2) := G_1 A_1 G_2 - \Pi (w_1, A_1, w_2), \\
    \Upsilon_{123} (A_1, A_2) & \equiv \Upsilon (w_1, A_1, w_2, A_2, w_3) := G_1 A_1 G_2 A_2 G_3 - \Pi  (w_1, A_1, w_2, A_2, w_3).
\end{align}    
\end{subequations}
For $k=1,2$, we introduce the normalized differences
\begin{subequations} \label{def:normalized-difference}
\begin{align}
    \ParaAVE_k (w_1, A_1, \cdots, w_k, A_k) 
    & := N \eta^{k/2} \abs{\angles{\Upsilon (w_1, A_1, \cdots, A_{k-1}, w_k) A_k}}, \\
    \ParaISO_k (w_1, A_1, \cdots, A_k, w_{k+1}, \fku, \fkv) 
    & := \sqrt{N \eta^{k+1}} \abs{\angles{\fku, \Upsilon (w_1, A_1, \cdots, A_k, w_{k+1}) \fkv}},
\end{align}
\end{subequations}
where we used the short hand notation $\eta := \min_{j} \eta_{j} = \min_{j} \abs{\Im w_{j}}$. We also introduce the following control parameters concerning the uncentralized forms of resolvent chains containing $k$ regular matrices, $1 \leq k \leq 4$,
\begin{subequations}
\begin{align}
    \BoundAVE_k (w_1, A_1, \cdots, w_k, A_k) 
    & := \abs{\angles{G_1 A_1 \cdots G_k A_k}}, \\
    \BoundISO_k (w_1, A_1, \cdots, A_k, w_{k+1}, \fku, \fkv) 
    & := \abs{\angles{\fku, G_1 A_1 \cdots G_k A_k G_{k+1} \fkv}}.
\end{align}
\end{subequations}

Recall the definitions of the spectral domains $\bbd_0 (\tau, \tau^\prime)$ and $\bbd_\infty (\varepsilon, \tau)$ in (\ref{def:bbd-0}) and (\ref{def:bbd-infty}), respectively. Due to our utilization of Cauchy's integral formula (\ref{eqn:contour-integral-product}) below, when handling spectral parameters within a specific domain, we might have to rely on estimates for those parameters from a slightly broader domain. As our estimates are refined through iteration, we need to introduce a sequence of nested domains that progressively approach our target domain $\mathbb{D}_\infty$.

\begin{lemma}[Nested spectral domains] \label{lemma:nested-domains}
Fix some (small) $\varepsilon, \tau > 0$ and let $\tau^\prime > 0$ be some sufficiently small constant depending only on $\tau$. Given arbitrary (large) $L \in \bbn$, there exists a sequence of nested spectral domains $\bbd_\ell \equiv \bbd_\ell (\varepsilon, \tau, \tau^\prime)$, $\ell \in \llbracket L \rrbracket$, such that
\begin{equation*}
    \bbd_0 (\tau, \tau^\prime) \supset \bbd_1 (\varepsilon, \tau, \tau^\prime) \supset \cdots \supset \bbd_L (\varepsilon, \tau, \tau^\prime) \supset \bbd_\infty (\varepsilon, \tau),
\end{equation*}
and the following statements hold for every $\ell \in \llbracket L \rrbracket$,
\begin{enumerate}[label=(\roman*)]
    \item $\bbd_\ell$ is formed by the union of $2K$ disjoint convex polygons.
    \item $\abs{\Im w} \geq \ell N^{-1+\varepsilon/2}$ whenever $w \in \bbd_{\ell}$.
    \item \label{item:dist-boundary-domains} $\operatorname{dist} (w, \partial \bbd_{\ell}) \geq c_{\ell} \abs{\eta}$ for all $w \in \bbd_{\ell + 1}$, where the domain $\bbd_{L+1}$ is understood as $\bbd_{\infty}$ and the constant $c_\ell \equiv c_{\ell} (\varepsilon, \tau) > 0$ is $N$-independent.
\end{enumerate}
\end{lemma}

There are various approaches to construct the spectral domains $\bbd_\ell$ to fulfill the conditions outlined in Lemma \ref{lemma:nested-domains}. However, the specific construction of these domains is not essential for subsequent discussions. Therefore, we refrain from providing a formal proof for Lemma \ref{lemma:nested-domains}. Nevertheless, for visual clarity and reader comprehension, we depict one possible construction of the $\bbd_\ell$'s in Figure \ref{fig:spectral-domains} along with a brief description of their formation. We also remark that the motivation behind imposing requirement \ref{item:dist-boundary-domains} is clarified in (\ref{bound:integral-abs-resolvent}). For what follows, we fix some sufficiently large $L \equiv L(\varepsilon, \tau) \in \bbn$ and consistently refer to $\bbd_\ell$, $\ell \in \llbracket L \rrbracket$, as a sequence of spectral domains satisfying the requirements specified in Lemma \ref{lemma:nested-domains}. 

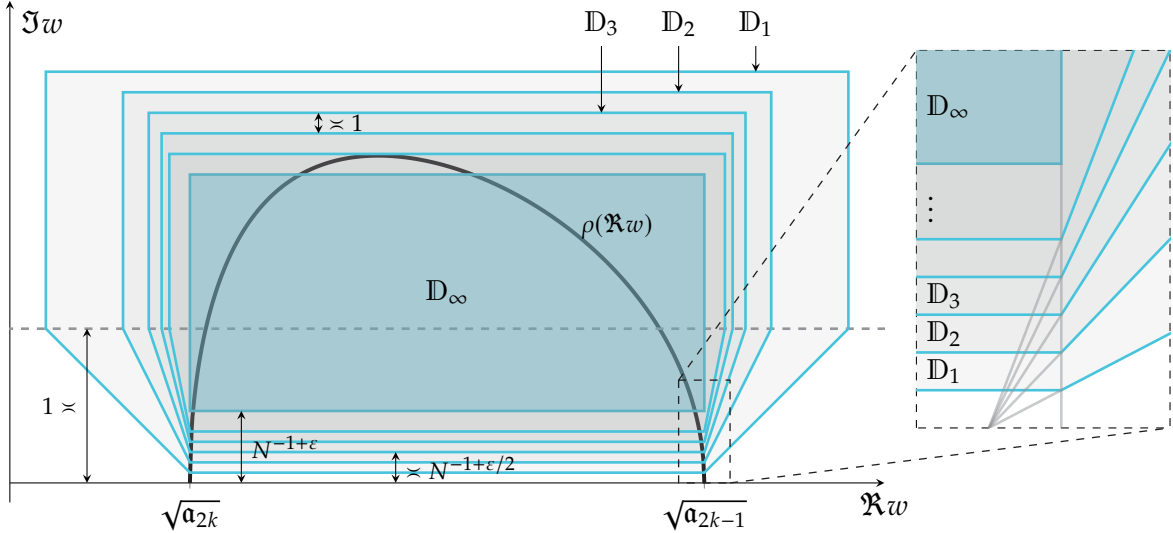
\begin{figure}[htbp]
    \centering
    \begin{tikzpicture}
        \begin{axis}[
            xlabel={$\Re w$},
            ylabel={$\Im w$},
            width=0.8\linewidth,
            height=0.5\linewidth,
            xmin=0.15,
            xmax=1.85,
            ymin=0,
            ymax=0.9,
            axis lines=middle,
            xtick={0.5,1.5},
            xticklabels={$\sqrt{\fka_{2k}}$,$\sqrt{\fka_{2k-1}}$},
            ytick=\empty,
            axis equal,
            every axis x label/.style={
                at={(ticklabel* cs:1)},
                anchor=north,
            },
            name=axMain
        ]
            \draw[line width=1.5pt,color=Black,smooth] (0.5,0) -- plot[domain=0.5:1.5,samples=301] ({\x},{density(\x)}) -- (1.5,0);
            \draw[line width=1pt,color=Gray,dashed] (-1,0.3) -- (5,0.3);
            \draw[line width=1pt,color=SkyBlue,fill=SkyBlue,fill opacity=0.4] (0.5, 0.14) rectangle (1.5, 0.6);
            \drawPolygon{1}{5}
            \drawPolygon{2}{4}
            \drawPolygon{3}{3}
            \drawPolygon{4}{2}
            \drawPolygon{5}{1}
            \draw[stealth-stealth] (0.6,0) -- (0.6,0.14) node[midway,right] {\footnotesize $N^{-1+\varepsilon}$};
            \draw[stealth-stealth] (0.9,0) -- (0.9,0.06) node[midway,right] {\footnotesize $\asymp N^{-1+\varepsilon/2}$};
            \draw[stealth-stealth] (0.3,0) -- (0.3,0.3) node[midway,left] {\footnotesize $1 \asymp$};
            \draw[stealth-stealth] (0.75,0.68) -- (0.75,0.72) node[midway,right] {\footnotesize $\asymp 1$};
            \node[above] (D1) at (Dm1 |- 0,0.85) {$\bbd_1$};
            \draw[-stealth] (D1) -- (Dm1);
            \node[above] (D2) at (Dm2 |- 0,0.85) {$\bbd_2$};
            \draw[-stealth] (D2) -- (Dm2);
            \node[above] (D3) at (Dm3 |- 0,0.85) {$\bbd_3$};
            \draw[-stealth] (D3) -- (Dm3);
            \node (Dinf) at (1, 0.37) {$\bbd_\infty$};
            \node (dens) at (1.33, 0.5) {\footnotesize $\rho(\Re w)$};
    
            \coordinate (vt1) at (1.45,0.2);
            \coordinate (vt2) at (1.55,0);
            \draw[line width=0.5pt,color=Black,dashed] (vt1) rectangle (vt2);
        \end{axis}
    
        \begin{axis}[
            xlabel={},
            ylabel={},
            width=0.3\linewidth,
            height=0.4\linewidth,
            xmin=1.46,
            xmax=1.53,
            ymin=0,
            ymax=0.2,
            axis lines=box,
            axis line style=dashed,
            xtick=\empty,
            ytick=\empty,
            at={($(axMain.south west)+(0.73*\linewidth,0.06*\linewidth)$)},
            name=axSub
        ]
            \draw[line width=1pt,color=SkyBlue,fill=SkyBlue,fill opacity=0.4] (0.5, 0.14) rectangle (1.5, 0.6);
            \drawPolygon{1}{5}
            \drawPolygon{2}{4}
            \drawPolygon{3}{3}
            \drawPolygon{4}{2}
            \drawPolygon{5}{1}
            \coordinate (Fc) at (1.48,0) {};
            \draw[line width=1pt,color=Gray,draw opacity=0.5] (Br1) -- (Fc);
            \draw[line width=1pt,color=Gray,draw opacity=0.5] (Br2) -- (Fc);
            \draw[line width=1pt,color=Gray,draw opacity=0.5] (Br3) -- (Fc);
            \draw[line width=1pt,color=Gray,draw opacity=0.5] (Br4) -- (Fc);
            \draw[line width=1pt,color=Gray,draw opacity=0.5] (Br5) -- (Fc);
            \draw[line width=1pt,color=Gray,draw opacity=0.5] (1.5,0.14) -- (1.5,0);
            \node[right] (D1sub) at (1.46,0.03) {$\bbd_1$};
            \node[right] (D2sub) at (1.46,0.05) {$\bbd_2$};
            \node[right] (D3sub) at (1.46,0.07) {$\bbd_3$};
            \node[right] (D1sub) at (1.46,0.12) {$\vdots$};
            \node[right] (D1sub) at (1.46,0.17) {$\bbd_\infty$};
        \end{axis}
        
        \draw[line width=0.5pt,color=Black,dashed] (vt1) -- (axSub.north west);
        \draw[line width=0.5pt,color=Black,dashed] (vt2) -- (axSub.south east);
    \end{tikzpicture}
    \caption{Depicted is one construction of the spectral domains $\bbd_\ell$ satisfying the conditions listed in Lemma \ref{lemma:nested-domains}. The component associated with the interval $[\sqrt{\fka_{2k}}, \sqrt{\fka_{2k-1}}]$ in the upper half-plane is displayed. The bottom two vertices of $\bbd_{\ell}$ are located at $(\sqrt{\fka_{2k}}, \ell N^{-1+\varepsilon/2})$ and $(\sqrt{\fka_{2k-1}}, \ell N^{-1+\varepsilon/2})$. The two vertices on the grey dashed line have a constant-order imaginary part, while their real parts are selected such that the slopes of the two oblique edges are $\pm \ell$. Equivalently, the grey lines in the subplot intersect at $(\sqrt{\fka_{2k-1}} - N^{-1+\varepsilon/2}, 0)$. Since $\ell \lesssim 1$, we have $\arctan (\ell+1) - \arctan (\ell) \gtrsim 1$, which can be used to verify the requirement \ref{item:dist-boundary-domains}.}
    \label{fig:spectral-domains}
\end{figure}

The subsequent two identities are utilized to simplify the analysis of resolvent chains involving $\bfId^+$ by reducing them to shorter ones.

\begin{lemma}
Let $w_1, w_2 \in \bbc \backslash \bbr$ with $w_1 \not= w_2$. We have the resolvent identity
\begin{equation}
    G(w_1) G(w_2) = \frac{G(w_1) - G(w_2)}{w_1 - w_2}.
    \label{eqn:resolvent-difference-product}
\end{equation}
Moreover, by Cauchy's integral formula, for $w_1, w_2 \in \bbd_{\ell + 1}$ with $\fks_1 = \fks_2$ we have 
\begin{equation}
    G(w_1) G(w_2) = \frac{1}{2 \pi \mathrm{i}} \oint_{\partial \bbd_{\ell}} \frac{G(\zeta)}{(\zeta - w_1) (\zeta - w_2)} \mathrm{d} \zeta,
    \label{eqn:contour-integral-product}
\end{equation}
where the boundary is oriented counter-clockwise.
\end{lemma}

Next, let us make some conventions for statements regarding stochastic domination. Consider a family of nonnegative random variables, $\mathcal{Y} \equiv \mathcal{Y}^{(N)} (\{ w_{j} \}, \{ A_{j} \}, \{ D_{j} \}, \{ \fku_{j} \})$. Throughout the proof, we reserve the letter $A$ to denote regular matrices, while use $D$ to refer to generic deterministic matrices. Let $y \equiv y^{(N)} (\eta)$ be some deterministic control parameter depending only on $N$ and $\eta = \min_j \abs{\Im w_{j}}$. For the sake of brevity, we consistently use streamlined statements like
\begin{equation*}
    \mathcal{Y} \prec y,
    \quad \text{ uniformly in } \bbd_\ell,
\end{equation*}
to indicate that $\mathcal{Y}$ is stochastically dominated by $y$, uniformly for spectral parameters $w_{j} \in \bbd_\ell$, deterministic matrices $A_{j}$ regular in the sense of Definition \ref{Def:regular-obs-in-chain}, as well as generic deterministic matrices $D_j$ and vectors $\fku_{j}$ with bounded norms. We even omit the mention of $\bbd_\ell$ when the spectral domain is evident from the context.

With the notations introduced above, our objective, the local laws presented in Theorem \ref{thm:local-law-regular}, can be rephrased as follows,
\begin{equation}
    \ParaAVE_1 + \ParaISO_1 \prec 1,
    \quad
    \ParaAVE_2 \prec (N \eta)^{1/2},
    \quad
    \ParaISO_2 \prec (N \eta)^{1/4},
    \quad \text{ uniformly in } \bbd_{\infty}.
    \label{bound:objective}
\end{equation}
The following inductive hypothesis is frequently employed in our proof,
\begin{equation}
    \ParaAVE_k \prec \paraAVE_k
    \qand
    \ParaISO_k \prec \paraISO_k,
    \quad k = 1,2,
    \quad \text{ uniformly in } \bbd_{\ell},
    \label{bound:inductive-hypothesis}
\end{equation}
where $\paraAVE_k, \paraISO_k$ are deterministic control parameters depending only on $N$ and $\eta$. In light of Lemma \ref{lemma:opD-bound-regular}, it is not difficult to see that if we introduce
\begin{equation}
    \boundAVE_1 := 1 + \frac{\paraAVE_1}{N \eta^{1/2}},
    \quad 
    \boundISO_1 := 1 + \frac{\paraISO_1}{\sqrt{N \eta^{2}}},
    \quad 
    \boundAVE_2 := 1 + \frac{\paraAVE_2}{N \eta},
    \quad 
    \boundISO_2 := \frac{1}{\eta} + \frac{\paraISO_2}{\sqrt{N \eta^{3}}},
\end{equation}
then we have the following bounds provided (\ref{bound:inductive-hypothesis}),
\begin{equation}
    \BoundAVE_k \prec \boundAVE_k
    \qand
    \BoundISO_k \prec \boundISO_k,
    \quad k = 1,2,
    \quad \text{ uniformly in } \bbd_{\ell}.
\end{equation}

As will be demonstrated later, our application of the cumulant expansion formula (\ref{eqn:cumulant-expansion}) frequently involves resolvent chains containing more than two regular matrices. To avoid directly dealing with these quantities, we rely on the following reduction lemma, which simplifies the analysis of these longer resolvent chains to shorter ones, albeit with the inclusion of a prefactor (powers of $N \eta$) that is manageable for our purposes. We defer the proof to Section \ref{subsec:reduction-inequalities}.

\begin{lemma}[Reduction inequalities] \label{lemma:reduction-inequality}
Suppose we have (\ref{bound:inductive-hypothesis}). By introducing the control parameters
\begin{equation}
    \boundISO_3 := \frac{\sqrt{N}}{\eta} 
    \Big [ {1 + \frac{\paraAVE_2}{N \eta} } \Big ]^{1/2}
    \Big [ {1 + \frac{\paraISO_2}{(N \eta)^{1/2}} } \Big ]
    \qand    
    \boundISO_4 := \frac{N}{\eta} 
    \Big [ {1 + \frac{\paraAVE_2}{N \eta} } \Big ]
    \Big [ {1 + \frac{\paraISO_2}{(N \eta)^{1/2}} } \Big ],
    \label{def:bound-iso-3-4}
\end{equation}
we have the following bounds uniformly in $\bbd_{\ell}$,
\begin{equation}
    \BoundISO_k \prec \boundISO_k,
    \quad \text{ for } \quad
    k = 3,4.
    \label{bound:reduction-ISO-3-4}
\end{equation}
Moreover, if the matrices $A_k$ below are regular w.r.t. their surrounding spectral parameters, then we also have the following bounds uniformly in $\bbd_{\ell}$,
\begin{equation}
    \norm{G \fku}^2 \prec \frac{1}{\eta},
    \quad
    \norm{G_2 A_1 G_1 \fku}^2 \prec \frac{\boundISO_2}{\eta}
    \qand
    \norm{G_3 A_2 G_2 A_1 G_1 \fku}^2 \prec \frac{\boundISO_4}{\eta},
    \label{bound:vec-chain-with-Id-middle}
\end{equation}
\end{lemma}

\subsection{Self-improving inequalities: proof of Theorem \ref{thm:local-law-regular}} \label{subsec:self-improving}

We are now in a position to present the system of self-improving inequalities that lead to our multi-resolvent local laws (\ref{bound:multi-local-law-regular}). It is worth noting that the resolvent chains involving more than two regular matrices have already been processed by Lemma \ref{lemma:reduction-inequality}. Consequently, the inequalities in (\ref{bound:self-improving-ineqs}) below do not involve control parameters with indices larger than $2$.

However, as previously mentioned, the reduction inequalities in Lemma \ref{lemma:reduction-inequality} are unable to accurately capture the magnitude of the forms $\BoundISO_3$ and $\BoundISO_4$. The large prefactors in (\ref{def:bound-iso-3-4}) ultimately result in the non-optimal convergence rate in the local laws (\ref{bound:ave-law-two-regular}) and (\ref{bound:iso-law-two-regular}). In the context of Wigner matrices \cite{cipolloniOptimalMultiresolventLocal2022,cipolloniRankuniformLocalLaw2022}, this issue is circumvented by extending the derivation of self-improving inequalities to $k=3,4$, and using refined estimates of $\BoundISO_3$ and $\BoundISO_4$ to achieve the optimal estimate for $\ParaAVE_2$ and $\ParaISO_2$. We could have adopted this strategy to enhance our estimates in Proposition \ref{prop:self-improving} and Theorem \ref{thm:local-law-regular}. However, considering the complexity of our regularization compared to Wigner matrices, implementing such a strategy would significantly lengthen the arguments. Given the scope and focus of this article, we defer this technical improvement to future research.

\begin{proposition}[Self-improving inequalities] \label{prop:self-improving}
Given (\ref{bound:inductive-hypothesis}), we have the following uniformly in $\bbd_{\ell + 1}$,
\begin{subequations} \label{bound:self-improving-ineqs}
\begin{align}
    \ParaAVE_1 & \prec 1 
    + \frac{\paraAVE_1}{N \eta} 
    + \frac{\paraISO_1}{(N \eta)^{1/2}} 
    + \frac{(\paraISO_2)^{1/2}}{(N \eta)^{1/4}},
    \label{bound:master-AVE-1} \\
    \ParaISO_1 & \prec 1 
    + \frac{\paraAVE_1}{(N \eta)^{1/2}} 
    + \frac{\paraISO_1}{(N \eta)^{1/2}} 
    + \frac{(\paraISO_2)^{1/2}}{(N \eta)^{1/4}},
    \label{bound:master-ISO-1} \\
    \ParaAVE_2 & \prec (N \eta)^{1/2} 
    + \paraAVE_1
    + \frac{(\paraAVE_1)^2}{(N \eta)^{1/2}}
    + \frac{(\paraISO_1)^2}{(N \eta)^{1/2}}
    + \frac{\paraAVE_2}{(N \eta)^{1/2}} 
    + \paraISO_2,
    \label{bound:master-AVE-2} \\
    \ParaISO_2 & \prec (N \eta)^{1/4} 
    + \paraISO_1
    + \frac{(\paraAVE_1)^2}{N \eta}
    + \frac{(\paraISO_1)^2}{(N \eta)^{1/2}}
    +  \frac{\paraAVE_2}{(N \eta)^{3/4}} 
    + \frac{\paraISO_2}{(N \eta)^{1/4}}.
    \label{bound:master-ISO-2}
\end{align}
\end{subequations}
\end{proposition}

The following lemma of iteration is repeatedly utilized in the proof of Theorem \ref{thm:local-law-regular} and can be found in \cite[Lemma 4.11]{cipolloniOptimalLowerBound2023}.

\begin{lemma}[Iteration] \label{lemma:iteration}
Fix some (small) $c > 0$ and (large) $C > 0$. Suppose $\mathcal{Y} \equiv \mathcal{Y} (w) \geq 0$ is a family of nonnegative random variables satisfying $\mathcal{Y} \prec N^C$ on $\bbd_1$, and for every fixed $\ell \in \bbn$, 
\begin{equation*}
    \mathcal{Y} \prec y^{\mathrm{old}} \text{ uniformly in } \bbd_{\ell}
    \quad \Longrightarrow \quad
    \mathcal{Y} \prec y^{\mathrm{new}} + \frac{y^{\mathrm{old}}}{N^c} \text{ uniformly in } \bbd_{\ell+1},
\end{equation*}
where $y^{\mathrm{old}}, y^{\mathrm{new}} > 0$ are deterministic control parameters. Then, there exists some $L_0 \equiv L_0 (\varepsilon, \tau, c, C) \in \bbn$ such that the following holds: for every fixed $\ell \in \bbn$,
\begin{equation*}
    \mathcal{Y} \prec y^{\mathrm{old}} \text{ uniformly in } \bbd_{\ell}
    \quad \Longrightarrow \quad
    \mathcal{Y} \prec y^{\mathrm{new}} \text{ uniformly in } \bbd_{\ell+L_0}.
\end{equation*}
\end{lemma}

With Proposition \ref{prop:self-improving} and Lemma \ref{lemma:iteration}, we are now ready for the proof of our multi-resolvent local laws with regular observables.

\begin{proof}[Proof of Theorem \ref{thm:local-law-regular}]
We can initiate the iteration procedure with the following trivial estimates that directly follow from $\| G \| \leq 1/\eta$,
\begin{equation*}
    \ParaAVE_k \prec N^{k}
    \qand
    \ParaISO_k \prec N^{k+1},
    \quad k = 1,2,
    \quad \text{ uniformly in } \bbd_{1}.
\end{equation*}
By summing up the inequalities (\ref{bound:master-AVE-1}) and (\ref{bound:master-ISO-1}), it follows that
\begin{equation*}
    (\ref{bound:inductive-hypothesis})
    \quad \Longrightarrow \quad 
    \ParaAVE_1 + \ParaISO_1 \prec 1 
    + \frac{\paraAVE_1 + \paraISO_1}{(N \eta)^{1/2}} 
    + \frac{(\paraISO_2)^{1/2}}{(N \eta)^{1/4}},
    \quad \text{ uniformly in } \bbd_{\ell + 1}.
\end{equation*}
Therefore, by employing Lemma \ref{lemma:iteration}, we find that for some sufficiently large $L^\prime \equiv L^\prime (\varepsilon, \tau) \in \bbn$, the following bound holds provided (\ref{bound:inductive-hypothesis}),
\begin{equation}
    \ParaAVE_1 + \ParaISO_1 \prec 1 
    + \frac{(\paraISO_2)^{1/2}}{(N \eta)^{1/4}},
    \quad \text{ uniformly in } \bbd_{\ell + L^\prime}.
    \label{tmp:iteration-AVE-ISO-1}
\end{equation}
Plugging (\ref{tmp:iteration-AVE-ISO-1}) into (\ref{bound:master-AVE-2}), we arrive at
\begin{equation*}
    (\ref{bound:inductive-hypothesis})
    \quad \Longrightarrow \quad 
    \ParaAVE_2 \prec (N \eta)^{1/2} 
    + \frac{\paraAVE_2}{(N \eta)^{1/2}} 
    + \paraISO_2,
    \quad \text{ uniformly in } \bbd_{\ell + L^\prime}.
\end{equation*}
By utilizing Lemma \ref{lemma:iteration} again, it follows that (\ref{bound:inductive-hypothesis}) yields
\begin{equation}
    \ParaAVE_2 \prec (N \eta)^{1/2} + \paraISO_2,
    \quad \text{ uniformly in } \bbd_{\ell + L^{\prime \prime}},
    \label{tmp:iteration-AVE-2}
\end{equation}
for some sufficiently large $L^{\prime \prime} \equiv L^{\prime \prime} (\varepsilon, \tau)$. Finally, plugging the bounds (\ref{tmp:iteration-AVE-ISO-1}) and (\ref{tmp:iteration-AVE-2}) into (\ref{bound:master-ISO-2}), we conclude with the implication
\begin{equation*}
    (\ref{bound:inductive-hypothesis})
    \quad \Longrightarrow \quad 
    \ParaISO_2 \prec (N \eta)^{1/4} 
    + \frac{\paraISO_2}{(N \eta)^{1/4}},
    \quad \text{ uniformly in } \bbd_{\ell + L^{\prime \prime}}.
\end{equation*}
This, together with Lemma \ref{lemma:iteration}, implies that $\ParaISO_2 \prec (N \eta)^{1/4}$ uniformly in $\bbd_{L^{\prime \prime \prime}}$ for some sufficiently large $L^{\prime \prime \prime} \equiv L^{\prime \prime \prime} (\varepsilon, \tau)$. Now, we can complete the proof of (\ref{bound:objective}) and thus of Theorem \ref{thm:local-law-regular} by applying this estimate back to (\ref{tmp:iteration-AVE-ISO-1}) and (\ref{tmp:iteration-AVE-2}).
\end{proof}

\subsection{Cumulant expansion: proof of Proposition \ref{prop:self-improving}} \label{subsec:cumulant-expansion}

The remainder of Section \ref{sec:self-improving-inequalities} is dedicated to deriving the self-improving inequalities (\ref{bound:self-improving-ineqs}). As mentioned, the argument hinges on the cumulant expansion formula (\ref{eqn:cumulant-expansion}) in conjunction with the following proposition, which expresses the centralized forms of resolvent chains as fully underlined terms. One can compare Proposition \ref{prop:repre-full-underlined} with its counterpart in \cite{cipolloniOptimalLowerBound2023}. The primary difference is that here we require the regularity of $B_1 \bfSigma^\pm$ instead of $B_1$ itself. This distinction arises from the anisotropic nature of the operator $\opSd$ in our context.

\begin{proposition}[Representation as fully underlined] \label{prop:repre-full-underlined}
Given (\ref{bound:inductive-hypothesis}), there exists a deterministic matrix $B_1 \equiv B_1 (w_1, A_1, w_2)$ that depends linearly on $A_1$, such that the matrices $B_1 \bfSigma^\pm$ have $\{ w_1, w_k \}$-regularity\footnote{In fact, for $B_1$ constructed in Section \ref{sec:representation}, the matrices $B_1 \bfSigma^\pm$ are $\{ w_1, w \}$-regular for any given $w \in \bbd_0$. This is because the pre-regularization is performed at $w_1$. See Lemma \ref{lemma:one-point-pre-regularization}.} for all $k=1,2,3$ and the following equations hold,
\begin{subequations} \label{eqn:representations}
\begin{align}
    \angles{\Upsilon_1 A_1}
    & = - \angles{\underline{H G_1 B_1}} + \oprec(\rpeAVE_1),
    \label{eqn:representation-AVE-1} \\
    \angles{\fku, \Upsilon_{12} (A_1) \fkv}
    & = - \angles{\fku, \underline{G_1 B_1 H G_2} \fkv} + \oprec(\rpeISO_1),
    \label{eqn:representation-ISO-1} \\
    \angles{\Upsilon_{12} (A_1) A_2}
    & = - \angles{\underline{G_1 B_1 H G_2 A_2}} + \oprec(\rpeAVE_2),
    \label{eqn:representation-AVE-2} \\
    \angles{\fku, \Upsilon_{123} (A_1, A_2) \fkv}
    & = - \angles{\fku, \underline{G_1 B_1 H G_2 A_2 G_3} \fkv} + \oprec(\rpeISO_2),
    \label{eqn:representation-ISO-2}
\end{align}
\end{subequations}
where the error terms are controlled by the parameters
\begin{subequations}
\begin{align}
    \rpeAVE_1 & := \frac{1}{N \eta^{1/2}} \Big [ {1 
    + \frac{\paraAVE_1}{N \eta} 
    + \frac{\paraISO_1}{(N \eta)^{1/2}}} \Big ],
    \label{bound:representation-err-AVE-1} \\
    \rpeISO_1 & := \frac{1}{\sqrt{N \eta^2}} \Big [ {1 
    + \frac{\paraAVE_1}{(N \eta)^{1/2}} 
    + \frac{\paraISO_1}{N \eta} 
    + \frac{(\paraISO_2)^{1/2}}{(N \eta)^{3/4}} } \Big ],
    \label{bound:representation-err-ISO-1} \\
    \rpeAVE_2 & := \frac{1}{N \eta} \Big [ {1 
    + \paraAVE_1
    + \frac{(\paraAVE_1)^2}{(N \eta)^{1/2}}
    + \frac{(\paraISO_1)^2}{(N \eta)^{1/2}}
    + \frac{\paraAVE_2}{N \eta}
    + \frac{\paraISO_2}{(N \eta)^{1/2}}} \Big ],
    \label{bound:representation-err-AVE-2} \\
    \rpeISO_2 & := \frac{1}{\sqrt{N \eta^3}} \Big [ {1 
    + \paraISO_1
    + \frac{(\paraAVE_1)^2}{N \eta}
    + \frac{(\paraISO_1)^2}{(N \eta)^{1/2}}
    + \frac{\paraAVE_2}{N \eta}
    + \frac{\paraISO_2}{(N \eta)^{1/2}}} \Big ].
    \label{bound:representation-err-ISO-2}
\end{align}    
\end{subequations}    
\end{proposition}

The proof of Proposition \ref{prop:repre-full-underlined} is detailed in Section \ref{sec:representation}. Here, let us focus on proving Proposition \ref{prop:self-improving}. The following lemma summarizes the compatibility of stochastic domination $\prec$ with expectation $\bbe$. See e.g. \cite[Lemma 7.1]{benaych-georgesLecturesLocalSemicircle2018} for a proof.

\begin{lemma} \label{lemma:prec-expectation}
Suppose that the deterministic control parameter $y$ satisfies $y \geq N^{-C}$ for some (large) $C > 0$. Also assume that for all $p \in \bbn$ there is a constant $C_p$ such that the family of nonnegative random variables $\mathcal{Y}$ satisfies $\bbe \mathcal{Y}^p \prec N^{C_p}$. Then we have the equivalence
\begin{equation*}
    \mathcal{Y} \prec y
    \quad \Longleftrightarrow \quad
    \bbe \mathcal{Y}^p \prec y^p \text{ for every fixed } p \in \bbn.
\end{equation*}
\end{lemma}

We recurrently use Lemma \ref{lemma:prec-expectation} in the proof, sometimes without explicit reference. Particularly, this lemma allows us to derive the self-improving inequalities by estimating the moments of the average/isotropic forms. Next, we have the following identity, which is useful when dealing with the second-order contribution that arises in cumulant expansion. Recall the definitions of $\dimu$ and $\bfSigma^\pm$ in (\ref{eqn:H-sum-of-rank2}) and (\ref{def:Sigma-pm}). It is not difficult to verify that, for generic matrices $D_1, D_2$ and vectors $\fku_1, \fku_2, \fkv_1, \fkv_2$ of compatible size, we have
\begin{subequations}
\begin{align}
    \sum_{i, \mu} \angles{\dimu D_1} \angles{\dimu D_2}
    & = \frac{1}{2N} \sum_{\alpha = \pm} \alpha [ {
    \angles{D_1 \bfSigma^\alpha D_2 \bfSigma^\alpha} 
    + \angles{D_1 \bfSigma^\alpha D_2^\top \bfSigma^\alpha} } ], 
    \label{eqn:sum-dimu-matrices} \\
    \sum_{i, \mu} \angles{\fku_1, \dimu \fkv_1} \angles{\fku_2, \dimu \fkv_2}
    & = \frac{1}{2} \sum_{\alpha = \pm} \alpha [ {
    \angles{\fku_1, \bfSigma^\alpha \fkv_2} \angles{\fku_2, \bfSigma^\alpha \fkv_1}
    + \angles{\fku_1, \bfSigma^\alpha \bar{\fku}_2} \angles{\bar{\fkv}_2, \bfSigma^\alpha \fkv_1} } ].  
    \label{eqn:sum-dimu-vectors} 
\end{align}
\end{subequations}

We now start to derive the self-improving inequalities in Proposition \ref{prop:self-improving}.

\shortpara{Derivation of (\ref{bound:master-AVE-1})}: By (\ref{eqn:H-sum-of-rank2}) and definition of the second-order renormalization (\ref{def:second-order-renormalization}), we have
\begin{equation*}
    \underline{H G B}
    = \sum_{i, \mu} x_{i \mu} (\dimu G B) - \frac{1}{N} \parimu (\dimu G B).
\end{equation*}
Therefore, for arbitrary fixed $p \in \bbn$, we can leverage (\ref{eqn:representation-AVE-1}), and then (\ref{eqn:cumulant-expansion}) and (\ref{eqn:derivative-G}) to get
\begin{align*}
    \bbe \abs{\angles{\Upsilon A}}^{2p}
    & = {- \bbe \big [ {\big ( {\angles{\underline{H G B}} + \oprec(\rpeAVE_1)} \big )
    \angles{\Upsilon A}^{p-1} \angles{\Upsilon^* A^*}^{p}} \big ]} \\
    & \prec \rpeAVE_1 \bbe \abs{\angles{\Upsilon A}}^{2p-1}
    + \bbe \big [ {\gaussAVE_1 \abs{\angles{\Upsilon A}}^{2p-2} } \big ]
    + \sum_{r + \operatorname{sum} (\bbt) \geq 2}
    \bbe \big [ {\highAVE_1 (r, \bbt) \abs{\angles{\Upsilon A}}^{2p-1-\abs{\bbt}} } \big ],
\end{align*}
where the quantities $\gaussAVE_1$ and $\highAVE_1 (r, \bbt)$ will be specified below, and the summation is taken over nonnegative integers $r$ and multisets of natural numbers $\bbt$. We use $\abs{\bbt}$ to denote the cardinality of $\bbt$ and define $\operatorname{sum} (\bbt) := \sum_{t \in \bbt} t$. Here, the second-order contribution from cumulant expansion (\ref{eqn:cumulant-expansion}) comprises two parts. The leading part has been offset due to the fully underlined term, while the subleading part is encapsulated by the quantity
\begin{equation}
    \gaussAVE_1
    := \frac{1}{N^2} \sum_{\alpha = \pm}
    \abs{\angles{G B \bfSigma^\alpha G A G \bfSigma^\alpha}}
    + \abs{\angles{G B \bfSigma^\alpha G A^\top G \bfSigma^\alpha}} 
    + \cdots,
    \label{def:gaussian-AVE-1}
\end{equation}
where we used (\ref{eqn:sum-dimu-matrices}) to derive this expression. We opt not to list all the terms, as the omitted terms\footnote{The number of omitted terms is finite.} can be controlled in a manner analogous to the listed representative terms. On the other hand, for the contribution stemming from higher-order cumulants, we used the notation
\begin{equation}
    \highAVE_1 (r, \bbt)
    := N^{-(r + \operatorname{sum} (\bbt) + 1)/2} \sum_{i, \mu} \big | {\parimu^r \angles{\dimu G B} } \big |
    \prod_{t \in \bbt} \big | {\parimu^{t} \angles{G A} } \big |.
    \label{def:high-AVE-1}
\end{equation}
Strictly speaking, the summation over $r$ and $\bbt$ should have an upper limit and we should include a remainder term as in (\ref{eqn:cumulant-expansion}). However, by choosing this upper limit to be sufficiently large, while depending only $p$ (e.g., $100p$), one can easily prove that the corresponding remainder term is dominated by $\oprec (N^{-2p})$. In other words, this additional error term has no significant impact on our argument. Therefore, we omit this technical detail from the equations above.

The quantities $\gaussAVE_1$ and $\highAVE_1 (r, \bbt)$ are estimated in Lemmas \ref{lemma:gaussian-contribution} and \ref{lemma:high-order-contribution} below. Applying the estimates (\ref{bound:gaussian-AVE-1}) and (\ref{bound:high-AVE-1}) provided in these lemmas, we obtain
\begin{equation*}
    \bbe \abs{\angles{\Upsilon A}}^{2p}
    \prec \sum_{q=1}^{2p} \bigg ( {\rpeAVE_1 + \frac{1}{N \eta^{1/2}} \Big [ {1 
    + \frac{\paraISO_1}{(N \eta)^{1/2}} 
    + \frac{(\paraISO_2)^{1/2}}{(N \eta)^{1/4}}} \Big ]} \bigg )^q
    \big ( {\bbe \abs{\angles{\Upsilon A}}^{2p}} \big )^{1-q/(2p)}.
\end{equation*}
Utilizing Young's inequalities, we therefore arrive at
\begin{equation*}
    \bbe \abs{\angles{\Upsilon A}}^{2p} \prec
    \bigg ( {\rpeAVE_1 + \frac{1}{N \eta^{1/2}} \Big [ {1 
    + \frac{\paraISO_1}{(N \eta)^{1/2}} 
    + \frac{(\paraISO_2)^{1/2}}{(N \eta)^{1/4}}} \Big ]} \bigg )^{2p}.
\end{equation*}
Recalling (\ref{bound:representation-err-AVE-1}) and Lemma \ref{lemma:prec-expectation}, this concludes the proof of (\ref{bound:master-AVE-1}) since $p \in \bbn$ is arbitrary.


\shortpara{Derivation of (\ref{bound:master-ISO-1})}: The derivation of the remaining three self-improving inequalities follows the same procedure as that for (\ref{bound:master-AVE-1}). Therefore, we only highlight the definitions that are necessary for the subsequent discussions in Sections \ref{subsec:gaussian-contribution} and \ref{subsec:high-order-contribution}. Referring to (\ref{def:second-order-renormalization}), we have
\begin{equation*}
    \underline{G_1 B_1 H G_2}
    = \sum_{i, \mu} x_{i \mu} (G_1 B_1 \dimu G_2) - \frac{1}{N} \parimu (G_1 B_1 \dimu G_2).
\end{equation*}
Hence, we can apply the cumulant expansion formula (\ref{eqn:cumulant-expansion}) and the representation (\ref{subsec:repre-ISO-1}) to get
\begin{align*}
    \bbe \abs{\angles{\fku, \Upsilon_{12} (A_1) \fkv}}^{2p}
    = & \ {- \bbe \big [ { \big ( {\angles{\fku, \underline{G_1 B_1 H G_2} \fkv} + \oprec(\rpeISO_1)} \big )
    \angles{\fku, \Upsilon_{12} (A_1) \fkv}^{p-1} 
    \angles{\fku, \Upsilon_{1^* 2^*} (A_1) \fkv}^{p}} \big ]} \\
    \prec & \ \rpeISO_1 \bbe \abs{\angles{\fku, \Upsilon_{12} (A_1) \fkv}}^{2p-1}
    + \bbe \big [ {\gaussISO_1 \abs{\angles{\fku, \Upsilon_{12} (A_1) \fkv}}^{2p-2} } \big ] \\
    & + \sum_{r + \operatorname{sum} (\bbt) \geq 2}
    \bbe \big [ {\highISO_1 (r, \bbt) 
    \abs{\angles{\fku, \Upsilon_{12} (A_1) \fkv}}^{2p-1-\abs{\bbt} } } \big ].
\end{align*}
Here, the error term related to the second-order cumulant is given by
\begin{align}
\begin{split}
    \gaussISO_1 := & \ \frac{1}{N} \sum_{\alpha = \pm} 
    \abs{ \angles{\fku, G_1 B_1 \bfSigma^{\alpha} G_2 \fkv}
    \angles{\fku, G_1 A_1 G_2 \bfSigma^{\alpha} G_2 \fkv} } \\
    & + \frac{1}{N} \sum_{\alpha = \pm} 
    \abs{ \angles{\fku, G_1 B_1 \bfSigma^{\alpha} G_1 A_1 G_2 \fkv}
    \angles{\fku, G_1 \bfSigma^{\alpha} G_2 \fkv} } + \cdots,
\end{split} \label{def:gaussian-ISO-1}
\end{align}
while the error term associated with higher-order cumulants is defined as
\begin{equation}
    \highISO_1 (r, \bbt)
    := N^{-(r + \operatorname{sum} (\bbt) + 1)/2}  \sum_{i, \mu} \big | {\parimu^r \angles{\fku, G_1 B_1 \dimu G_2 \fkv} } \big |
    \prod_{t \in \bbt} \big | {\parimu^{t} \angles{\fku, G_1 A_1 G_2 \fkv} } \big |.
    \label{def:high-ISO-1}
\end{equation}
We obtain the self-improving inequality (\ref{bound:master-ISO-1}) by utilizing the estimates (\ref{bound:gaussian-ISO-1}) and (\ref{bound:high-ISO-1}).

\shortpara{Derivation of (\ref{bound:master-AVE-2})}: Again, we employ (\ref{eqn:cumulant-expansion}) and (\ref{eqn:representation-AVE-2}) to get
\begin{align*}
    \bbe \abs{\angles{\Upsilon_{12} (A_1) A_2}}^{2p}
    \prec & \ \rpeAVE_2 \bbe \abs{\angles{\Upsilon_{12} (A_1) A_2}}^{2p-1}
    + \bbe \big [ {\gaussAVE_2 \abs{\angles{\Upsilon_{12} (A_1) A_2}}^{2p-2} } \big ] \\
    & + \sum_{r + \operatorname{sum} (\bbt) \geq 2}
    \bbe \big [ {\highAVE_2 (r, \bbt) 
    \abs{\angles{\Upsilon_{12} (A_1) A_2}}^{2p-1-\abs{\bbt} } } \big ],
\end{align*}
where we introduced
\begin{align}
\begin{split}
    \gaussAVE_2 = & \ \frac{1}{N^2} \sum_{\alpha = \pm} 
    \abs{\angles{G_2 A_2 G_1 B_1 \bfSigma^\alpha 
    G_1 A_1 G_2 A_2 G_1 \bfSigma^\alpha}} \\
    & + \frac{1}{N^2} \sum_{\alpha = \pm} \abs{\angles{G_2 A_2 G_1 B_1 \bfSigma^\alpha 
    G_1 A_2^\top G_2 A_1^\top G_1 \bfSigma^\alpha}} 
    + \cdots
\end{split} \label{def:gaussian-AVE-2}    
\end{align}
as well as
\begin{equation}
    \highAVE_2 (r, \bbt)
    := N^{-(r + \operatorname{sum} (\bbt) + 1)/2} \sum_{i, \mu} \big | {\parimu^r \angles{G_1 B_1 \dimu G_2 A_2} } \big |
    \prod_{t \in \bbt} \big | {\parimu^{t} \angles{G_1 A_1 G_2 A_2} } \big |.
    \label{def:high-AVE-2}
\end{equation}
This time, we employ (\ref{bound:gaussian-AVE-2}) and (\ref{bound:high-AVE-2}) below to conclude the derivation.


\shortpara{Derivation of (\ref{bound:master-ISO-2})}: Finally, we have
\begin{align*}
    \bbe \abs{\angles{\fku, \Upsilon_{123} (A_1, A_2) \fkv}}^{2p}
    \prec & \ \rpeISO_2 \bbe \abs{\angles{\fku, \Upsilon_{123} (A_1, A_2) \fkv}}^{2p-1}
    + \bbe \big [ {\gaussISO_2 \abs{\angles{\fku, \Upsilon_{123} (A_1, A_2) \fkv}}^{2p-2} } \big ] \\
    & + \sum_{r + \operatorname{sum} (\bbt) \geq 2}
    \bbe \big [ {\highISO_2 (r, \bbt) 
    \abs{\angles{\fku, \Upsilon_{123} (A_1, A_2) \fkv}}^{2p-1-\abs{\bbt} } } \big ],
\end{align*}
where we introduced
\begin{align}
\begin{split}
    \gaussISO_2 := & \ \frac{1}{N} \sum_{\alpha = \pm} \abs{ \angles{\fku, G_1 B_1 \bfSigma^{\alpha} G_3 \fkv}
    \angles{\fku, G_1 A_1 G_2 A_2 G_3 \bfSigma^{\alpha} G_2 A_2 G_3 \fkv} } \\
    & + \frac{1}{N} \sum_{\alpha = \pm} \abs{ \angles{\fku, G_1 B_1 \bfSigma^{\alpha} G_2 A_2 G_3 \fkv}
    \angles{\fku, G_1 A_1 G_2 \bfSigma^{\alpha} G_2 A_2 G_3 \fkv} } \\
    & + \frac{1}{N} \sum_{\alpha = \pm}
    \abs{ \angles{\fku, G_1 B_1 \bfSigma^{\alpha} G_1 A_1 G_2 A_2 G_3 \fkv}
    \angles{\fku, G_1 \bfSigma^{\alpha} G_2 A_2 G_3 \fkv} }
    + \cdots
\end{split} \label{def:gaussian-ISO-2}
\end{align}
as well as
\begin{equation}
    \highISO_2 (r, \bbt)
    := N^{-(r + \operatorname{sum} (\bbt) + 1)/2} \sum_{i, \mu} \big | {\parimu^r \angles{\fku, G_1 B_1 \dimu G_2 A_2 G_3 \fkv} } \big |
    \prod_{t \in \bbt} \big | {\parimu^{t} \angles{\fku, G_1 A_1 G_2 A_2 G_3 \fkv} } \big |.
    \label{def:high-ISO-2}
\end{equation}
The derivation is completed by applying (\ref{bound:gaussian-ISO-2}) and (\ref{bound:high-ISO-2}) below.


As highlighted above, the following two lemmas provide control over the intermediate quantities that arise in the derivation of the self-improving inequalities. Their proofs are presented in Section \ref{subsec:gaussian-contribution} and Section \ref{subsec:high-order-contribution}, respectively.

\begin{lemma} \label{lemma:gaussian-contribution}
Given the inductive hypothesis (\ref{bound:inductive-hypothesis}), we have the following uniformly in $\bbd_{\ell+1}$,
\begin{subequations}
\begin{align}
    (\gaussAVE_1)^{1/2} & \prec \frac{1}{N \eta^{1/2}} \Big [ {1 
    + \frac{(\paraISO_2)^{1/2}}{(N \eta)^{1/4}}} \Big ], 
    \label{bound:gaussian-AVE-1} \\
    (\gaussISO_1)^{1/2} & \prec \frac{1}{\sqrt{N \eta^2}} \Big [ {1 
    + \frac{\paraISO_1}{(N \eta)^{1/2}} 
    + \frac{(\paraISO_2)^{1/2}}{(N \eta)^{1/4}}} \Big ], 
    \label{bound:gaussian-ISO-1} \\
    (\gaussAVE_2)^{1/2} & \prec \frac{1}{N \eta} \Big [ {(N \eta)^{1/2} 
    + \frac{\paraAVE_2}{(N \eta)^{1/2}} 
    + \paraISO_2} \Big ], 
    \label{bound:gaussian-AVE-2} \\
    (\gaussISO_2)^{1/2} & \prec \frac{1}{\sqrt{N \eta^3}} \Big [ { (N \eta)^{1/4} 
    + \frac{(\paraISO_1)^2}{(N \eta)^{3/4}} 
    + \frac{\paraAVE_2}{(N \eta)^{3/4}} 
    + \frac{\paraISO_2}{(N \eta)^{1/4}} } \Big ]. 
    \label{bound:gaussian-ISO-2} 
\end{align}
\end{subequations}
\end{lemma}

\begin{lemma} \label{lemma:high-order-contribution}
Given the inductive hypothesis (\ref{bound:inductive-hypothesis}), we have the following uniformly in $\bbd_{\ell+1}$,
\begin{subequations}
\begin{align}
    \highAVE_1 (r, \bbt)^{1/(\abs{\bbt} + 1)}
    & \prec \frac{1}{N \eta^{1/2}} \Big [ {1 
    + \frac{\paraISO_1}{(N \eta)^{1/2}} 
    + \frac{(\paraISO_2)^{1/4}}{(N \eta)^{1/8}}} \Big ],
    \label{bound:high-AVE-1} \\
    \highISO_1 (r, \bbt)^{1/(\abs{\bbt} + 1)}
    & \prec \frac{1}{\sqrt{N \eta^2}} \Big [ {1 
    + \frac{\paraISO_1}{(N \eta)^{1/2}} 
    + \frac{(\paraISO_2)^{1/4}}{(N \eta)^{1/8}}} \Big ], 
    \label{bound:high-ISO-1} \\
    \highAVE_2 (r, \bbt)^{1/(\abs{\bbt} + 1)}
    & \prec \frac{1}{N \eta} \Big [ {(N \eta)^{1/4} 
    + \frac{(\paraISO_1)^2}{(N \eta)^{3/4}} 
    + \frac{\paraISO_2}{(N \eta)^{1/4}}} \Big ], 
    \label{bound:high-AVE-2} \\
    \highISO_2 (r, \bbt)^{1/(\abs{\bbt} + 1)}
    & \prec \frac{1}{\sqrt{N \eta^3}} \Big [ {(N \eta)^{1/4} 
    + \frac{(\paraISO_1)^2}{(N \eta)^{3/4}} 
    + \frac{\paraISO_2}{(N \eta)^{1/4}}} \Big ]. 
    \label{bound:high-ISO-2}
\end{align}
\end{subequations}
\end{lemma}

\subsection{Second-order contribution: proof of Lemma \ref{lemma:gaussian-contribution}} \label{subsec:gaussian-contribution}

Let us begin with the two estimates concerning the averaged forms, (\ref{bound:gaussian-AVE-1}) and (\ref{bound:gaussian-AVE-2}). Recall that $B_1$ comes from Proposition \ref{prop:repre-full-underlined}. Hence, the matrices $B_1 \bfSigma^\pm$ are $\{ w_1, w_k \}$-regular for all $k=1,2,3$. Upon examining the resolvent chains in (\ref{def:gaussian-AVE-1}) and (\ref{def:gaussian-AVE-2}), we observe that all the deterministic matrices are regular w.r.t. their surrounding spectral parameters expect for the matrix $\bfSigma^\pm$ that we place at the end of the chains. Therefore, we can utilize the spectral decomposition of $\bfSigma^\pm$ and then apply the isotropic estimates, yielding
\begin{equation*}
    \gaussAVE_1 \prec \frac{\boundISO_2}{N^2} 
    = \frac{1}{N^2 \eta} \Big [ {1 + \frac{\paraISO_2}{(N \eta)^{1/2}} } \Big ]
    \qand
    \gaussAVE_2 \prec \frac{\boundISO_4}{N^2} 
    = \frac{1}{N \eta} \Big [ {1 + \frac{\paraAVE_2}{N \eta} } \Big ]
    \Big [ {1 + \frac{\paraISO_2}{(N \eta)^{1/2}} } \Big ],
\end{equation*}
where we used the definition of $\boundISO_4$ in (\ref{def:bound-iso-3-4}).

Moving forward, let us consider the estimate (\ref{bound:gaussian-ISO-1}). We can utilize (\ref{bound:vec-chain-with-Id-middle}) for the second factor in the first line and the second factor in the second line of (\ref{def:gaussian-ISO-1}). It turns out that
\begin{equation*}
    \gaussISO_1 \prec \frac{\boundISO_1 (\boundISO_2)^{1/2}}{N \eta} 
    + \frac{\boundISO_2}{N \eta} 
    \lesssim \frac{1}{N \eta^2} \Big [ {1 + \frac{(\paraISO_1)^2}{N \eta} + \frac{\paraISO_2}{(N \eta)^{1/2}} } \Big ].
\end{equation*}
It remains to demonstrate the estimate (\ref{bound:gaussian-ISO-2}). For the three representative terms outlined in (\ref{def:gaussian-ISO-2}), we utilize (\ref{bound:vec-chain-with-Id-middle}) to control their second factor due to the presence of the non-regular matrix $\bfSigma^\pm$. Specifically, the first line of (\ref{def:gaussian-ISO-2}) can be controlled by
\begin{equation*}
    \frac{\boundISO_1 (\boundISO_2)^{1/2} (\boundISO_4)^{1/2}}{N \eta}
    \lesssim \frac{(N \eta)^{1/2}}{N \eta^3} 
    \Big [ {1 + \frac{\paraISO_1}{(N \eta)^{1/2}} } \Big ]
    \Big [ {1 + \frac{\paraAVE_2}{N \eta} } \Big ]^{1/2}
    \Big [ {1 + \frac{\paraISO_2}{(N \eta)^{1/2}} } \Big ].
\end{equation*}
While the second line of (\ref{def:gaussian-ISO-2}) is dominated by
\begin{equation*}
    \frac{(\boundISO_2)^2}{N \eta} = \frac{(N \eta)^{1/2}}{N \eta^3} 
    \Big [ {1 + \frac{\paraISO_2}{(N \eta)^{1/2}} } \Big ]^2.
\end{equation*}
Finally, for the third line of (\ref{def:gaussian-ISO-2}), we can estimate it with
\begin{equation*}
    \frac{\boundISO_3 (\boundISO_2)^{1/2}}{N \eta} 
    = \frac{(N \eta)^{1/2}}{N \eta^3} 
    \Big [ {1 + \frac{\paraAVE_2}{N \eta} } \Big ]^{1/2}
    \Big [ {1 + \frac{\paraISO_2}{(N \eta)^{1/2}} } \Big ]^{3/2}.
\end{equation*}
Summarizing the three estimates above concludes the proof of (\ref{bound:gaussian-ISO-2}).

\subsection{Higher-order contribution: proof of Lemma \ref{lemma:high-order-contribution}}
\label{subsec:high-order-contribution}

This section is devoted to the proof of Lemma \ref{lemma:high-order-contribution}. Recall (\ref{eqn:derivative-G}). We have the following trivial bounds given (\ref{bound:inductive-hypothesis}),
\begin{alignat*}{2}
    & \big | { \parimu^r \angles{\dimu G B} } \big |
    \prec N^{-1},
    & 
    & \big | {\parimu^r \angles{G_1 B_1 \dimu G_2 A_2} } \big |
    \prec N^{-1} \boundISO_1, \\
    & \big | {\parimu^{t} \angles{G A}} \big | 
    \prec N^{-1} \boundISO_1, 
    & \qquad 
    & \big | {\parimu^t \angles{G_1 A_1 G_2 A_2} } \big |
    \prec N^{-1} \big [ {(\boundISO_1)^2 + \boundISO_2} \big ].
\end{alignat*}
Plugging these bounds into the definition of $\highAVE_1$ and $\highAVE_2$ in (\ref{def:high-AVE-1}) and (\ref{def:high-AVE-2}), we obtain
\begin{align*}
    \highAVE_1 (r, \bbt)
    & \prec \frac{N^2}{N^{(r + \operatorname{sum} (\bbt) + 1)/2}} 
    \cdot \frac{1}{(N \eta^{1/2})^{\abs{\bbt} + 1}} 
    \Big [ {1 + \frac{\paraISO_1}{(N \eta)^{1/2}}} \Big ]^{\abs{\bbt} + 1}, \\
    \highAVE_2 (r, \bbt)
    & \prec \frac{N^2}{N^{(r + \operatorname{sum} (\bbt) + 1)/2}} 
    \cdot \frac{1}{(N \eta)^{\abs{\bbt} + 1}} 
    \Big [ {1 + \frac{(\paraISO_1)^2}{N \eta} + \frac{\paraISO_2}{(N \eta)^{1/2}}} \Big ]^{\abs{\bbt} + 1}, 
\end{align*}
which establishes the averaged bounds (\ref{bound:high-AVE-1}) and (\ref{bound:high-AVE-2}) when $r + \operatorname{sum} (\bbt) \geq 3$. Similarly, for the isotropic forms defined in (\ref{def:high-ISO-1}) and (\ref{def:high-ISO-2}), we can apply the elementary bounds
\begin{alignat*}{2}
    & \big | {\parimu^r \angles{\fku, G_1 B_1 \dimu G_2 \fkv} } \big |
    \prec 1 ,
    & 
    & \big | {\parimu^t \angles{\fku, G_1 A_1 G_2 \fkv} } \big |
    \prec \boundISO_1, \\
    & \big | {\parimu^r \angles{\fku, G_1 B_1 \dimu G_2 A_2 G_3 \fkv} } \big |
    \prec \boundISO_1, 
    & \qquad 
    & \big | {\parimu^{t} \angles{\fku, G_1 A_1 G_2 A_2 G_3 \fkv} } \big | 
    \prec (\boundISO_1)^2 + \boundISO_2.
\end{alignat*}
It follows that
\begin{align*}
    \highISO_1 (r, \bbt)
    & \prec \frac{N^{(\abs{\bbt} + 5)/2}}{N^{(r + \operatorname{sum} (\bbt) + 1)/2}}
    \cdot \frac{1}{(\sqrt{N \eta^2})^{\abs{\bbt} + 1}} 
    \Big [ {1 + \frac{\paraISO_1}{(N \eta)^{1/2}}} \Big ]^{\abs{\bbt} + 1}, \\
    \highISO_2 (r, \bbt)
    & \prec \frac{N^{(\abs{\bbt} + 5)/2}}{N^{(r + \operatorname{sum} (\bbt) + 1)/2}}
    \cdot \frac{1}{(\sqrt{N \eta^3})^{\abs{\bbt} + 1}} 
    \Big [ {1 + \frac{(\paraISO_1)^2}{N \eta} + \frac{\paraISO_2}{(N \eta)^{1/2}}} \Big ]^{\abs{\bbt} + 1},
\end{align*}
which resolves (\ref{bound:high-ISO-1}) and (\ref{bound:high-ISO-2}) when $r + \operatorname{sum} (\bbt) - \abs{\bbt} \geq 4$. 

In conclusion, it remains to demonstrate the averaged bounds (\ref{bound:high-AVE-1}) and (\ref{bound:high-AVE-2}) when $r + \operatorname{sum} (\bbt) = 2$, and the isotropic bounds (\ref{bound:high-ISO-1}) and (\ref{bound:high-ISO-2}) when $2 \leq r + \operatorname{sum} (\bbt) < \abs{\bbt} + 3$. In these cases, we need to handle the summation $\sum_{i, \mu}$ more delicately instead of estimating it simply with $O (N^2)$. Specifically, we rely on the following lemma.

\begin{lemma} \label{lemma:sum-off-diagonal}
Suppose the $A_k$'s are regular in the sense of Definition \ref{Def:regular-obs-in-chain}, and the $D_k$'s are generic deterministic matrices with $\norm{D_k} \lesssim 1$. Then, given (\ref{bound:inductive-hypothesis}), we have the following estimates uniformly in $\bbd_{\ell}$,
\begin{subequations}
\begin{align}
    \sum_{i} \abs{\angles{\fku, G_1 D_1 \bfe_i}}^2
    + \sum_{\mu} \abs{\angles{\fku, G_1 D_1 \bfe_\mu}}^2
    & \prec \frac{1}{\eta}, \\
    \sum_{i} \abs{\angles{\fku, G_1 A_1 G_2 D_2 \bfe_i}}^2
    + \sum_{\mu} \abs{\angles{\fku, G_1 A_1 G_2 D_2 \bfe_\mu}}^2
    & \prec \frac{1}{\eta^2} \Big [ {1 + \frac{\paraISO_2}{(N \eta)^{1/2}}} \Big ].
    \label{bound:sum-off-diagonal-1-regular}
\end{align}
\end{subequations}
\end{lemma}

We informally refer to the isotropic forms of the type $\angles{\fku, G_1 A_1 \cdots A_{k-1} G_k D_k \bfe_i}$ as off-diagonal resolvent entries. Lemma \ref{lemma:sum-off-diagonal} refines the trivial estimate for the summation over these off-diagonal resolvent entries with $k$ regular matrices by a factor of $N \eta$, at the expense of introducing control parameters indexed by $2k$. We could have provided an analogous estimate for the summation over $\angles{\fku, G_1 A_1 G_2 A_2 G_3 D_3 \bfe_i}$. However, the suboptimal control of $\BoundISO_4$ through reduction inequalities (\ref{bound:reduction-ISO-3-4}) offsets our expected gain. In other words, in our context, extending Lemma \ref{lemma:sum-off-diagonal} to off-diagonal entries with two regular matrices does not yield any effective improvement. 

The proof of Lemma \ref{lemma:sum-off-diagonal} is deferred to Section \ref{subsec:reduction-inequalities}. For the subsequent discussions, we frequently utilize Lemma \ref{lemma:sum-off-diagonal} alongside the Cauchy-Schwarz inequality. For instance, we have
\begin{equation*}
    \sum_{i} \abs{\angles{\fku, G_1 A_1 G_2 D_2 \bfe_i}}
    \prec \frac{N^{1/2}}{\eta} \Big [ {1 + \frac{\paraISO_2}{(N \eta)^{1/2}}} \Big ]^{1/2}.
\end{equation*}

\shortpara{Proof of (\ref{bound:high-AVE-1}) and (\ref{bound:high-AVE-2}) for $r + \operatorname{sum} (\bbt) = 2$ with $r \in \{ 0, 2 \}$}: In this case, the first factor of the summands in (\ref{def:high-AVE-1}) and (\ref{def:high-AVE-2}) must contain at least one off-diagonal resolvent entry. Consequently, we can leverage Lemma \ref{lemma:sum-off-diagonal} to obtain
\begin{align*}
    \highAVE_1 (r, \bbt)
    & \prec \frac{1}{N \eta^{1/2}} \Big [ {\frac{\boundISO_1}{N}} \Big ]^{\abs{\bbt}}
    \lesssim \frac{1}{(N \eta^{1/2})^{\abs{\bbt} + 1}} 
    \Big [ {1 + \frac{\paraISO_1}{(N \eta)^{1/2}}} \Big ]^{\abs{\bbt} + 1}, \\
    \highAVE_2 (r, \bbt)
    & \prec \frac{1}{N \eta} \Big [ {1 + \frac{\paraISO_2}{(N \eta)^{1/2}}} \Big ]^{1/2}
    \Big [ {\frac{(\boundISO_1)^2 + \boundISO_2}{N}} \Big ]^{\abs{\bbt}}
    \lesssim \frac{1}{(N \eta)^{\abs{\bbt} + 1}} 
    \Big [ {1 + \frac{(\paraISO_1)^2}{N \eta} + \frac{\paraISO_2}{(N \eta)^{1/2}}} \Big ]^{\abs{\bbt} + 1}.
\end{align*}

\shortpara{Proof of (\ref{bound:high-AVE-1}) and (\ref{bound:high-AVE-2}) for $r=1, \bbt = \{1 \}$}: In this case, at least one off-diagonal resolvent entry is present in the second factor of the summands in (\ref{bound:high-AVE-1}) and (\ref{bound:high-AVE-2}). It turns out that
\begin{align*}
    \highAVE_1 (r, \bbt)
    & \prec \frac{1}{N^2 \eta} \Big [ {1 + \frac{\paraISO_2}{(N \eta)^{1/2}}} \Big ]^{1/2}
    \asymp \frac{1}{(N \eta^{1/2})^2} \Big [ {1 + \frac{(\paraISO_2)^{1/4}}{(N \eta)^{1/8}}} \Big ]^{2}, \\
    \highAVE_2 (r, \bbt)
    & \prec \frac{\boundISO_1 \boundISO_2}{N^{3/2}} 
    \lesssim \frac{1}{(N \eta)^2} \Big [ {1 + \frac{(\paraISO_1)^2}{N \eta}} \Big ]
    \Big [ {1 + \frac{\paraISO_2}{(N \eta)^{1/2}}} \Big ]
    \cdot (N \eta)^{1/2}.
\end{align*}
The estimates (\ref{bound:high-AVE-1}) and (\ref{bound:high-AVE-2}) are obtained by taking the square root of the preceding bounds.

\shortpara{Proof of (\ref{bound:high-ISO-1}) and (\ref{bound:high-ISO-2}) for $r + \operatorname{sum} (\bbt) - \abs{\bbt} \in \{ 2, 3 \}$}: In this case, we have $r = \operatorname{sum} (\bbt) \geq \abs{\bbt} + 2$. Hence, applying Lemma \ref{lemma:sum-off-diagonal} yields
\begin{align*}
    \highISO_1 (r, \bbt) & \prec \frac{1}{N^{(\abs{\bbt} + 3)/2}}
    \cdot \frac{N}{\eta} \cdot (\boundISO_1)^{\abs{\bbt}}
    \lesssim \frac{\eta^{\abs{\bbt}/2}}{(\sqrt{N \eta^2})^{\abs{\bbt} + 1}} 
    \Big [ {1 + \frac{\paraISO_1}{(N \eta)^{1/2}}} \Big ]^{\abs{\bbt} + 1}, \\
    \highISO_2 (r, \bbt) & \prec \frac{1}{N^{(\abs{\bbt} + 3)/2}}
    \cdot \bigg \{ \frac{N}{\eta} \boundISO_1
    + \frac{N}{\eta^{3/2}} \Big [ {1 + \frac{\paraISO_2}{(N \eta)^{1/2}}} \Big ]^{1/2} \bigg \} \cdot \big [ {(\boundISO_1)^2 + \boundISO_2} \big ]^{\abs{\bbt}} \\
    & \lesssim \frac{\eta^{\abs{\bbt}/2}}{(\sqrt{N \eta^3})^{\abs{\bbt} + 1}} 
    \Big [ {1 + \frac{(\paraISO_1)^2}{N \eta} + \frac{\paraISO_2}{(N \eta)^{1/2}}} \Big ]^{\abs{\bbt} + 1}.
\end{align*}
Note that we have $\eta^{\abs{\bbt}/2} \lesssim 1$.

\shortpara{Proof of (\ref{bound:high-ISO-1}) and (\ref{bound:high-ISO-2}) for $r + \operatorname{sum} (\bbt) - \abs{\bbt} = 0$}: Essentially, in this case we must have $r = 0$ and $\bbt = \{ 1, \cdots, 1 \}$. In addition, we have $\operatorname{sum} (\bbt) = \abs{\bbt} \geq 2$. Therefore, we can apply Lemma \ref{lemma:sum-off-diagonal} to both the two factors of the terms in (\ref{def:high-ISO-1}) and (\ref{def:high-ISO-2}), resulting in
\begin{align*}
    \highISO_1 (r, \bbt) & \prec \frac{1}{N^{(\abs{\bbt} + 1)/2}}
    \cdot \frac{1}{\eta^{5/2}} \Big [ {1 + \frac{\paraISO_2}{(N \eta)^{1/2}}} \Big ]^{1/2}
    \cdot (\boundISO_1)^{\abs{\bbt} - 1} \\
    & \lesssim \frac{\eta^{\abs{\bbt}/2-1}}{(\sqrt{N \eta^2})^{\abs{\bbt} + 1}} 
    \Big [ {1 + \frac{(\paraISO_2)^{1/4}}{(N \eta)^{1/8}}} \Big ]^{2}
    \Big [ {1 + \frac{\paraISO_1}{(N \eta)^{1/2}}} \Big ]^{\abs{\bbt} - 1}, \\
    \highISO_2 (r, \bbt) & \prec \frac{1}{N^{(\abs{\bbt} + 1)/2}}
    \cdot \bigg \{ \frac{N^{1/2}}{\eta^2}  \boundISO_2 \Big [ {1 + \frac{\paraISO_2}{(N \eta)^{1/2}}} \Big ]^{1/2}
    + \frac{1}{\eta^{7/2}} \Big [ {1 + \frac{\paraISO_2}{(N \eta)^{1/2}}} \Big ]^{3/2} \bigg \} 
    \cdot \big [ {(\boundISO_1)^2 + \boundISO_2} \big ]^{\abs{\bbt}-1} \\
    & \lesssim \frac{\eta^{\abs{\bbt}/2-1}}{(\sqrt{N \eta^3})^{\abs{\bbt} + 1}} 
    \Big [ {1 + \frac{(\paraISO_1)^2}{N \eta} + \frac{\paraISO_2}{(N \eta)^{1/2}}} \Big ]^{\abs{\bbt} + 1}
    \cdot (N \eta)^{1/2}.
\end{align*}
Since $\abs{\bbt} \geq 2$, the factor $(N \eta)^{1/2}$ appearing at the end of the estimate for $\highISO_2 (r, \bbt)$ can, at most, weaken the bound in square bracket by $(N \eta)^{1/6}$.

\shortpara{Proof of (\ref{bound:high-ISO-1}) and (\ref{bound:high-ISO-2}) for $r = 1$ and $\operatorname{sum} (\bbt) - \abs{\bbt} = 0$}: In this case, we have $\bbt = \{ 1, \cdots, 1 \}$ with $\abs{\bbt} \geq 1$. Again, we apply Lemma \ref{lemma:sum-off-diagonal} to both the two factors, leading to
\begin{align*}
    \highISO_1 (r, \bbt) & \prec \frac{1}{N^{(\abs{\bbt} + 2)/2}}
    \cdot \frac{N^{1/2}}{\eta^{2}} \Big [ {1 + \frac{\paraISO_2}{(N \eta)^{1/2}}} \Big ]^{1/2}
    \cdot (\boundISO_1)^{\abs{\bbt} - 1} \\
    & \lesssim \frac{\eta^{(\abs{\bbt}-1)/2}}{(\sqrt{N \eta^2})^{\abs{\bbt} + 1}} 
    \Big [ {1 + \frac{(\paraISO_2)^{1/4}}{(N \eta)^{1/8}}} \Big ]^{2}
    \Big [ {1 + \frac{\paraISO_1}{(N \eta)^{1/2}}} \Big ]^{\abs{\bbt} - 1}, \\
    \highISO_2 (r, \bbt) & \prec \frac{1}{N^{(\abs{\bbt} + 2)/2}}
    \cdot \bigg \{ \frac{N}{\eta} \boundISO_1 \boundISO_2
    + \frac{N^{1/2}}{\eta^{5/2}}  \boundISO_1 \Big [ {1 + \frac{\paraISO_2}{(N \eta)^{1/2}}} \Big ] \bigg \} 
    \cdot \big [ {(\boundISO_1)^2 + \boundISO_2} \big ]^{\abs{\bbt}-1} \\
    & \lesssim \frac{\eta^{(\abs{\bbt}-1)/2}}{(\sqrt{N \eta^3})^{\abs{\bbt} + 1}} 
    \Big [ {1 + \frac{(\paraISO_1)^2}{N \eta} + \frac{\paraISO_2}{(N \eta)^{1/2}}} \Big ]^{\abs{\bbt} + 1}
    \cdot (N \eta)^{1/2}.
\end{align*}
As $\abs{\bbt} \geq 1$, the impact of the factor $(N \eta)^{1/2}$ on the r.h.s of (\ref{bound:high-ISO-2}) is at most $(N \eta)^{1/4}$.

\shortpara{Proof of (\ref{bound:high-ISO-1}) and (\ref{bound:high-ISO-2}) for $r = 0$ and $\operatorname{sum} (\bbt) - \abs{\bbt} = 1$}: Finally, in this case, we need to examine $\bbt = \{ 2, 1, \cdots, 1 \}$ with $\abs{\bbt} \geq 1$. By employing Lemma \ref{lemma:sum-off-diagonal} once more, we obtain
\begin{align*}
    \highISO_1 (r, \bbt) & \prec \frac{1}{N^{(\abs{\bbt} + 2)/2}}
    \cdot \frac{N^{1/2}}{\eta^{3/2}} \boundISO_1
    \cdot (\boundISO_1)^{\abs{\bbt} - 1}
    \lesssim \frac{\eta^{(\abs{\bbt}-1)/2}}{(\sqrt{N \eta^2})^{\abs{\bbt} + 1}} 
    \Big [ {1 + \frac{\paraISO_1}{(N \eta)^{1/2}}} \Big ]^{\abs{\bbt} + 1}, \\
    \highISO_2 (r, \bbt) & \prec \frac{1}{N^{(\abs{\bbt} + 2)/2}}
    \cdot \bigg \{ \frac{N^{1/2}}{\eta^{2}} \boundISO_2 \Big [ {1 + \frac{\paraISO_2}{(N \eta)^{1/2}}} \Big ]^{1/2}
    + \frac{N^{1/2}}{\eta^{5/2}}  \boundISO_1 \Big [ {1 + \frac{\paraISO_2}{(N \eta)^{1/2}}} \Big ] \bigg \} 
    \cdot \big [ {(\boundISO_1)^2 + \boundISO_2} \big ]^{\abs{\bbt}-1} \\
    & \lesssim \frac{\eta^{(\abs{\bbt}-1)/2}}{(\sqrt{N \eta^3})^{\abs{\bbt} + 1}} 
    \Big [ {1 + \frac{(\paraISO_1)^2}{N \eta} + \frac{\paraISO_2}{(N \eta)^{1/2}}} \Big ]^{\abs{\bbt} + 1}.
\end{align*}
Summarizing the estimates above, we complete the proof of Lemma \ref{lemma:high-order-contribution}.

\section{Representation as fully underlined} \label{sec:representation}

This section is devoted to the proof of Proposition \ref{prop:repre-full-underlined}. Here, we present several technical lemmas that are instrumental in deriving the representations in (\ref{eqn:representations}). Readers can navigate directly to the actual proof and refer back to these lemmas when they are utilized.

The following lemma introduces the concept of ``pre-regularizing'' a matrix. The motivation behind this operation is to ensure that the product of the pre-regularized matrix and $\bfSigma^\pm$ becomes regularized in the sense of Definition \ref{Def:regular-obs}.

\begin{lemma}[One-point pre-regularization] \label{lemma:one-point-pre-regularization}
Let $D$ be a $(M+N) \times (M+N)$ deterministic matrix with $\norm{D} \lesssim 1$. Suppose its top-left $M \times M$ and bottom-right $N \times N$ diagonal blocks are given by $D_M$ and $D_N$, respectively. For $w \in \bbd_0$, we define the one-point pre-regularization of $D$ w.r.t. $w$ as
\begin{equation}
    (D)_{w}^{\diamond} 
    := D - \frac{\angles{\Im \Gamma (w) \Sigma D_M}}{\angles{\Im \Gamma (w) \Sigma}} \bfId_M
    - \angles{D_N} \bfId_N.
    \label{def:one-point-pre-regularization}
\end{equation}
Equivalently, we can write
\begin{equation}
    (D)_{w}^{\diamond} = D - \coefOnePre_w^+ (D) \bfId^+ - \coefOnePre_w^- (D) \bfId^-,
    \qwhere
    \coefOnePre_w^\pm (D)
    := \frac{1}{2} \left ( \frac{\angles{\Im \Gamma (w) \Sigma D_M}}{\angles{\Im \Gamma (w) \Sigma}} \pm \angles{D_N} \right ).
    \label{def:coefficients-one-point-pre}
\end{equation}    
Suppose $w_1, w_2 \in \bbd_0$ such that $w \in \{ w_1, w_2 \}$. Then, $(D)^\diamond_w \bfSigma^\pm$ is $\{ w_1, w_2 \}$-regular.
\end{lemma}

Given $w_1, w_2 \in \bbd_0$, we introduce the deterministic block-diagonal matrices
\begin{equation}
    \Xi^\pm \equiv \Xi^\pm (w_1, w_2) 
    := \Pi_2^{-1} \bfId^\pm - \opSd [\Pi_1 \bfId^\pm].
    \label{def:matrix-Xi}
\end{equation}
Explicitly, by the MDE (\ref{eqn:single-MDE}) or the self-consistent equation (\ref{eqn:self-consistent-w}), we have
\begin{subequations}  \label{eqn:explicit-Xi}
\begin{align}
    \Xi^+ & = \Pi_2^{-1} \bfId^+ - \opSd [\Pi_1 \bfId^+]
    = \begin{bmatrix}
        - w_2 - (m_1 + m_2) \Sigma & 0 \\
        0 & w_1 + 1/m_1 + 1/m_2
    \end{bmatrix}, \\
    \Xi^- & = \Pi_2^{-1} \bfId^- - \opSd [\Pi_1 \bfId^-]
    = \begin{bmatrix}
        - w_2 + (m_1 - m_2) \Sigma & 0 \\
        0 & w_1 + 1/m_1 - 1/m_2
    \end{bmatrix}.
\end{align}
\end{subequations}
Moreover, recalling our definition of $\opX_{12}$ in (\ref{def:operator-X12}) and $\Pi_{12}$ in (\ref{def:chain-deter-app-length2}), it is straightforward to verify
\begin{equation}
    \opX_{12} [\Xi^\pm] = \Pi_2^{-1} \bfId^\pm
    \qand
    \Pi_{12} (\Xi^\pm) = \Pi_1 \bfId^\pm.
    \label{tmp:D-Xi-equals-Pi}
\end{equation}

In the proof of Proposition \ref{prop:repre-full-underlined}, regularization of the matrices $\Xi^\pm$ is required. Given the explicit formulas for $\Xi^\pm$ in (\ref{eqn:explicit-Xi}), this essentially translates to regularizing $\bfSigma_M = \operatorname{diag} (\Sigma, 0)$. Initially, one might consider employing the one-point regularization defined in (\ref{def:one-point-regularization}). However, such regularization of $\Xi^\pm$ would fail to guarantee the lower estimate stated in Lemma \ref{lemma:coef-Sigma-circ} below, complicating the subsequent proof. Therefore, we are compelled to consider the following two-point regularization instead. The coefficients $\coefTwo_{w_1, w_2} (\bfSigma_M)$ given in (\ref{def:coefficients-two-point}) below are specifically designed to ensure the lower estimate (\ref{bound:coefficients-Sigma-circ}), and therefore, appear somewhat artificial.

Recall that the parameters $\fkt(w_1, w_2)$ and $\fkb(w_1, w_2)$ were introduced in (\ref{def:fkt-fkb}).
\begin{lemma}[Two-point regularization] \label{lemma:two-point-regularization}
Given any two spectral parameters $w_1, w_2 \in \bbd_0$, we define the two-point regularization of $\bfSigma_M$ w.r.t. $(w_1, w_2)$ as follows (note that the order of $w_1, w_2$ matters):
\begin{equation}
    (\bfSigma_M)^\circ_{w_1, w_2} 
    := \bfSigma_M - \coefTwo_{w_1, w_2} (\bfSigma_M) \bfId_M,
    \label{def:two-point-regularization-Sigma}
\end{equation}
where the coefficient is given by
\begin{equation}
    \coefTwo_{w_1, w_2} (\bfSigma_M)
    := \mathbb{1}(\abs{w_1-w_2} \leq \tau^\prime) \cdot \frac{w_2}{w_1} \cdot \begin{cases}
        \fkb(w_1^*, w_2)^{-1}, & \text{ if } \fks_1 = \fks_2 \text{ and } d_1 \geq d_2, \\
        \fkb(w_1, w_2^*)^{-1}, & \text{ if } \fks_1 = \fks_2 \text{ and } d_1 < d_2, \\
        \fkt(w_1, w_2), & \text{ if } \fks_1 \not= \fks_2.
    \end{cases}
    \label{def:coefficients-two-point}
\end{equation}
The matrix $(\bfSigma_M)^\circ_{w_1, w_2}$ is $\{ w_1, w_2 \}$-regular.
\end{lemma}

Now, with (\ref{def:matrix-Xi}) and (\ref{def:two-point-regularization-Sigma}), we can define the two-point regularization of $\Xi^\pm$ w.r.t. $( w_1, w_2 )$ as 
\begin{equation}
    (\Xi^\pm)^\circ_{w_1,w_2} := (\pm m_1 - m_2) (\bfSigma_M)^\circ_{w_1,w_2}.
    \label{def:two-point-regularization-Xi}
\end{equation}
The proofs of Lemmas \ref{lemma:one-point-pre-regularization} and \ref{lemma:two-point-regularization} are provided in Section \ref{subsec:proof-regularity}. Here, we summarize how these two lemmas are utilized in this section. Let $w_1, w_2 \in \bbd_0$ and set $V_1 \equiv \opX_{12} [A_1]$ for some deterministic $A_1$. The following two decompositions are crucial in deriving the representations in (\ref{eqn:representations}),
\begin{subequations} \label{eqn:decompositions}
\begin{align}
    V_1 \Pi_2 & = (V_1 \Pi_2)^\diamond_{w_1} + \coefOnePre^+_{w_1} (V_1 \Pi_2) \bfId^+ + \coefOnePre^-_{w_1} (V_1 \Pi_2) \bfId^-,
    \label{eqn:decomposition-V-Pi} \\
    \Xi^\pm & = (\Xi^\pm)^\circ_{w_1,w_2} + \coefTwo^+_{w_1,w_2} (\Xi^\pm) \bfId^+ + \coefTwo^-_{w_1,w_2} (\Xi^\pm) \bfId^-,
    \label{eqn:decomposition-Xi}
\end{align}    
\end{subequations}
where the coefficients $\coefOnePre^\pm$ in (\ref{eqn:decomposition-V-Pi}) are given by (\ref{def:coefficients-one-point-pre}), while the coefficients $\coefTwo^\pm_{w_1,w_2} (\Xi^\pm)$ in (\ref{eqn:decomposition-Xi}) are defined through the equations (\ref{eqn:explicit-Xi}), (\ref{def:coefficients-two-point}) and (\ref{def:two-point-regularization-Xi}). Although one can explicitly write down the formulas of $\coefTwo^\pm_{w_1,w_2} (\Xi^\pm)$, we refrain from doing so since these explicit formulas are not relevant to our proof. We have the following two estimates concerning the coefficients in (\ref{eqn:decompositions}), the proofs of which are postponed to Section \ref{subsec:coeffcients}. Recall the operator $\opX$ that we defined in (\ref{def:operator-X12}).

\begin{lemma} \label{lemma:coef-V1-Pi2}
Let $w_1, w_2 \in \bbd_0$ and suppose that the deterministic matrix $A_1$ is $\{ w_1, w_2 \}$-regular. Then, for $V_1 \equiv \opX_{12} [A_1]$ and $\Xi^\pm \equiv \Xi^\pm (w_1, w_2)$ defined in (\ref{def:matrix-Xi}), we have
\begin{equation}
    \abs{\coefOnePre_{w_1}^+ (V_1 \Pi_2) \coefTwo_{w_1,w_2}^+ (\Xi^+)} 
    + \abs{\coefOnePre_{w_1}^- (V_1 \Pi_2) \coefTwo_{w_1,w_2}^+ (\Xi^-)}
    \lesssim \abs{w_1-w_2}^{1/2} + \eta^{1/2},
    \label{bound:coefficients-V-Pi-Xi}
\end{equation}
where we abbreviate $\eta := \eta_1 \wedge \eta_2$.
\end{lemma}

\begin{lemma} \label{lemma:coef-Sigma-circ}
Let $w_1, w_2 \in \bbd_0$, and set $\bfSigma_M^\circ \equiv (\bfSigma_M)_{w_1, w_2}^\circ$ as in (\ref{def:two-point-regularization-Sigma}). Then, with $\tilde{V}_1 \equiv \opX_{12} [\bfSigma_M^\circ]$, we have
\begin{equation}
    \abs{1 + (m_1+m_2) \coefOnePre_{w_1}^+ (\tilde{V}_1 \Pi_2)
    - (m_1-m_2) \coefOnePre_{w_1}^- (\tilde{V}_1 \Pi_2)} \gtrsim 1.
    \label{bound:coefficients-Sigma-circ}
\end{equation}
\end{lemma}

Finally, we present the following lemma, crucial for controlling resolvent chains involving $\bfId_+$. The proof of this lemma can be found in Section \ref{subsec:reduction-inequalities}.

\begin{lemma} \label{lemma:chains-with-Id}
Let $w_1, w_2, w_3 \in \bbd_0$, and assume that $A_2$ is regular w.r.t. its surrounding spectral parameters. Then, given (\ref{bound:inductive-hypothesis}), we have the following estimates uniformly in $\bbd_{\ell+1}$,
\begin{subequations} \label{bound:chains-with-Id}  
\begin{align}
    (\abs{w_1-w_2}^{1/2} + \eta^{1/2}) \cdot
    \abs{\angles{\fku, \Upsilon_{12} (\bfId^+) \fkv}} 
    & \prec \frac{1}{\sqrt{N \eta^2}}, 
    \label{bound:chains-with-Id-ISO-1} \\
    (\abs{w_1-w_2}^{1/2} + \eta^{1/2}) \cdot
    \abs{\angles{\Upsilon_{12} (\bfId^+) A_2}} 
    & \prec \frac{1 + \paraAVE_1}{N \eta}, 
    \label{bound:chains-with-Id-AVE-2} \\
    (\abs{w_1-w_2}^{1/2} + \eta^{1/2}) \cdot
    \abs{\angles{\fku, \Upsilon_{123} (\bfId^+, A_2) \fkv}} 
    & \prec \frac{1 + \paraISO_1}{\sqrt{N \eta^3}}.
    \label{bound:chains-with-Id-ISO-2}  
\end{align}    
\end{subequations}
\end{lemma}

\subsection{Averaged law with one regular matrix} \label{subsec:repre-AVE-1}

In this section, we derive the representation (\ref{eqn:representation-AVE-1}). Combining the expression (\ref{eqn:underline-HG}) for $\underline{HG}$ and the MDE (\ref{eqn:single-MDE}), we obtain the following expansion of $G$,
\begin{equation}
    G = \Pi - \Pi \underline{H G} + \Pi \opS [G - \Pi] G.
    \label{eqn:one-resolvent-expand}
\end{equation}
Let $A$ be $\{ w, w \}$-regularized and abbreviate $\opB \equiv \opB_{w, w}$ and $\opX \equiv \opX_{w, w}$. Multiplying both sides of (\ref{eqn:one-resolvent-expand}) by $V \equiv \opX [A]$ and taking the normalized trace yields
\begin{equation*}
    \angles{G V}
    = \angles{\Pi V}
    - \angles{\Pi \underline{H G} V}
    + \angles{\Pi \opS [G - \Pi] G V}.
\end{equation*}
This, together with (\ref{def:operator-B12}) and (\ref{eqn:opB-opX-inverse}), imply that
\begin{equation*}
    \angles{(G - \Pi) A}
    = \angles{\opB[G - \Pi] V}
    = - \angles{\underline{H G} V \Pi}
    + \angles{\opSd [G - \Pi] (G - \Pi) V \Pi}
    + \angles{\opSo [G - \Pi] G V \Pi}.
\end{equation*}
Next, we apply the decomposition (\ref{eqn:decomposition-V-Pi}) of $V \Pi$ to the three terms in the r.h.s. of the preceding equation. Let us abbreviate $\coefOnePre^\pm \equiv \coefOnePre^\pm_w$ when there is no ambiguity. This step leads to
\begin{align}
\begin{split}
    \angles{(G - \Pi) A}
    = & \ {- \angles{\underline{H G} (V \Pi)^\diamond}}
    + \angles{\opSd [G - \Pi] (G - \Pi) (V \Pi)^\diamond}
    + \angles{\opSo [G - \Pi] G (V \Pi)^\diamond} \\
    & + \sum_{\alpha = \pm} \coefOnePre^\alpha (V \Pi) \angles{- \underline{H G}  \bfId^\alpha
    + \opSd [G - \Pi] (G - \Pi) \bfId^\alpha
    + \opSo [G - \Pi] G \bfId^\alpha}.
\end{split} \label{tmp:V-Pi-diamond-ave-1}
\end{align}
Note that Lemma \ref{lemma:one-point-pre-regularization} ensures the $\{ w, w \}$-regularity of $(V \Pi)^\diamond \bfSigma^\pm$. Therefore, utilizing the single resolvent local laws (\ref{eqn:single-resolvent-local-law}) along with the inductive hypothesis (\ref{bound:inductive-hypothesis}), the last two terms in the first line of (\ref{tmp:V-Pi-diamond-ave-1}) can be controlled as follows,
\begin{align*}
    \angles{(G - \Pi) \bfSigma^\pm} \angles{(G - \Pi) (V \Pi)^\diamond \bfSigma^\pm}
    & = \oprecBig{\frac{1}{N \eta} \cdot \frac{\paraAVE_1}{N \eta^{1/2}}}, \\
    \frac{1}{N} \angles{(G - \Pi) \bfSigma^\pm G (V \Pi)^\diamond \bfSigma^\pm}
    & = \oprecBig{\frac{1}{N} + \frac{\boundISO_1}{N}}.
\end{align*}
We proceed to rewrite the term with angular brackets in the second line of (\ref{tmp:V-Pi-diamond-ave-1}). Let $\Xi^\pm \equiv \Xi^\pm (w, w)$ be defined as in (\ref{def:matrix-Xi}) and recall $\opX [\Xi^\pm] = \Pi^{-1} \bfId^\pm$ from (\ref{tmp:D-Xi-equals-Pi}). We have
\begin{align*}
    & \ \angles{- \underline{H G} \bfId^\alpha
    + \opSd [G - \Pi] (G - \Pi) \bfId^\alpha
    + \opSo [G - \Pi] G \bfId^\alpha} \\
    = & \ {- \angles{(I + w G + \opSd [\Pi] G
    + \opSd [G - \Pi] \Pi) \bfId^\alpha}} \\
    = & \ {- \angles{(I - \Pi^{-1} G + \opSd [G - \Pi] \Pi) \bfId^\alpha}}
    = \angles{\Pi^{-1} \opB [G - \Pi] \bfId^\alpha}
    = \angles{(G - \Pi) \Xi^\alpha},
\end{align*}
where we utilized (\ref{eqn:underline-HG}) and $\opSo[\Pi] = 0$ in the first step, the MDE (\ref{eqn:single-MDE}) in the second step, and (\ref{eqn:opB-opX-inverse}) in the last step. Plugging this back into (\ref{tmp:V-Pi-diamond-ave-1}), we obtain
\begin{equation*}
    \angles{(G - \Pi) A}
    = - \angles{\underline{H G (V \Pi)^\diamond}} 
    + \sum_{\alpha = \pm} \coefOnePre^{\alpha} (V \Pi)  
    \angles{(G - \Pi) \Xi^\alpha}
    + \oprec (\rpeAVE_1).
\end{equation*}
Now, we can apply the decomposition (\ref{eqn:decomposition-Xi}) of $\Xi^\pm$ to the r.h.s and end up with
\begin{align}
\begin{split}
    \angles{(G-\Pi) A}
    = & \ {- \angles{\underline{H G (V \Pi)^\diamond}} 
    - 2m \coefOnePre^{+} (V \Pi) \angles{(G-\Pi) \bfSigma_M^{\circ}}} \\
    & + [ {\coefOnePre^+ (V \Pi) \coefTwo^+ (\Xi^+) + \coefOnePre^- (V \Pi) \coefTwo^+ (\Xi^-)} ] \angles{G - \Pi} + \oprec (\rpeAVE_1),    
\end{split} \label{tmp:Sigma-circ-ave-1}
\end{align}
where we abbreviate $\bfSigma_M^{\circ} \equiv (\bfSigma_M)^{\circ}_{w,w}$ and implicitly used that $\angles{(G - \Pi) \bfId^-} = 0$. 

Leveraging Lemma \ref{lemma:coef-V1-Pi2}, the coefficients preceding $\angles{G - \Pi}$ in the second line of (\ref{tmp:Sigma-circ-ave-1}) can be controlled by $\eta^{1/2}$. With the single resolvent average law (\ref{eqn:single-resolvent-local-law}), we conclude that
\begin{equation}
    \angles{(G-\Pi) A}
    = - \angles{\underline{H G (V \Pi)^\diamond}} 
    - 2m \coefOnePre^{+} (V \Pi) \angles{(G-\Pi) \bfSigma_M^{\circ}}
    + \oprec (\rpeAVE_1).
    \label{tmp:underline-with-Sigma-circ}
\end{equation}
Note that the representation (\ref{tmp:underline-with-Sigma-circ}) is derived for all $\{ w, w \}$-regular matrices $A$. In particular, it is applicable to $\bfSigma_M^{\circ} \equiv (\bfSigma_M)_{w, w}^{\circ}$ since the regularity is guaranteed by Lemma \ref{lemma:two-point-regularization}. By specifying $A = \bfSigma_M^{\circ}$ in (\ref{tmp:underline-with-Sigma-circ}) and applying Lemma \ref{lemma:coef-Sigma-circ}, we can obtain the underline representation of $\angles{(G-\Pi) \bfSigma_M^\circ}$ without altering the control parameter of error terms,
\begin{equation*}
    \angles{(G-\Pi) \bfSigma_M^\circ}
    = - \frac{1}{1 + 2m \coefOnePre^{+} (\tilde{V} \Pi)} 
    \angles{\underline{H G (\tilde{V} \Pi)^\diamond}}
    + \oprec (\rpeAVE_1),
    \qwhere \tilde{V} \equiv \opX [\bfSigma_M^\circ].    
\end{equation*}
Plugging this representation of $\angles{(G-\Pi) \bfSigma_M^\circ}$ back into (\ref{tmp:underline-with-Sigma-circ}) completes the derivation of (\ref{eqn:representation-AVE-1}). Explicitly, the pre-regularized observable $B \equiv B(w, A)$ in (\ref{eqn:representation-AVE-1}) is given by
\begin{equation*}
    B = -(V \Pi)^\diamond + \frac{2m \coefOnePre^{+} (V \Pi)}{1 + 2m \coefOnePre^{+} (\tilde{V} \Pi)} (\tilde{V} \Pi)^\diamond,
    \qwhere \tilde{V} \equiv \opX [\bfSigma_M^\circ].
\end{equation*}

\subsection{Isotropic law with one regular matrix} \label{subsec:repre-ISO-1}

This section is dedicated to deriving (\ref{eqn:representation-ISO-1}). The initial step involves obtaining an expansion of the resolvent chain $G_1 A_1 G_2$, akin to (\ref{eqn:one-resolvent-expand}). Let $A_1$ be $\{ w_1, w_2 \}$-regularized and set $V_1 \equiv \opX_{12} [A_1]$. Left multiplying both sides of (\ref{eqn:one-resolvent-expand}) with $G_1 V_1$ results in
\begin{align*}
    G_1 V_1 G_2
    & = G_1 V_1 \Pi_2 - G_1 V_1 \Pi_2 \underline{H G_2} + G_1 V_1 \Pi_2 \opS [G_2 - \Pi_2] G_2 \\
    & = G_1 V_1 \Pi_2 - \underline{G_1 V_1 \Pi_2 H G_2} 
    + G_1 \opS [G_1 V_1 \Pi_2] G_2 + G_1 V_1 \Pi_2 \opS [G_2 - \Pi_2] G_2,    
\end{align*}
where we extended the underline to the whole term using that
\begin{align*}
    \underline{G_1 V_1 \Pi_2 H G_2}
    & = G_1 V_1 \Pi_2 H G_2 - \frac{1}{N} \sum_{i, \mu} G_1 V_1 \Pi_2 \dimu (\parimu G_2) 
    + \frac{1}{N} \sum_{i, \mu} G_1 \dimu G_1 V_1 \Pi_2 \dimu G_2 \\
    & = G_1 V_1 \Pi_2 \underline{H G_2}
    + G_1 \opS[G_1 V_1 \Pi_2] G_2.    
\end{align*}
By subtracting both sides of the equation with $G_1 \opSd [\Pi_1 V_1 \Pi_2] G_2$ and recalling (\ref{def:operator-X12}), we obtain\footnote{The equation (\ref{eqn:two-resolvent-expand}) also illustrates the intuition behind the definition of $\Pi_{12} (A_1)$ in (\ref{def:chain-deter-app-length2}).}
\begin{align}
\begin{split}
    G_1 A_1 G_2
    = & \ \Pi_1 V_1 \Pi_2 
    + (G_1 - \Pi_1) V_1 \Pi_2 
    - \underline{G_1 V_1 \Pi_2 H G_2} \\
    & + G_1 \opSd [(G_1 - \Pi_1) V_1 \Pi_2] G_2
    + G_1 \opSo [G_1 V_1 \Pi_2] G_2
    + G_1 V_1 \Pi_2 \opS [G_2 - \Pi_2] G_2.    
\end{split} \label{eqn:two-resolvent-expand}    
\end{align}
Now, applying the decomposition (\ref{eqn:decomposition-V-Pi}) to the terms involving $V_1 \Pi_2$, we obtain
\begin{align}
\begin{split}
    \Upsilon_{12} (A_1)
    = & \ {- \underline{G_1 (V_1 \Pi_2)^\diamond H G_2}}
    + \sum_{\alpha = \pm} \coefOnePre^\alpha (V_1 \Pi_2) \mathcal{T}_{12}^\alpha
    + (G_1 - \Pi_1) V_1 \Pi_2 \\
    & + G_1 \opSd [(G_1 - \Pi_1) (V_1 \Pi_2)^\diamond] G_2 
    + G_1 \opSo [G_1 (V_1 \Pi_2)^\diamond] G_2
    + G_1 (V_1 \Pi_2)^\diamond \opS [G_2 - \Pi_2] G_2,
\end{split} \label{tmp:V-Pi-diamond-iso-1}
\end{align}
where abbreviate $(V_1 \Pi_2)^\diamond \equiv (V_1 \Pi_2)^\diamond_{w_1}$ and used the notation
\begin{equation*}
    \mathcal{T}_{12}^\alpha := - \underline{G_1 \bfId^\alpha H G_2}
    + G_1 \opSd [(G_1 - \Pi_1) \bfId^\alpha] G_2
    + G_1 \opSo [G_1 \bfId^\alpha] G_2
    + G_1 \bfId^\alpha \opS [G_2 - \Pi_2] G_2.
\end{equation*}
We can now multiply both sides of (\ref{tmp:V-Pi-diamond-iso-1}) on the left with $\fku^*$ and on the right with $\fkv$. Note that for the first term on r.h.s. and the three terms in the second line, we can control them by $\rpeISO_1$ since
\begin{align*}
    \angles{\fku, (G_1 - \Pi_1) V_1 \Pi_2 \fkv}
    & = \oprecBig{\frac{1}{\sqrt{N \eta}}}, \\
    \angles{(G_1 - \Pi_1) (V_1 \Pi_2)^\diamond \bfSigma^\pm} 
    \angles{\fku, G_1 \bfSigma^\pm G_2 \fkv} 
    & = \oprecBig{\frac{\paraAVE_1}{N \eta^{1/2}} \cdot \frac{1}{\eta}}, \\
    \frac{1}{N} \angles{\fkv, G_2 \bfSigma^\pm G_1 (V_1 \Pi_2)^\diamond \bfSigma^\pm G_1 \fku} 
    & = \oprecBig{\frac{(\boundISO_2)^{1/2}}{N \eta}}, \\
    \angles{(G_2 - \Pi_2) \bfSigma^\pm} 
    \angles{\fku, G_1 (V_1 \Pi_2)^\diamond \bfSigma^\pm G_2 \fkv} 
    & = \oprecBig{\frac{\boundISO_1}{N \eta}}, \\
    \frac{1}{N} \angles{\fku, G_1 (V_1 \Pi_2)^\diamond \bfSigma^\pm (G_2 - \Pi_2) \bfSigma^\pm G_2 \fkv} 
    & = \oprecBig{\frac{1}{N \eta} + \frac{(\boundISO_2)^{1/2}}{N \eta}},
\end{align*}
where (\ref{bound:vec-chain-with-Id-middle}) was utilized. Additionally, we took advantage of the regularity of $(V_1 \Pi_2)^\diamond_{w_1}$ w.r.t. both $\{ w_1, w_2 \}$ and $\{ w_1, w_1 \}$ since the pre-regularization is performed at $w_1$.

Next, we proceed to rewrite $\mathcal{T}_{12}^\pm$. As per (\ref{def:second-order-renormalization}), the underlined term in $\mathcal{T}_{12}^\alpha$ can be expressed as
\begin{equation*}
    \underline{G_1 \bfId^\alpha H G_2}
    = G_1 \bfId^\alpha H G_2 
    + G_1 \opS [G_1 \bfId^\alpha] G_2 
    + G_1 \bfId^\alpha \opS [G_2] G_2.
\end{equation*}
Let $\Xi^\pm \equiv \Xi^\pm (w_1, w_2)$ be defined as in (\ref{def:matrix-Xi}). It turns out that we can rewrite $\mathcal{T}_{12}^\pm$ as follows,
\begin{align*}
    - \mathcal{T}_{12}^\alpha 
    & = G_1 \bfId^\alpha H G_2 
    + G_1 \opSd [\Pi_1 \bfId^\alpha] G_2
    + G_1 \bfId^\alpha \opS [\Pi_2] G_2 \\
    & = G_1 \bfId^\alpha
    + G_1 (w_2 \bfId^\alpha + \opSd [\Pi_1 \bfId^\alpha] + \opSd [\Pi_2] \bfId^\alpha) G_2 \\
    & = G_1 \bfId^\alpha - G_1 \Xi^\alpha G_2 
    = (G_1 - \Pi_1) \bfId^\alpha - [G_1 \Xi^\alpha G_2 - \Pi_{12} (\Xi^\alpha)]
    = \Upsilon_1 \bfId^\alpha - \Upsilon_{12} (\Xi^\alpha),
\end{align*}
where in the last line we utilized $\Pi_{12} (\Xi^\pm) = \Pi_1 \bfId^\pm$ from (\ref{tmp:D-Xi-equals-Pi}). We can plug this equation into (\ref{tmp:V-Pi-diamond-iso-1}). In summary, we have reached the following representation,
\begin{equation*}
    \angles{\fku, \Upsilon_{12} (A_1) \fkv}
    = - \angles{\fku, \underline{G_1 (V_1 \Pi_2)^\diamond H G_2} \fkv}
    + \sum_{\alpha = \pm} \coefOnePre^\alpha (V_1 \Pi_2) \angles{\fku, \Upsilon_{12} (\Xi^\alpha) \fkv} 
    + \oprec (\rpeISO_1),
\end{equation*}
where the terms $\angles{\fku, \Upsilon_1 \bfId^\alpha \fkv}$ are also absorbed into $\oprec (\rpeISO_1)$, leveraging the single resolvent isotropic law (\ref{eqn:single-resolvent-local-law}). Now, by decomposing $\Xi^\alpha$ as in (\ref{eqn:decomposition-Xi}), we obtain
\begin{align}
\begin{split}
    \angles{\fku, \Upsilon_{12} (A_1) \fkv}
    = & \ {- \angles{\fku, \underline{G_1 (V_1 \Pi_2)^\diamond H G_2} \fkv}} \\
    & - [ {(m_1+m_2) \coefOnePre^+ (V_1 \Pi_2)
    - (m_1-m_2) \coefOnePre^- (V_1 \Pi_2)} ]
    \angles{\fku, \Upsilon_{12} (\bfSigma_M^\circ) \fkv} \\
    & + [ {\coefOnePre^+ (V_1 \Pi_2) \coefTwo^+ (\Xi^+) + \coefOnePre^- (V_1 \Pi_2) \coefTwo^+ (\Xi^-)} ] \angles{\fku, \Upsilon_{12} (\bfId^+) \fkv} \\
    & + [ {\coefOnePre^+ (V_1 \Pi_2) \coefTwo^- (\Xi^+) + \coefOnePre^- (V_1 \Pi_2) \coefTwo^- (\Xi^-)} ] \angles{\fku, \Upsilon_{12} (\bfId^-) \fkv}
    + \oprec(\rpeISO_1).
\end{split} \label{tmp:Sigma-circ-iso-1}
\end{align}

Leveraging Lemma \ref{lemma:coef-V1-Pi2} as well as the estimate (\ref{bound:chains-with-Id-ISO-1}), we can dominate the third line of (\ref{tmp:Sigma-circ-iso-1}) by $\rpeISO_1$. Furthermore, using the chiral symmetry (\ref{eqn:chiral-symmetry}) and the resolvent identity (\ref{eqn:resolvent-difference-product}), the isotropic form in the last line of (\ref{tmp:Sigma-circ-iso-1}) can be controlled as follows\footnote{Strictly speaking, we also need to leverage the counterparts of (\ref{eqn:chiral-symmetry}) and (\ref{eqn:resolvent-difference-product}) for the deterministic equivalent $\Pi$. In fact, since the density $\rho$ is even, we have $m(-w) = -m(w)$, and thus (\ref{eqn:chiral-symmetry}) holds for $\Pi$ in place of $G$. For the resolvent identity (\ref{eqn:resolvent-difference-product}), its counterpart for $\Pi$ is given in (\ref{eqn:opD-as-divi-diff-2}) below, which follows directly from the definition (\ref{def:chain-deter-app-length2}).},
\begin{equation*}
    \abs{\angles{\fku, \Upsilon_{12} (\bfId^-) \fkv}}
    \lesssim \abs{\angles{\fku, (G_1 - \Pi_1) \bfId^- \fkv}} 
    + \abs{\angles{\bfId^- \fku, (G_2 - \Pi_2) \fkv}}
    \prec \frac{1}{\sqrt{N \eta}}.
\end{equation*}
In conclusion, we have
\begin{align*}
    \angles{\fku, \Upsilon_{12} (A_1) \fkv}
    = & \ {- \angles{\fku, \underline{G_1 (V_1 \Pi_2)^\diamond H G_2} \fkv}} \\
    & - [ {(m_1+m_2) \coefOnePre^+ (V_1 \Pi_2)
    - (m_1-m_2) \coefOnePre^- (V_1 \Pi_2)} ]
    \angles{\fku, \Upsilon_{12} (\bfSigma_M^\circ) \fkv} + \oprec(\rpeISO_1).
\end{align*}
Now, by repeating the steps following (\ref{tmp:underline-with-Sigma-circ}), we complete the derivation of (\ref{eqn:representation-ISO-1}). Note that for (\ref{eqn:representation-ISO-1}), the pre-regularized observable $B_1 \equiv B_1 (w_1, A_1, w_2)$ is given by
\begin{equation*}
    B_1 = -(V_1 \Pi_2)^\diamond 
    + \frac{(m_1+m_2) \coefOnePre^+ (V_1 \Pi_2) - (m_1-m_2) \coefOnePre^- (V_1 \Pi_2)}
    {1 + (m_1+m_2) \coefOnePre^+ (\tilde{V}_1 \Pi_2) - (m_1-m_2) \coefOnePre^- (\tilde{V}_1 \Pi_2)} 
    (\tilde{V}_1 \Pi_2)^\diamond, \qwhere \tilde{V}_1 \equiv \opX_{12} [\bfSigma_M^\circ].
\end{equation*}

\subsection{Averaged law with two regular matrices} \label{subsec:repre-AVE-2}

The derivation of (\ref{eqn:representation-AVE-2}) closely resembles that of (\ref{eqn:representation-ISO-1}). Therefore, we only outline the steps that differ in this case. We need to multiply both sides of (\ref{tmp:V-Pi-diamond-iso-1}) with $A_2$ and then take the normalized trace. The negligible terms can be managed as follows,
\begin{align*}
    \angles{(G_1 - \Pi_1) V_1 \Pi_2 A_2} 
    & = \oprecBig{\frac{1}{N \eta} + \frac{\paraAVE_1}{N \eta^{1/2}}}, \\
    \angles{(G_1 - \Pi_1) (V_1 \Pi_2)^\diamond \bfSigma^\pm} 
    \angles{G_1 \bfSigma^\pm G_2 A_2} 
    & = \oprecBig{\frac{\paraAVE_1}{N \eta^{1/2}} \cdot \boundISO_1}, \\
    \frac{1}{N} \angles{G_2 \bfSigma^\pm G_1 (V_1 \Pi_2)^\diamond \bfSigma^\pm G_1 A_2} 
    & = \oprecBig{\frac{\boundISO_2}{N}}, \\
    \angles{(G_2 - \Pi_2) \bfSigma^\pm} 
    \angles{G_1 (V_1 \Pi_2)^\diamond \bfSigma^\pm G_2 A_2} 
    & = \oprecBig{\frac{\boundAVE_2}{N \eta}}, \\
    \frac{1}{N} \angles{G_1 (V_1 \Pi_2)^\diamond \bfSigma^\pm (G_2 - \Pi_2) \bfSigma^\pm G_2 A_2} 
    & = \oprecBig{\frac{1}{N \eta} + \frac{\boundISO_2}{N}}.
\end{align*}
By replicating the steps that led up to equation (\ref{tmp:Sigma-circ-iso-1}), we can obtain
\begin{align*}
    \angles{\Upsilon_{12} (A_1) A_2}
    = & \ {- \angles{\underline{G_1 (V_1 \Pi_2)^\diamond H G_2 A_2}}} \\
    & - [ {(m_1+m_2) \coefOnePre^+ (V_1 \Pi_2)
    - (m_1-m_2) \coefOnePre^- (V_1 \Pi_2)} ]
    \angles{\Upsilon_{12} (\bfSigma_M^\circ) A_2} \\
    & + [ {\coefOnePre^+ (V_1 \Pi_2) \coefTwo^+ (\Xi^+) + \coefOnePre^- (V_1 \Pi_2) \coefTwo^+ (\Xi^-)} ] \angles{\Upsilon_{12} (\bfId^+) A_2} \\
    & + [ {\coefOnePre^+ (V_1 \Pi_2) \coefTwo^- (\Xi^+) + \coefOnePre^- (V_1 \Pi_2) \coefTwo^- (\Xi^-)} ] \angles{\Upsilon_{12} (\bfId^-) A_2}
    + \oprec (\rpeAVE_2).
\end{align*}
Here, the last two lines can again be incorporated into $\oprec (\rpeAVE_2)$. Specifically, for the third line, we can apply Lemma \ref{lemma:coef-V1-Pi2} and (\ref{bound:chains-with-Id-AVE-2}). While for the last line, we can utilize the chiral symmetry (\ref{eqn:chiral-symmetry}) and the resolvent identity (\ref{eqn:resolvent-difference-product}) to get
\begin{equation*}
    \abs{\angles{\Upsilon_{12} (\bfId^-) A_2}} 
    \lesssim \abs{\angles{(G_1 - \Pi_1) \bfId^- A_2}} + \abs{\angles{(G_2 - \Pi_2) A_2 \bfId^-}}
    \prec \frac{1}{N \eta} + \frac{\paraAVE_1}{N \eta^{1/2}}.
\end{equation*}
The subsequent steps, being completely analogous, are omitted here for the sake of brevity.

\subsection{Isotropic law with two regular matrices} \label{subsec:repre-ISO-2}

In this section, we focus on deriving the last representation (\ref{eqn:representation-ISO-2}) in Proposition \ref{prop:repre-full-underlined}. Again, we need to obtain an expansion of the resolvent chain $G_1 A_1 G_2 A_2 G_2$, resembling (\ref{eqn:one-resolvent-expand}) and (\ref{eqn:two-resolvent-expand}). Let $A_k$ be $\{ w_k, w_{k+1} \}$-regularized and abbreviate $V_k \equiv \opX_{w_k,w_{k+1}} [A_k]$, $k = 1,2$. Multiplying both sides of (\ref{eqn:one-resolvent-expand}) on the left by $G_1 V_1$ and on the right by $A_2 G_3$, we obtain
\begin{align*}
    G_1 V_1 G_2 A_2 G_3
    = & \ G_1 V_1 \Pi_2 A_2 G_3
    - G_1 V_1 \Pi_2 \underline{H G_2} A_2 G_3
    + G_1 V_1 \Pi_2 \opS [G_2 - \Pi_2] G_2 A_2 G_3 \\
    = & \ G_1 V_1 \Pi_2 A_2 G_3
    - \underline{G_1 V_1 \Pi_2 H G_2 A_2 G_3} \\
    & + G_1 \opS [G_1 V_1 \Pi_2] G_2 A_2 G_3
    + G_1 V_1 \Pi_2 \opS [G_2 A_2 G_3] G_3
    + G_1 V_1 \Pi_2 \opS [G_2 - \Pi_2] G_2 A_2 G_3.
\end{align*}

We subtract both sides of the equation with $G_1 \opSd [\Pi_1 V_1 \Pi_2] G_2 A_2 G_3$ and apply (\ref{def:operator-X12}), which transforms the l.h.s. into the target chain $G_1 A_1 G_2 A_2 G_3$. Furthermore, we can subtract both sides of the resulting equation with 
\begin{equation*}
    \Pi_{123} (A_1, A_2)
    = \Pi_{13} ({V_1 \Pi_2 V_2})
    = \Pi_{13} ({V_1 \Pi_2 A_2}) + \Pi_{13} ({V_1 \Pi_2 \opSd[\Pi_2 V_2 \Pi_3]}).
\end{equation*}
After rearranging, we find that these two steps lead to
\begin{align}
\begin{split}
    & \ G_1 A_1 G_2 A_2 G_3 - \Pi_{123} (A_1, A_2) \\
    = & \ G_1 V_1 \Pi_2 A_2 G_3 - \Pi_{13} (V_1 \Pi_2 A_2) 
    + G_1 V_1 \Pi_2 \opS [\Pi_2 V_2 \Pi_3] G_3 
    - \Pi_{13} ( {V_1 \Pi_2 \opS [\Pi_2 V_2 \Pi_3]} ) \\ 
    & - \underline{G_1 V_1 \Pi_2 H G_2 A_2 G_3} 
    + G_1 \opSd [(G_1 - \Pi_1) V_1 \Pi_2] G_2 A_2 G_3 
    + G_1 \opSo [G_1 V_1 \Pi_2] G_2 A_2 G_3 \\
    & + G_1 V_1 \Pi_2 \opS [G_2 A_2 G_3 - \Pi_2 V_2 \Pi_3] G_3 
    + G_1 V_1 \Pi_2 \opS [G_2 - \Pi_2] G_2 A_2 G_3.
\end{split} \label{eqn:three-resolvent-expand}
\end{align} 
Applying the decomposition (\ref{eqn:decomposition-V-Pi}) again to the terms involving $V_1 \Pi_2$, we arrive at
\begin{align}
\begin{split}
    \Upsilon_{123} (A_1, A_2)
    = & \ {- \underline{G_1 (V_1 \Pi_2)^\diamond H G_2 A_2 G_3}
    + \sum_{\alpha = \pm} \coefOnePre^\alpha (V_1 \Pi_2) \mathcal{T}_{123}^{\alpha}} \\
    & + \Upsilon_{13} (V_1 \Pi_2 A_2) 
    + \Upsilon_{13} ( {V_1 \Pi_2 \opS [\Pi_2 V_2 \Pi_3]} ) \\
    & + G_1 \opSd [(G_1 - \Pi_1) (V_1 \Pi_2)^\diamond] G_2 A_2 G_3 
    + G_1 \opSo [G_1 (V_1 \Pi_2)^\diamond] G_2 A_2 G_3 \\
    & + G_1 (V_1 \Pi_2)^\diamond \opS [\Upsilon_{23} (A_2)] G_3 
    + G_1 (V_1 \Pi_2)^\diamond \opS [G_2 - \Pi_2] G_2 A_2 G_3
\end{split} \label{tmp:V-Pi-diamond-iso-2}
\end{align} 
where we used the notation
\begin{align*}
    \mathcal{T}_{123}^\alpha = & \ {- \underline{G_1 \bfId^\alpha H G_2 A_2 G_3}
    + G_1 \opSd [(G_1 - \Pi_1) \bfId^\alpha] G_2 A_2 G_3} \\
    & + G_1 \opSo [G_1 \bfId^\alpha] G_2 A_2 G_3
    + G_1 \bfId^\alpha \opS [G_2 A_2 G_3 - \Pi_2 V_2 \Pi_3] G_3 
    + G_1 \bfId^\alpha \opS [G_2 - \Pi_2] G_2 A_2 G_3.
\end{align*}

Next, we multiply both sides of (\ref{tmp:V-Pi-diamond-iso-2}) on the left with $\fku^*$ and on the right with $\fkv$. All terms on the r.h.s. except for the first line can be dominated by $\rpeISO_2$. Specifically, for the two terms in the second line of (\ref{tmp:V-Pi-diamond-iso-2}), we can use (\ref{bound:chains-with-Id-ISO-1}) to get
\begin{align*}
    \abs{\angles{\fku, \Upsilon_{13} (V_1 \Pi_2 A_2) \fkv}}
    + \abs{\angles{\fku, \Upsilon_{13} ( {V_1 \Pi_2 \opS [\Pi_2 V_2 \Pi_3]} ) \fkv}}
    \prec \frac{1}{\sqrt{N \eta^3}} + \frac{\paraISO_1}{\sqrt{N \eta^2}}.
\end{align*}
While the terms in the last two lines of (\ref{tmp:V-Pi-diamond-iso-2}), we can leverage (\ref{bound:vec-chain-with-Id-middle}) to get
\begin{align*}
    \angles{(G_1 - \Pi_1) (V_1 \Pi_2)^\diamond \bfSigma^\pm}
    \angles{\fku, G_1 \bfSigma^\pm G_2 A_2 G_3 \fkv}
    & = \oprecBig{\frac{\paraAVE_1}{N \eta^{1/2}} \cdot \frac{(\boundISO_2)^{1/2}}{\eta}}, \\ 
    \frac{1}{N} \angles{\fkv, G_3 A_2^\top G_2 \bfSigma^\pm G_1 (V_1 \Pi_2)^\diamond \bfSigma^\pm G_1 \fku} 
    & = \oprecBig{\frac{\boundISO_2}{N \eta}}, \\ 
    \angles{\Upsilon_{23} (A_2) \bfSigma^\pm} \angles{\fku, G_1 (V_1 \Pi_2)^\diamond \bfSigma^\pm G_3 \fkv} 
    & = \oprecBig{\frac{\paraISO_1}{\sqrt{N \eta^2}} \cdot \boundISO_1}, \\ 
    \frac{1}{N} \angles{\fku, G_1 (V_1 \Pi_2)^\diamond \bfSigma^\pm (G_3 A_2^\top G_2 - \Pi_3 V_2^\top \Pi_2) \bfSigma^\pm G_3 \fkv} 
    & = \oprecBig{\frac{1}{N \eta} + \frac{(\boundISO_4)^{1/2}}{N \eta}}, \\ 
    \angles{(G_2 - \Pi_2) \bfSigma^\pm} \angles{\fku, G_1 (V_1 \Pi_2)^\diamond \bfSigma^\pm G_2 A_2 G_3 \fkv}
    & = \oprecBig{\frac{\boundISO_2}{N \eta}}, \\ 
    \frac{1}{N} \angles{\fku, G_1 (V_1 \Pi_2)^\diamond \bfSigma^\pm (G_2 - \Pi_2) \bfSigma^\pm G_2 A_2 G_3 \fkv}
    & = \oprecBig{\frac{\boundISO_2}{N \eta} + \frac{(\boundISO_2)^{1/2}}{N \eta}}. 
\end{align*}
On the other hand, we can rewrite $\mathcal{T}_{123}^\pm$ as follows. As per definition (\ref{def:second-order-renormalization}), we have
\begin{equation*}
    \underline{G_1 \bfId^\alpha H G_2 A_2 G_3} 
    = G_1 \bfId^\alpha H G_2 A_2 G_3
    + G_1 \opS [G_1 \bfId^\alpha] G_2 A_2 G_3
    + G_1 \bfId^\alpha \opS [G_2 A_2 G_3] G_3    
    + G_1 \bfId^\alpha \opS [G_2] G_2 A_2 G_3.
\end{equation*}
Consequently, 
\begin{align*}
    - \mathcal{T}_{123}^\alpha 
    & = G_1 \bfId^\alpha H G_2 A_2 G_3
    + G_1 \opSd [\Pi_1 \bfId^\alpha] G_2 A_2 G_3
    + G_1 \bfId^\alpha \opS [\Pi_2 V_2 \Pi_3] G_3 
    + G_1 \bfId^\alpha \opS [\Pi_2] G_2 A_2 G_3 \\
    & = G_1 \bfId^\alpha ( {A_2 + \opS [\Pi_2 V_2 \Pi_3]} ) G_3
    + G_1 (w_2 \bfId^\alpha + \opSd [\Pi_1 \bfId^\alpha] + \opSd [\Pi_2] \bfId^\alpha) G_2 A_2 G_3 \\
    & = G_1 \bfId^\alpha V_2 G_3
    - G_1 \Xi^\alpha G_2 A_2 G_3 \\
    & = [G_1 \bfId^\alpha V_2 G_3 - \Pi_{13} (\bfId^\alpha V_2)]
    - [G_1 \Xi^\alpha G_2 A_2 G_3 - \Pi_{123} (\Xi^\alpha, A_2)],
\end{align*}
where in the last step we have utilized (\ref{def:chain-deter-app-length3}) and (\ref{tmp:D-Xi-equals-Pi}) to get
\begin{equation*}
    \Pi_{123} (\Xi^\alpha, A_2)
    = \Pi_{13} (\opX_{12} [\Xi^\alpha] \Pi_2 V_2) 
    = \Pi_{13} (\bfId^\alpha V_2).
\end{equation*}
To summarise, we have obtained the following representation,
\begin{equation*}
    \angles{\fku, \Upsilon_{123} (A_1, A_2) \fkv}
    = - \angles{\fku, \underline{G_1 (V_1 \Pi_2)^\diamond H G_2 A_2 G_3} \fkv}
    + \sum_{\alpha = \pm} \coefOnePre^\alpha (V_1 \Pi_2) \angles{\fku, \Upsilon_{123} (\Xi^\alpha, A_2) \fkv}
    + \oprec (\rpeISO_2).
\end{equation*}

Now, the subsequent steps are largely analogous to the derivations in the preceding sections. The only point to highlight here is regarding the isotropic forms $\angles{\fku, \Upsilon_{123} (\bfId^\pm, A_2) \fkv}$, which arise from applying the decomposition (\ref{eqn:decomposition-Xi}). For the term involving $\bfId^+$, we need to utilize Lemma \ref{lemma:coef-V1-Pi2} along with (\ref{bound:chains-with-Id-ISO-2}). On the other hand, for the isotropic form involving $\bfId^-$, we have
\begin{equation*}
    \abs{\angles{\fku, \Upsilon_{123} (\bfId^-, A_2) \fkv}}
    \lesssim \abs{\angles{\fku, \Upsilon_{13} (\bfId^- A_2) \fkv}}
    + \abs{\angles{\bfId^- \fku, \Upsilon_{23} (A_2) \fkv}}
    \prec \frac{1}{\sqrt{N \eta^3}} + \frac{\paraISO_1}{\sqrt{N \eta^2}}.
\end{equation*}
Again, the first step is achieved using the chiral symmetry (\ref{eqn:chiral-symmetry}) and the resolvent identity (\ref{eqn:resolvent-difference-product}). We also used (\ref{bound:chains-with-Id-ISO-1}) and the $\{ w_2, w_3 \}$-regularity of $A_2$ in the second step.

Therefore, we have successfully completed the derivation of (\ref{eqn:representation-ISO-2}), and consequently, we have finished the proof of Proposition \ref{prop:repre-full-underlined}.

\section{Proof of the technical lemmas} \label{sec:tech-lemmas}

This section collects the proofs of the technical lemmas used in Sections \ref{sec:self-improving-inequalities} and \ref{sec:representation}. We remark that the proofs provided in this section are self-contained and do not rely on any conclusion from Sections \ref{sec:self-improving-inequalities} and \ref{sec:representation}. More specifically, the proofs in Sections \ref{subsec:proof-stieltjes-transform}, \ref{subsec:proof-regularity} and \ref{subsec:coeffcients} are purely deterministic and depend only on our assumptions regarding the aspect ratio $M/N$ and the spectrum of $\Sigma$. The proofs in Sections \ref{subsec:reduction-inequalities} are built upon the inductive hypothesis (\ref{bound:inductive-hypothesis}) and the single resolvent local laws (\ref{eqn:single-resolvent-local-law}). In other words, the arguments in this article are not circular. 

\subsection{Divided difference of the Stieltjes transform}
\label{subsec:proof-stieltjes-transform}

This section is dedicated to proving Lemma \ref{lemma:diff-estimate}. We rely on the following characterization of the edges of $\varrho$ as detailed in \cite{knowlesAnisotropicLocalLaws2017}: the edges $\fka_k$ can be identified as the critical values of the function $f$, where the function $f$ is given in (\ref{eqn:self-consistent-z}). Specifically, by letting $\fkx_k := \fkm (\fka_k)$, we have $f^\prime (\fkx_k) = 0$ and $\fka_k = f(\fkx_k)$. Moreover, according to \cite[Lemma A.3]{knowlesAnisotropicLocalLaws2017}, we have $\abs{f^{\prime\prime}(\fkx_{k})} \asymp 1$ and
\begin{equation}
    z - \fka_{k}
    = \frac{f^{\prime\prime}(\fkx_{k})}{2} (\fkm(z) - \fkx_{k})^{2} 
    + {O} (\abs{\fkm(z)-\fkx_{k}}^{3}) 
    \quad \text{ whenever } \quad 
    \abs{z-\fka_{k}} \leq c,
    \label{tmp:square-root-expansion}
\end{equation}
where $c \equiv c (\tau) > 0$ is some small constant. Therefore, by letting $c$ sufficiently small, we have
\begin{equation}
    \abs{\fkm(z) - \fkx_{k}} \asymp | {z-\fka_{k}} |^{1/2} = d(z)^{1/2},
    \quad \text{ whenever } \quad 
    \abs{z-\fka_{k}} \leq c.
    \label{tmp:square-root-behavior}
\end{equation}

We prove each of the three estimates in Lemma \ref{lemma:diff-estimate} sequentially.

\begin{proof}[Proof of (\ref{eqn:diff-estimate-same})]
We begin by examining the case where $z_1 = z_2 = z$. In this case, $\fkm [z_1, z_2]$ is given by $\fkm^\prime \equiv \partial_z \fkm (z)$. Differentiating the self-consistent equation (\ref{eqn:self-consistent-z}), we obtain
\begin{equation*}
    {\abs{\fkm}^2} / \abs{\fkm^\prime}
    = \bigg | 1 - \sum_i \frac{\fkm^2 \sigma_i^2}{(1 + \fkm \sigma_i)^2} \bigg |
    \asymp d^{1/2},
\end{equation*}
where the comparison relation has been established in the proof of \cite[Lemma A.5]{knowlesAnisotropicLocalLaws2017} for the edge regime and \cite[Lemma A.6]{knowlesAnisotropicLocalLaws2017} for the bulk regime. Now $\abs{\fkm^\prime} \asymp d^{-1/2}$ follows from $\abs{\fkm} \asymp 1$.

For what follows, we assume $z_1 \not= z_2$. For the lower bound in (\ref{eqn:diff-estimate-same}), we utilize the following expression for $\fkm [z_1, z_2]^{-1}$, which can be deduced using (\ref{eqn:self-consistent-z}),
\begin{equation*}
    \frac{z_1 - z_2}{\fkm_1 - \fkm_2}
    = \frac{1}{\fkm_1 \fkm_2} - \frac{1}{N} \operatorname{tr} \frac{\Sigma^2}{(1 + \fkm_1 \Sigma) (1 + \fkm_2 \Sigma)}.
\end{equation*}
The lower bound then follows from Lemma \ref{lemma:m-order-estimate}. Let us turn to the upper bound in (\ref{eqn:diff-estimate-same}). Using $\abs{\fkm} \asymp 1$, we readily obtain the following elementary estimate, 
\begin{equation}
    \abs{\fkm [z_1, z_2]} \lesssim \abs{z_1 - z_2}^{-1}.
    \label{tmp:basic-diff-bound}
\end{equation}
As $\fks_1 = \fks_2$ and $\fkm$ is analytic on the upper (resp. lower) half-plane, we can employ the mean value theorem. To be precise, let $\mathbb{L} \equiv \mathbb{L} (z_1, z_2) \subset \bbc$ represent the line segment from $z_1$ to $z_2$. Then
\begin{equation}
    \abs{\fkm [z_1, z_2]} \lesssim \max_{z_0 \in \mathbb{L}} \abs{\fkm^\prime (z_0)}
    \asymp \max_{z_0 \in \mathbb{L}} d(z_0)^{-1/2}.
    \label{tmp:mean-value-diff-bound}
\end{equation}
By definition of $d(z)$, there exist integers $k_1, k_2$ such that
\begin{equation}
    d_1 = \abs{z_1 - \fka_{k_1}} 
    \qand
    d_2 = \abs{z_2 - \fka_{k_2}}.
    \label{tmp:index-edge}
\end{equation}

The case where $k_1 \not= k_2$ is relatively straightforward to handle.
\begin{itemize}[leftmargin=*]
    \item If $\abs{z_1 - z_2} \geq \tau/5$, then by (\ref{tmp:basic-diff-bound}) we have $\abs{\fkm [z_1, z_2]} \lesssim \abs{z_1 - z_2}^{-1} \lesssim 1$.
    \item Suppose $\abs{z_1 - z_2} < \tau/5$. In this case, referring to the definition of $d(z)$ and $\abs{\fka_{k_1} - \fka_{k_2}} \geq \tau$ from the edge regularity condition (\ref{assump:edge-regularity}), we must have $\abs{k_1 - k_2} = 1$ and $d(z_0) = \abs{z_0 - \fka_{k_1}} \wedge \abs{z_0 - \fka_{k_2}}$. On the other hand, we have $d_1 \geq \tau/4$ and $d_2 \geq \tau/4$, as otherwise $k_1 = k_2$. Consequently, by the triangle inequality, $d_0 \geq \tau/20$. Now we have $\abs{\fkm [z_1, z_2]} \lesssim 1$ according to (\ref{tmp:mean-value-diff-bound}).
\end{itemize}

It remains to examine $k_1 = k_2 = k$. We only need to derive the upper bound $(d_1 + d_2)^{-1/2}$ since $\abs{z_1-z_2} \leq d_1+d_2$ in this case. Moreover, it suffices to consider
\begin{equation}
    \abs{z_1 - z_2} > (d_1 \vee d_2) / 2
    \qand
    d_1, d_2 \leq c,
    \label{tmp:diff-avoid-trivial}
\end{equation}
where $c > 0$ is the constant introduced in (\ref{tmp:square-root-expansion}). Indeed, if $\abs{z_1 - z_2} \leq (d_1 \vee d_2) / 2$, the upper bound follows from $\min_{z_0 \in \mathbb{L}} d_0 \geq (d_1 \vee d_2) / 2$ and (\ref{tmp:mean-value-diff-bound}). On the other hand, if $\abs{z_1 - z_2} > (d_1 \vee d_2) / 2$ and $d_1 \vee d_2 \gtrsim 1$, then the upper bound is a consequence of $\abs{z_1 - z_2} \gtrsim 1$ and (\ref{tmp:basic-diff-bound}). 

Referring to (\ref{tmp:square-root-expansion}) and (\ref{tmp:square-root-behavior}), for some sufficiently large $C \equiv C(\tau) > 0$, we have 
\begin{equation}
    \frac{\abs{f^{\prime\prime}(\fkx_{k})}}{2} 
    \big | {(\fkm_1 - \fkx_{k})^{2} - (\fkm_2 - \fkx_{k})^{2}} \big |
    \leq \abs{z_1 - z_2} + C (d_1 \vee d_2)^{3/2},
    \label{tmp:diff-expansion}
\end{equation}
as well as
\begin{equation}
    C^{-1} d_1^{1/2} \leq \abs{\fkm_1 - \fkx_k} \leq C d_1^{1/2}
    \qand
    C^{-1} d_2^{1/2} \leq \abs{\fkm_2 - \fkx_k} \leq C d_2^{1/2}.
    \label{tmp:diff-dist-m-to-edges}
\end{equation}

Let us first focus on the case where $d_1 > (4C)^2 d_2$ (a similar argument applies to $d_2 > (4C)^2 d_1$). By (\ref{tmp:diff-dist-m-to-edges}), this implies $\abs{\fkm_1 - \fkx_k} \geq 4 \abs{\fkm_2 - \fkx_k}$. Hence, by the triangle inequality, we must have
\begin{equation*}
    \abs{(\fkm_1+\fkm_2)/2 - \fkx_k} \geq \abs{\fkm_1 - \fkx_k}/8.
\end{equation*}
On the other hand, we can make $c > 0$ small enough such that $c \leq 1/C^2$. Given (\ref{tmp:diff-avoid-trivial}), this yields
\begin{equation*}
    C (d_1 \vee d_2)^{3/2} \leq C \sqrt{c} \cdot (d_1 \vee d_2)
    \leq 2 \abs{z_1 - z_2}.
\end{equation*}
Therefore, (\ref{tmp:diff-expansion}), along with the estimate $\abs{f^{\prime\prime}(\fkx_{k})} \asymp 1$, implies
\begin{equation*}
    \abs{\fkm_1 - \fkm_2} \cdot \abs{(\fkm_1+\fkm_2)/{2} - \fkx_k}
    \lesssim \abs{z_1 - z_2}.
\end{equation*}
Summarizing the estimates above, we arrive at
\begin{equation*}
    \abs{\fkm[z_1, z_2]}
    \lesssim \abs{(\fkm_1+\fkm_2)/2 - \fkx_k}^{-1}
    \lesssim \abs{\fkm_1 - \fkx_k}^{-1}
    \asymp d_1^{-1/2} \asymp (d_1 + d_2)^{-1/2}.
\end{equation*}

We now turn to the case where $d_1 \asymp d_2$. Let us denote the angle between the two segments connecting $z_1$ or $z_2$ with $\fka_k$ as $\alpha \in [0, \pi)$. We further divide the analysis into three cases. 
\begin{itemize}
    \item When $\alpha \in [0, \pi/4]$, we have $\cos \alpha \asymp 1$ and therefore $\min_{z_0 \in \mathbb{L}} d_0 \gtrsim (d_1 \wedge d_2) \cos \alpha \gtrsim d_1 + d_2$. Now the upper bound is obtained by utilizing (\ref{tmp:mean-value-diff-bound}).
    \item For $\alpha \in (\pi/4, 3\pi/4]$, we have $\sin \alpha \asymp 1$ and therefore
    \begin{equation*}
        \min_{z_0 \in \mathbb{L}} d_0 \gtrsim \frac{d_1 d_2 \sin \alpha}{\abs{z_1 - z_2}} 
        \gtrsim d_1 \sin \alpha \gtrsim d_1 \asymp d_1 + d_2.
    \end{equation*}
    \item Suppose $\alpha \in (3\pi/4, \pi)$. In this case, the estimate (\ref{tmp:mean-value-diff-bound}) might be too loose, necessitating a more delicate analysis. Without loss of generality, we can assume that
    \begin{equation*}
        z_1 = \fka_k - \nu_1 + \mathrm{i} \eta_1
        \qand
        z_2 = \fka_k + \nu_2 + \mathrm{i} \eta_2,
        \qwhere \nu_1, \nu_2 > 0 \text{ and } \operatorname{sgn} (\eta_1) = \operatorname{sgn} (\eta_2).
    \end{equation*}
    As $\alpha > 3\pi/4$, it must hold that $\nu_1 \asymp d_1 \asymp d_2 \asymp \nu_2$. Therefore, we have
    \begin{align*}
        \abs{\fkm[z_1, z_2]} 
        & = \bigg | { \int_0^1 \fkm^\prime ( (1-t) z_1 + t z_2 ) \mathrm{d} t} \bigg |
        \lesssim \int_0^1 d ((1-t) z_1 + t z_2)^{-1/2} \mathrm{d} t \\
        & \leq \int_0^1 \abs{t \nu_2 - (1-t) \nu_1}^{-1/2} \mathrm{d} t 
        \asymp (\nu_1 + \nu_2)^{-1/2} \asymp (d_1 + d_2)^{-1/2}.
    \end{align*}
\end{itemize}
This concludes the proof of the upper estimate in (\ref{eqn:diff-estimate-same}).
\end{proof}

\begin{remark}
Note that the lower estimate of $\abs{\fkm [z_1, z_2]}$ in (\ref{eqn:diff-estimate-opp}) can be further refined as follows: for sufficiently small $\tau^\prime > 0$ (depending only on $\tau$), we have for $w_1, w_2 \in \bbd (\tau, \tau^\prime)$,
\begin{equation}
    \abs{\fkm [z_1, z_2]} \gtrsim (d_1 + d_2)^{-1/2},
    \quad \text{ whenever } \quad
    \fks_1 = \fks_2 \text{ and } \abs{w_1 - w_2} \leq \tau^\prime.
    \label{eqn:diff-lower-refined}
\end{equation}
Essentially, (\ref{eqn:diff-lower-refined}) effectively refines (\ref{eqn:diff-estimate-opp}) when the two spectral parameters are sufficiently closed to the same edge, i.e. $\abs{z_1-\fka_{k}} \vee \abs{z_2-\fka_{k}} \leq c$ for some $k$, where $c > 0$ is the constant introduced in (\ref{tmp:square-root-expansion}). Indeed, in this case, one can apply the mean value theorem to $f$ to obtain, 
\begin{equation*}
    \abs{\fkm [z_1, z_2]}^{-1} \lesssim \max_{\fkm_0} \abs{f^\prime (\fkm_0)}
    \asymp \max_{\fkm_0} d(z_0)^{-1/2},
\end{equation*}
where the maximum is taken over $\fkm_0$ lying on the line segment connecting $\fkm_1$ to $\fkm_2$. Now, with $z_0 = f(\fkm_0)$, we can derive (\ref{eqn:diff-lower-refined}) through
\begin{equation*}
    \abs{f^\prime (\fkm_0)} \asymp \abs{\fkm^\prime (z_0)}^{-1} \asymp d(z_0)^{1/2} 
    \asymp \abs{\fkm_0 - \fkx_k} \lesssim \abs{\fkm_1 - \fkx_k} + \abs{\fkm_1 - \fkx_k}
    \asymp (d_1 + d_2)^{1/2}.
\end{equation*}
This refined lower estimate is utilized in the proof on Lemma \ref{lemma:coef-Sigma-circ} below.
\end{remark}

\begin{proof}[Proof of (\ref{eqn:diff-estimate-opp})]
Next, let us consider the upper estimate of $\abs{\fkm[z_1^*, z_2]}$. Here $\abs{\fkm[z_1^*, z_2]} \lesssim \abs{z_1 - z_2}$ directly follows from $\abs{m} \asymp 1$ and $\abs{z_1 - z_2} \leq \abs{z_1^* - z_2}$. We proceed to show that $\abs{\fkm[z_1^*, z_2]} \lesssim (\Im \fkm_1 / \Im z_1) \wedge (\Im \fkm_2 / \Im z_2)$. Without loss of generality, we assume $\Im \fkm_1 / \Im z_1 \leq \Im \fkm_2 / \Im z_2$. Note that
\begin{equation*}
    \abs{\fkm[z_1^*, z_2]}
    \leq \bigg | \frac{\fkm_1^* - \fkm_1}{z_1^* - z_2} \bigg | 
    + \bigg | \frac{\fkm_1 - \fkm_2}{z_1^* - z_2} \bigg |
    \leq \frac{2 \abs{\Im \fkm_1}}{\abs{z_1^* - z_2}} + \abs{\fkm [z_1, z_2]}.
\end{equation*}
For the first term on r.h.s., we can utilize $\abs{z_1^* - z_2} \geq \Im z_1$. While for the second term, we can employ (\ref{eqn:diff-estimate-same}). That is, we have
\begin{equation*}
    \abs{\Im \fkm_1} / \abs{z_1^* - z_2}
    \leq \Im \fkm_1 / \Im z_1
    \qand
    \abs{\fkm [z_1, z_2]}
    \lesssim d_1^{-1/2} \lesssim \Im \fkm_1 / \Im z_1,
\end{equation*}
where we also used (\ref{eqn:m-order-estimate}) for the second estimate. 

Hence, to finalize the proof of (\ref{eqn:diff-estimate-opp}), it remains to verify $\abs{\fkm[z_1^*, z_2]} \lesssim \abs{\fkm_1 - \fkm_2}^{-1}$. Note that
\begin{equation*}
    \abs{\fkm_1 - \fkm_2} \abs{\fkm [z_1^*, z_2]}
    \leq \frac{2 \abs{\Im \fkm_1} \abs{\fkm_1 - \fkm_2}}{\abs{z_1^* - z_2}} 
    + \frac{\abs{\fkm_1 - \fkm_2}^2}{\abs{z_1^* - z_2}}.
\end{equation*}
For the first term on r.h.s., we can employ (\ref{eqn:m-order-estimate}) and (\ref{eqn:diff-estimate-same}) to get
\begin{equation*}
    \frac{\abs{\Im \fkm_2} \abs{\fkm_1 - \fkm_2}}{\abs{z_1^* - z_2}}
    \leq \abs{\Im \fkm_2} \abs{\fkm[z_1, z_2]}
    \lesssim \frac{d_2^{1/2}}{(d_1 + d_2)^{1/2}}
    \leq 1.
\end{equation*}
We now proceed to the second term. Let $k_1, k_2$ be as defined in (\ref{tmp:index-edge}). If $k_1 \not= k_2$, based on the the discussion following (\ref{tmp:index-edge}), we have either $\abs{z_1 - z_2} \gtrsim 1$ or $d_1 \wedge d_2 \gtrsim 1$. For the former case, it is not hard to see that $\abs{\fkm_1 - \fkm_2}^2 / \abs{z_1^* - z_2} \lesssim 1$. Hence, it remains to examine the cases $d_1 \wedge d_2 \gtrsim 1$ or $k_1 = k_2$. Note that in either situation, we have $\abs{z_1^* - z_2} \lesssim d_1 \vee d_2$. Consequently,
\begin{equation*}
     \frac{\abs{\fkm_1 - \fkm_2}^2}{\abs{z_1^* - z_2}}
     \lesssim \frac{\abs{z_1 - z_2}^2}{(d_1 \vee d_2) \abs{z_1^* - z_2}}
     \leq \frac{\abs{z_1^* - z_2}}{d_1 \vee d_2}
     \lesssim 1.
\end{equation*}
In conclusion, we have completed the proof of (\ref{eqn:diff-estimate-opp}).
\end{proof}

\begin{proof}[Proof of (\ref{eqn:diff-compare-same-opp})]
As demonstrated in the proof of (\ref{eqn:diff-estimate-opp}), we have either $\abs{z_1 - z_2} \gtrsim 1$ or $\abs{z_1^* - z_2} \lesssim d_1 \vee d_2$. For the former case, we observe that
\begin{equation*}
    \abs{\fkm[z_1^*, z_2]} \gtrsim \abs{\fkm_1^* - \fkm_2} 
    \geq \abs{\fkm_1 - \fkm_2} \gtrsim \abs{\fkm[z_1, z_2]} 
    \gtrsim 1 \gtrsim \abs{z_1^* - z_2}^{1/2},
\end{equation*}
where we also used the lower estimate in (\ref{eqn:diff-estimate-same}). Now let us turn to the case where $\abs{z_1^* - z_2} \lesssim d_1 \vee d_2$. Thanks to our assumption of $w_1, w_2 \in \bbd_0$, we can leverage $\abs{\Im \fkm_\ell} \asymp d_\ell^{1/2}$ to obtain
\begin{equation}
    \abs{\fkm[z_1^*, z_2]}
    \geq \frac{\abs{\Im \fkm_1} + \abs{\Im \fkm_2}}{\abs{z_1^* - z_2}}
    \asymp \frac{(d_1 + d_2)^{1/2}}{\abs{z_1^* - z_2}}
    \gtrsim (d_1 + d_2)^{-1/2} \wedge \abs{z_1^* - z_2}^{-1/2}.
\end{equation}
The proof is complete by applying the upper estimate in (\ref{eqn:diff-estimate-same}).
\end{proof}

\subsection{Observable regularization} \label{subsec:proof-regularity}

In this section, we first prove the properties of regular matrices as stated in Lemma \ref{lemma:regular-obs-properties}. Following this, we provide proofs for the three lemmas concerning regularization: Lemmas \ref{lemma:one-point-regularization}, \ref{lemma:one-point-pre-regularization} and \ref{lemma:two-point-regularization}. Finally, we include Lemma \ref{lemma:increasing-scale} and its proof here, which essentially states that the regularity of matrices is preserved if the scales (imaginary parts) of the spectral parameters under consideration are increased. Let us start with some preliminary results that are recurrently utilized in the proof.

Recall the definition of the control parameter $\beta$ in (\ref{def:beta}). For simplicity, we use abbreviations such as $\beta_{12} \equiv \beta(w_1, w_2)$ and $\beta_{1^*2} \equiv \beta(w_1^*, w_2)$ when the spectral parameters $w_1, w_2$ are clear from the context. According to Lemmas \ref{lemma:m-order-estimate} and \ref{lemma:diff-estimate}, for $w_1, w_2 \in \bbd_0$ with $\fks_1 = \fks_2$, we have
\begin{subequations} \label{eqn:ctrl-para-bounds}
\begin{align}
    1 \gtrsim \beta_{12} & \gtrsim (d_1 + d_2)^{1/2} + \beta_{1^*2},
    \label{eqn:ctrl-para-bound-same} \\
    \abs{w_1-w_2}^{1/2} + \eta^{1/2} \gtrsim \beta_{1^*2} & \gtrsim 
    ({\Im w_1}/{\Im m_1}) + ({\Im w_2}/{\Im m_2}) + \abs{w_1 - w_2} + \abs{m_1 - m_2},
    \label{eqn:ctrl-para-bound-opp}
\end{align}
\end{subequations}
where $\eta \equiv \eta_1 \wedge \eta_2$. We also rely on the following properties concerning the deterministic matrix $\Gamma(w) = -(w + m \Sigma)^{-1}$. Suppose $w_1, w_2 \in \bbd_0$. By definition, we have the resolvent identity
\begin{subequations} \label{eqn:Gamma-properties}
\begin{equation}
    \Gamma_1 - \Gamma_2 
    = (w_1 - w_2) \Gamma_1 \Gamma_2
    + (m_1 - m_2) \Gamma_1 \Sigma \Gamma_2.
    \label{eqn:Gamma-resolvent-identity}
\end{equation}
Combined with $\norm{\Sigma}, \norm{\Gamma} \lesssim 1$, this yields the perturbation bound
\begin{equation}
    \norm{\Gamma_1 - \Gamma_2} \lesssim \abs{w_1 - w_2} + \abs{m_1 - m_2}.
    \label{eqn:Gamma-perturbation}
\end{equation}
Moreover, setting $w_1 = w_2^* = w$ in (\ref{eqn:Gamma-resolvent-identity}), we obtain
\begin{equation}
    \Im \Gamma 
    = (\Im w) \Gamma \Gamma^* + (\Im m) \Gamma \Sigma \Gamma^*.
    \label{eqn:Gamma-imaginary-part}
\end{equation}    
Lastly, using the self-consistent equation (\ref{eqn:self-consistent-w}), we can deduce that
\begin{equation}
    \angles{\Im \Gamma} = \Im m + \left ( 1 - \frac{M}{N} \right ) \frac{\Im w}{\abs{w}^2}
    \qand
    \angles{\Im \Gamma \Sigma} = \frac{\Im m}{\abs{m}^2} - \Im w.
    \label{eqn:Gamma-trace}
\end{equation}
\end{subequations}
It is noteworthy that, by Lemma \ref{lemma:m-order-estimate} and (\ref{eqn:Gamma-trace}), it is not difficult to verify that the coefficients introduced in (\ref{def:coefficients-one-point}), (\ref{def:coefficients-one-point-pre}) and (\ref{def:coefficients-two-point}) are of constant order provided $\norm{D} \lesssim 1$, i.e.
\begin{equation*}
    \abs{\coefOne_w^\pm (D)} \lesssim 1,
    \quad
    \abs{\coefTwo_{w_1, w_2}^\pm (\bfSigma_M)} \lesssim 1
    \qand
    \abs{\coefOnePre_w^\pm (D)} \lesssim 1.
\end{equation*}
We utilize this property without explicit reference to it.

\begin{proof}[Proof of Lemma \ref{lemma:regular-obs-properties}]
The statements \ref{item:ppt-reflected-para}, \ref{item:ppt-A-conjugate} and \ref{item:ppt-A-negative} can be easily verified using Definition \ref{Def:regular-obs}. Regarding the statement \ref{item:ppt-opX-bound-norm}, it is observed that due to the $\{ w_1, w_2 \}$-regularity of $A$, we have
\begin{equation*}
    \abs{\angles{\Gamma (\fkw_1) A_M \Gamma (\fkw_2) \Sigma} + \angles{A_N}} \lesssim \abs{1 - \fkt (\fkw_1, \fkw_2) \fkb (\fkw_1, \fkw_2)},
    \quad \text{ for } \fkw_1 \in \{ w_1, w_1^* \}, \fkw_2 \in \{ w_2, w_2^*\}.
\end{equation*}
Now, the statement \ref{item:ppt-opX-bound-norm} becomes evident upon applying this estimate to the second line of (\ref{eqn:explicit-X12}).

Next, we turn to statement \ref{item:ppt-constraint}. Without loss of generality, we assume $\fks_1 = \fks_2$. We show that $\abs{\angles{\Gamma_1 A_M \Gamma_2 \Sigma}} \lesssim \beta_{1 2}$ can be derived from $\abs{\angles{\Gamma_1^* A_M \Gamma_2 \Sigma}} \lesssim \beta_{1^*2}$. The other implications are entirely analogous and thus omitted. Applying $\Gamma = \Gamma^* + 2 \mathrm{i} \Im \Gamma$ and (\ref{eqn:Gamma-imaginary-part}), we get
\begin{align*}
    \abs{\angles{\Gamma_1 A_M \Gamma_2 \Sigma}} 
    & \lesssim \abs{\angles{\Gamma_1^* A_M \Gamma_2 \Sigma}} + \eta_1 + \abs{\Im m_1}
    \lesssim \beta_{1^*2} + d_1^{1/2}
    \lesssim \beta_{12},  
\end{align*}
where we also utilized (\ref{eqn:m-order-estimate}) in the second step and (\ref{eqn:ctrl-para-bound-same}) in the last step. This concludes our proof of Lemma \ref{lemma:regular-obs-properties}.
\end{proof}

\begin{proof}[Proof of Lemma \ref{lemma:one-point-regularization}]
Let us assume, without loss of generality, that $\fks_1 = \fks_2$ and $w = w_1$. For brevity, we denote the top-left and bottom-right diagonal blocks of the regularized observable $D^\circ \equiv (D)^\circ_{w_1}$ as $D^\circ_M \in \bbc^{M \times M}$ and $D^\circ_N \in \bbc^{N \times N}$, respectively. As per definition (\ref{def:one-point-regularization}), we have 
\begin{equation*}
    \angles{\Im \Gamma_1 D^\circ_M} = 0
    \qand \angles{D^\circ_N} = 0.
\end{equation*}
Therefore, by invoking Lemma \ref{lemma:regular-obs-properties} \ref{item:ppt-constraint}, to establish the regularity of $D^\circ$, it remains to verify 
\begin{equation*}
    \abs{\angles{\Gamma_1^* D_M^\circ \Gamma_2 \Sigma}} \lesssim \beta_{1^* 2}
    \qand
    \abs{\angles{\Gamma_1 D_M^\circ \Gamma_2^* \Sigma}} \lesssim \beta_{1 2^*}.
\end{equation*}
We only present the proof of the first estimate, as the second one can be proved analogously.\footnote{Actually, these two estimates make no difference when $D$ is real.} Leveraging the perturbation bound (\ref{eqn:Gamma-perturbation}) as well as the identity (\ref{eqn:Gamma-imaginary-part}), we obtain
\begin{align*}
    \angles{\Gamma_1^* D^\circ_M \Gamma_2 \Sigma}
    & = \angles{\Gamma_1^* D^\circ_M \Gamma_1 \Sigma}
    + {O} (\abs{w_1 - w_2} + \abs{m_1 - m_2}) \\
    & = {\angles{\Im \Gamma_1 D^\circ_M}} / {\Im m_1}
    + {O} ( {\abs{w_1 - w_2} + \abs{m_1 - m_2} + {\Im w_1} / {\Im m_1}} )
    = {O} (\beta_{1^*2}),
\end{align*}
where in the last step we used $\angles{\Im \Gamma_1 D^\circ_M} = 0$ and (\ref{eqn:ctrl-para-bound-opp}).
\end{proof}

\begin{proof}[Proof of Lemma \ref{lemma:one-point-pre-regularization}]
The proof of Lemma \ref{lemma:one-point-pre-regularization} is essentially the same as that of Lemma \ref{lemma:one-point-regularization}. Again, we consider $\fks_1 = \fks_2$ and $w = w_1$. We denote the two diagonal blocks of $D^\diamond \equiv (D)^\circ_{w_1}$ as $D^\diamond_M$ and $D^\diamond_N$, respectively. The estimate $\abs{\angles{\Gamma_1^* (D_M^\diamond \Sigma) \Gamma_2 \Sigma}} \lesssim \beta_{1^* 2}$ can be verified as follows,
\begin{align*}
    \angles{\Gamma_1^* (D^\diamond_M \Sigma) \Gamma_2 \Sigma}
    & = \angles{\Gamma_1^* (D^\diamond_M \Sigma) \Gamma_1 \Sigma}
    + {O} (\abs{w_1 - w_2} + \abs{m_1 - m_2}) \\
    & = \angles{\Im \Gamma_1 \Sigma D^\diamond_M} / {\Im m_1}
    + {O} ( {\abs{w_1 - w_2} + \abs{m_1 - m_2} + {\Im w_1} / {\Im m_1}} )
    = {O} (\beta_{1^*2}).
\end{align*}
The key point to note is that this time we should utilize $\angles{\Im \Gamma_1 \Sigma D^\diamond_M} = 0$ for the final step, which follows from the definition (\ref{def:one-point-pre-regularization}). This concludes the proof of Lemma \ref{lemma:one-point-pre-regularization}.
\end{proof}

\begin{proof}[Proof of Lemma \ref{lemma:two-point-regularization}]
Note that the bottom-right $N \times N$ block of $\bfSigma_M$ is null. Therefore, by referring to Definition \ref{Def:regular-obs} and Lemma \ref{lemma:one-point-regularization}, it suffices to prove that
\begin{equation}
    \abs{\coefTwo_{w_1, w_2} (\bfSigma_M) - \coefOne_{w_1}^+ (\bfSigma_M) - \coefOne_{w_1}^- (\bfSigma_M)}
    \lesssim \begin{cases}
        \beta_{1^*2} = \beta_{12^*}, & \text{ if } \fks_1 = \fks_2, \\
        \beta_{12} = \beta_{1^*2^*}, & \text{ if } \fks_1 \not= \fks_2,
    \end{cases}
    \label{bound:compare-coefOne-coefTwo}
\end{equation}
where $\coefTwo_{w_1, w_2}$ represents the coefficient for two-point regularization, while $\coefOne_{w_1}^\pm$ corresponds to the coefficients for two-point regularization. This estimate is trivial when $\abs{w_1 - w_2} \gtrsim 1$, as in this case, according to (\ref{eqn:ctrl-para-bounds}), we have $\beta (\fkw_1, \fkw_2) \gtrsim 1$, for all $\fkw_1 \in \{ w_1, w_1^* \}$ and $\fkw_2 \in \{ w_2, w_2^*\}$.

Therefore, for the subsequent analysis, we focus on the case where $w_1, w_2$ are sufficiently close such that the indicator function in (\ref{def:coefficients-two-point}) is nonzero. We also mention that the fraction $w_2/w_1$ in (\ref{def:coefficients-two-point}) is irrelevant here, as removing it will only introduce an error of ${O} (\abs{w_1 - w_2})$, which can be easily managed using (\ref{eqn:ctrl-para-bound-opp}). On the other hand, leveraging (\ref{eqn:self-consistent-w}) and (\ref{eqn:Gamma-trace}), we have
\begin{equation}
    2 \coefOne_{w}^+ (\bfSigma_M) = 2 \coefOne_{w}^- (\bfSigma_M)
    = \frac{\angles{\Im \Gamma \Sigma}}{\angles{\Im \Gamma}}
    = \frac{\Im m/\abs{m}^2 + {O} (\eta)}{\Im m + {O}(\eta)} 
    = \frac{1}{\abs{m}^2} + {O} ( {\Im w}/{\Im m} ),
    \label{eqn:ratio-Im-Gamma-Sigma}
\end{equation}
uniformly for all $w \in \bbd_0$. It turns out that when $\fks_1 = \fks_2$ and $d_1 \geq d_2$, we can use (\ref{eqn:ctrl-para-bound-opp}) to get
\begin{align*}
    \coefTwo_{w_1, w_2} (\bfSigma_M) - \coefOne_{w_1}^+ (\bfSigma_M) - \coefOne_{w_1}^- (\bfSigma_M)
    & = {1}/{(m_1^* m_2)} - {1}/{\abs{m_1}^2} + {O} ( \abs{w_1 - w_2} + {\Im w_1}/{\Im m_1} ) \\
    & = {O} ( \abs{w_1 - w_2} + \abs{m_1 - m_2} + {\Im w_1}/{\Im m_1} )
    = {O} (\beta_{1^*2}).
\end{align*}
The case where $\fks_1 = \fks_2$ but $d_1 < d_2$ can be handled similarly. As for $\fks_1 \not= \fks_2$, we have
\begin{align*}
    \coefTwo_{w_1, w_2} (\bfSigma_M) - \coefOne_{w_1}^+ (\bfSigma_M) - \coefOne_{w_1}^- (\bfSigma_M)
    & = \fkt_{12} - {1}/{\abs{m_1}^2} + {O} ( \abs{w_1 - w_2} + {\Im w_1}/{\Im m_1} ) \\
    & = \fkt_{12} - {1}/\fkb_{12} + {O} ( \abs{w_1 - w_2} + \abs{m_1 - m_2} + {\Im w_1}/{\Im m_1} ) \\
    & = {O} (\abs{1 - \fkt_{12} \fkb_{12}} + \abs{w_1 - w_2} + \abs{m_1 - m_2} + {\Im w_1}/{\Im m_1} )
    = {O} (\beta_{12}).
\end{align*}
This completes the proof of Lemma \ref{lemma:two-point-regularization}.
\end{proof}

As mentioned, we conclude this section with the following lemma and its proof.

\begin{lemma} \label{lemma:increasing-scale}
Consider $w_k = E_k + \mathrm{i} \eta_k \in \bbd_0 \cap \bbc_+$, $k = 1,2$. Assume that $A$ is $\{ w_1, w_2 \}$-regularized. Then, the matrix $A$ is also $\{ \tilde{w}_1, \tilde{w}_2 \}$-regularized for any spectral parameters $\tilde{w}_{k} = E_k + \mathrm{i} \tilde{\eta}_k \in \bbd_0$ with $\tilde{\eta}_k \geq \eta_k$. 
\end{lemma}

\begin{proof}[Proof of Lemma \ref{lemma:increasing-scale}]
It suffices to prove the lemma with $\tilde{w}_2 = w_2$. For simplicity, we use abbreviations such as $\tilde{z}_1 \equiv \tilde{w}_1^2$ and $\tilde{m}_1 \equiv m(\tilde{w}_1)$. Referring to Lemma \ref{lemma:regular-obs-properties} \ref{item:ppt-constraint}, we need to verify
\begin{equation*}
    \abs{\angles{\tilde{\Gamma}_1^* A_M \Gamma_2 \Sigma}} \lesssim \beta(\tilde{w}_1^*, w_2)
    \qand
    \abs{\angles{\tilde{\Gamma}_1 A_M \Gamma_2^* \Sigma}} \lesssim \beta(\tilde{w}_1, w_2^*).
\end{equation*}
Again, for illustration purposes, we only prove the former. By utilizing (\ref{eqn:Gamma-perturbation}), we have
\begin{equation*}
    \angles{\tilde{\Gamma}_1^* A_M \Gamma_2 \Sigma}
    = \angles{\Gamma_1^* A_M \Gamma_2 \Sigma}
    + {O} (\abs{\tilde{w}_1 - w_1} + \abs{\tilde{\fkm}_1 - \fkm_1}).    
\end{equation*}
Here we can invoke Lemma \ref{lemma:diff-estimate} to get
\begin{align*}
    \abs{\tilde{w}_1 - w_1} & \leq \tilde{\eta}_1
    \lesssim \Im \tilde{w}_1 / \Im \tilde{m}_1
    \lesssim \beta(\tilde{w}_1^*, w_2), \\
    \abs{\tilde{\fkm}_1 - \fkm_1} 
    & \lesssim \abs{\tilde{w}_1 - w_1} \cdot \abs{\fkm [\tilde{z}_1, z_1]}
    \lesssim \tilde{\eta}_1 / d(\tilde{z}_1)^{1/2} 
    \asymp \Im \tilde{w}_1 / \Im \tilde{m}_1
    \lesssim \beta(\tilde{w}_1^*, w_2).
\end{align*}
On the other hand, due to the $\{ w_1, w_2 \}$-regularity of $A$, we have $\abs{\angles{\Gamma_1^* A_M \Gamma_2 \Sigma}} \lesssim \beta(w_1^*, w_2)$. Therefore, the proof is complete once we can demonstrate that $\beta(w_1^*, w_2) \lesssim \beta(\tilde{w}_1^*, w_2)$, or equivalently, $\abs{\fkm[\tilde{z}_1^*, z_2]} \lesssim \abs{\fkm[z_1^*, z_2]}$. To establish this comparison relation, we can utilize $\abs{\tilde{z}_1^* - z_2} \gtrsim \abs{z_1^* - z_2}$, and the estimate $\abs{\tilde{\fkm}_1 - \fkm_1} \lesssim \tilde{\eta}_1 / d(\tilde{z}_1)^{1/2}$ obtained in the preceding equation, yielding
\begin{align*}
    \abs{\fkm[\tilde{z}_1^*, z_2]}
    & = \frac{\abs{\tilde{\fkm}_1^* - \fkm_2}}{\abs{\tilde{z}_1^* - z_2}}
    \leq \frac{\abs{\fkm_1^* - \fkm_2}}{\abs{\tilde{z}_1^* - z_2}}
    + \frac{\abs{\tilde{\fkm}_1 - \fkm_1}}{\abs{\tilde{z}_1^* - z_2}} \\
    & \lesssim \abs{\fkm[z_1^*, z_2]}
    + \frac{\tilde{\eta}_1 / d(\tilde{z}_1)^{1/2}}{\abs{\tilde{z}_1^* - z_2}} \\
    & \lesssim \abs{\fkm[z_1^*, z_2]}
    + \abs{\tilde{z}_1^* - z_2}^{-1/2}
    \lesssim \abs{\fkm[z_1^*, z_2]}
    + \abs{z_1^* - z_2}^{-1/2}
    \lesssim \abs{\fkm[z_1^*, z_2]}.
\end{align*}
Here, in the last line, we used $\tilde{\eta}_1 \lesssim d(\tilde{z}_1) \wedge \abs{\tilde{z}_1^* - z_2}$ for the first inequality and (\ref{eqn:diff-compare-same-opp}) for the last inequality. This completes the proof of Lemma \ref{lemma:increasing-scale}.
\end{proof}

\subsection{Coefficients in regularization} \label{subsec:coeffcients}

This section contains the proofs of the two lemmas used in Section \ref{sec:representation} concerning the coefficients introduced by regularizing $V_1 \Pi_2$ and $\Xi^\pm$.

\begin{proof}[Proof of Lemma \ref{lemma:coef-V1-Pi2}]
Recall that the coefficients $\coefTwo^\pm_{w_1, w_2} (\Xi^\pm)$ can be determined using the equations (\ref{eqn:explicit-Xi}), (\ref{def:coefficients-two-point}), (\ref{def:two-point-regularization-Xi}) and (\ref{eqn:decomposition-Xi}). As mentioned, we omit their explicit forms since they are not essential here. In fact, thanks to our bound (\ref{bound:compare-coefOne-coefTwo}) on the difference between $\coefTwo (\bfSigma_M)$ and $\coefOne^\pm (\bfSigma_M)$, it is not difficult to deduce that for $\alpha = \pm$,
\begin{equation*}
    \abs{\coefTwo^{\alpha}_{w_1, w_2} (\Xi^{\pm}) - \coefOne^{\alpha}_{w_1} (\Xi^{\pm})}
    \lesssim \abs{w_1 - w_2}^{1/2} + \eta^{1/2},
\end{equation*}
where we also used the upper bound of the control parameters $\beta$ from (\ref{eqn:ctrl-para-bound-opp}). This should not be surprising, as both the parameters $\vartheta$ and $\theta$ are designed for regularization. While we allow for certain flexibility in regularizing a matrix, as depicted by (\ref{eqn:regular-obs-condition}), the discrepancy between any two different methods of regularization should not surpass our predetermined threshold for tolerating such differences.

Returning to the proof of Lemma \ref{lemma:coef-V1-Pi2}, the discussion above implies that it is sufficient to derive (\ref{bound:coefficients-V-Pi-Xi}) with $\coefTwo^{\alpha}_{w_1, w_2}$ replaced by $\coefOne^{\alpha}_{w_1}$. Let us start with the case where $\fks_1 = \fks_2$. Referring to the explicit formula for $\Xi^\pm \equiv \Xi^\pm (w_1, w_2)$ in (\ref{eqn:explicit-Xi}) as well as the formula (\ref{eqn:ratio-Im-Gamma-Sigma}), when $\fks_1 = \fks_2$, we have
\begin{align*}
    \coefOne_{w_1}^+ (\Xi^+)
    & = {O} (\abs{w_1 - w_2} + \abs{m_1^* - m_2} + {\Im w_1} / {\Im m_1}) 
    = {O} (\abs{\Im m_1} + \abs{w_1 - w_2}^{1/2}), \\
    \coefOne_{w_1}^+ (\Xi^-)
    & = {O} (\abs{w_1 - w_2} + \abs{m_1 - m_2} )
    = {O} (\abs{w_1 - w_2}^{1/2}),
\end{align*}
where we used (\ref{eqn:diff-estimate-same}) and the fact that $\eta \lesssim \abs{\Im m}^2$ for $w \in \bbd_0$. Plugging these estimates into (\ref{bound:coefficients-V-Pi-Xi}), it remains to demonstrate
\begin{equation}
    \abs{\coefOnePre_{w_1}^+ (V_1 \Pi_2)} \cdot \abs{\Im m_1} \lesssim \abs{w_1 - w_2}^{1/2} + \eta^{1/2}.
    \label{tmp:target-sign-equal}
\end{equation}

We proceed to compute the coefficient $\coefOnePre_{w_1}^+ (V_1 \Pi_2)$. Let us denote the two diagonal blocks of $A_1$ as $A_M \in \bbc^{M \times M}$ and $A_N \in \bbc^{N \times N}$, respectively. Note that we have $\angles{A_N} = 0$ by regularity of $A_1$. Utilizing the explicit form of $\opX_{12}$ in (\ref{eqn:explicit-X12}), one can compute
\begin{subequations} \label{tmp:coef-V1-Pi2}
\begin{align}
    \angles{\Im \Pi_1 (V_1 \Pi_2) \bfSigma_M}
    & = \angles{\Im \Gamma_1 A_{M} \Gamma_2 \Sigma}
    + \frac{\fkb_{12} \angles{\Gamma_1 A_{M} \Gamma_2 \Sigma}}
    {1 - \fkt_{12} \fkb_{12}} \angles{\Im \Gamma_1 \Sigma \Gamma_2 \Sigma}, \\
    \angles{(V_1 \Pi_2) \bfId_N}
    & = \frac{m_2}{1 - \fkt_{12} \fkb_{12}}
    \angles{\Gamma_1 A_{M} \Gamma_2 \Sigma}.
\end{align} 
\end{subequations}
It turns out that we can utilize $\Im \Gamma = (\Gamma - \Gamma^*)/(2 \mathrm{i})$ to get
\begin{align}
\begin{split}
    2 \coefOnePre^+ (V_1 \Pi_2)
    & = \frac{\angles{\Im \Pi_1 (V_1 \Pi_2) \bfSigma_M}}{\angles{\Im \Gamma_1 \Sigma}} 
    + \angles{(V_1 \Pi_2) \bfId_N} \\
    & = \bigg ( \frac{1 - \fkt_{1^*2} \fkb_{12}}{2 \mathrm{i} \angles{\Im \Gamma_1 \Sigma}} + m_2 \bigg )
    \frac{\angles{\Gamma_1 A_M \Gamma_2 \Sigma}}{1 - \fkt_{12} \fkb_{12}}
    - \frac{\angles{\Gamma_1^* A_M \Gamma_2 \Sigma}}{2 \mathrm{i} \angles{\Im \Gamma_1 \Sigma}} \\
    & = \frac{1 - \fkt_{1^*2} \fkb_{1^*2}}{ 1 - \fkt_{12} \fkb_{12}}
    \frac{\angles{\Gamma_1 A_M \Gamma_2 \Sigma}}{2 \mathrm{i} \angles{\Im \Gamma_1 \Sigma}}
    + \left ( 1 - \frac{\Im m_1 \fkt_{1^*2}}{\angles{\Im \Gamma_1 \Sigma}} \right )
    \frac{m_2 \angles{\Gamma_1 A_M \Gamma_2 \Sigma}}{1 - \fkt_{12} \fkb_{12}}
    - \frac{\angles{\Gamma_1^* A_M \Gamma_2 \Sigma}}{2 \mathrm{i} \angles{\Im \Gamma_1 \Sigma}}.    
\end{split} \label{eqn:rewrite-coef-V-Pi}
\end{align}
Note that $A_1$ satisfies (\ref{eqn:regular-obs-condition}) due to its $\{ w_1, w_2 \}$-regularity. Therefore, the first and the third term in the last line of (\ref{eqn:rewrite-coef-V-Pi}) can be controlled by $\beta_{1^*2} / \abs{\Im m_1}$. As for the second term, we have
\begin{align*}
    1 - \frac{\Im m_1 \fkt_{1^*2}}{\angles{\Im \Gamma_1 \Sigma}}
    & = \bigg ( {1 - \frac{\Im m_1}{\angles{\Im \Gamma_1 \Sigma} \fkb_{1^*2}}} \bigg )
    + \frac{\Im m_1 (1 - \fkt_{1^*2} \fkb_{1^*2})}{\angles{\Im \Gamma_1 \Sigma} \fkb_{1^*2}} \\
    & = \bigg ( {1 - \frac{1/(m_1^*m_2)}{1/\abs{m_1}^2 + {O}({\Im w_1}/{\Im m_1})}} \bigg )
    + {O} (\beta_{1^*2}) \\
    & = {O} (\abs{m_1 - m_2} + {\Im w_1}/{\Im m_1} + \beta_{1^*2})
    = {O} (\beta_{1^*2}),
\end{align*}
where we also used (\ref{eqn:ctrl-para-bound-opp}) and (\ref{eqn:ratio-Im-Gamma-Sigma}). To summarize, we have established 
\begin{equation*}
    \abs{\coefOnePre_{w_1}^+ (V_1 \Pi_2)} \cdot \abs{\Im m_1} \lesssim \beta_{1^*2},    
\end{equation*}
which results in (\ref{tmp:target-sign-equal}) as one can employ the upper estimate of $\beta_{1^*2}$ from (\ref{eqn:ctrl-para-bound-opp}).

Now, it remains to verify (\ref{bound:coefficients-V-Pi-Xi}) for the case where $\fks_1 \not= \fks_2$, which is much simpler. In fact, by leveraging (\ref{eqn:diff-estimate-same}) and (\ref{eqn:ratio-Im-Gamma-Sigma}), we directly arrive at
\begin{equation*}
    \coefOne_{w_1}^+ (\Xi^\pm)
    = {O} (\abs{w_1 - w_2} + \abs{m_1^* - m_2} + {\Im w_1} / {\Im m_1}) 
    = {O} (\abs{w_1 - w_2}^{1/2} + \eta^{1/2}).
\end{equation*}
This concludes our proof of Lemma \ref{lemma:coef-V1-Pi2}.
\end{proof}

\begin{proof}[Proof of Lemma \ref{lemma:coef-Sigma-circ}]
By definition of $\coefOnePre^\pm$ in (\ref{def:coefficients-one-point-pre}), the target quantity can be rewritten as
\begin{equation*}
    \caQ := 1 + (m_1+m_2) \coefOnePre^+ (\tilde{V}_1 \Pi_2)
    - (m_1-m_2) \coefOnePre^- (\tilde{V}_1 \Pi_2)
    = 1 + m_2 \frac{\angles{\Im \Pi_1 (\tilde{V}_1 \Pi_2) \bfSigma_M}}{\angles{\Im \Gamma_1 \Sigma}} 
    + m_1 \angles{(\tilde{V}_1 \Pi_2) \bfId_N},
\end{equation*}
where $\tilde{V}_1 = \opX_{12} [\bfSigma_M^\circ]$ and $\bfSigma_M^\circ$ is regularized using (\ref{def:two-point-regularization-Sigma}). As per (\ref{def:two-point-regularization-Sigma}), the matrix $\bfSigma_M^\circ$ is null outside its top-left $M \times M$ block. By abuse of notation, here we denote this top-left block as $\Sigma^\circ$. In view of (\ref{tmp:coef-V1-Pi2}), the quantity $\caQ$ can be further expressed as 
\begin{equation*}
    \caQ = 1 + \frac{\angles{\Im \Gamma_1 \Sigma^\circ (m_2 \Gamma_2 \Sigma)}}{\angles{\Im \Gamma_1 \Sigma}}
    + \frac{\fkb_{12} \angles{\Gamma_1 \Sigma^\circ \Gamma_2 \Sigma}}
    {1 - \fkt_{12} \fkb_{12}} 
    \frac{\angles{\Im \Gamma_1 \Sigma (m_2 \Gamma_2 \Sigma)}}{\angles{\Im \Gamma_1 \Sigma}}
    + \frac{\fkb_{12} \angles{\Gamma_1 \Sigma^\circ \Gamma_2 \Sigma}}{1 - \fkt_{12} \fkb_{12}}.
\end{equation*}
Plugging the definition of $\bfSigma_M^\circ$ from (\ref{def:two-point-regularization-Sigma}) into the first term on r.h.s. of the above equation and applying the identity $I + m \Sigma \Gamma = -w \Gamma$ to the last two terms, we arrive at
\begin{equation*}
    \caQ = - \frac{\angles{\Im \Gamma_1 \Sigma (w_2 \Gamma_2)}}{\angles{\Im \Gamma_1 \Sigma}}
    - m_2 \coefTwo (\bfSigma_M) \frac{\angles{\Im \Gamma_1 \Sigma \Gamma_2}}{\angles{\Im \Gamma_1 \Sigma}}
    - \frac{\fkb_{12} \angles{\Gamma_1 \Sigma^\circ \Gamma_2 \Sigma}}{1 - \fkt_{12} \fkb_{12}} 
    \frac{\angles{\Im \Gamma_1 \Sigma (w_2 \Gamma_2)}}{\angles{\Im \Gamma_1 \Sigma}}
    =: \tilde{\caQ} \cdot \frac{\angles{\Im \Gamma_1 \Sigma \Gamma_2}}{\angles{\Im \Gamma_1 \Sigma}}. 
\end{equation*}
Therefore, to prove (\ref{bound:coefficients-Sigma-circ}), it suffices to demonstrate $\abs{\tilde{\caQ}} \gtrsim 1$. The quantity $\tilde{\caQ}$ introduced earlier can be rewritten as follows,
\begin{align*}
    \tilde{\caQ} 
    & = - w_2 - m_2 \coefTwo (\bfSigma_M)
    - \frac{w_2 \fkb_{12} \angles{\Gamma_1 \Sigma^\circ \Gamma_2 \Sigma}}{1 - \fkt_{12} \fkb_{12}} \\
    & = - w_2 \bigg ( {1 + \frac{\fkt_{12} \fkb_{12}}{1 - \fkt_{12} \fkb_{12}}} \bigg )
    - \coefTwo (\bfSigma_M) \bigg ( {m_2 - \frac{\fkb_{12} \angles{\Gamma_1 \Sigma (w_2 \Gamma_2)}}{1 - \fkt_{12} \fkb_{12}}} \bigg ) \\
    & = - \frac{w_2}{1 - \fkt_{12} \fkb_{12}}
    - \coefTwo (\bfSigma_M) \bigg ( { \frac{m_2}{1 - \fkt_{12} \fkb_{12}}
    + \frac{\fkb_{12} \angles{\Gamma_1 \Sigma}}{1 - \fkt_{12} \fkb_{12}} } \bigg ) \\
    & = - \frac{w_2}{1 - \fkt_{12} \fkb_{12}}
    - \coefTwo (\bfSigma_M) \bigg ( { \frac{m_2}{1 - \fkt_{12} \fkb_{12}}
    - \frac{w_1 \fkb_{12} + m_2}{1 - \fkt_{12} \fkb_{12}} } \bigg )
    = \frac{w_1 \fkb_{12} \coefTwo (\bfSigma_M) - w_2}{1 - \fkt_{12} \fkb_{12}},
\end{align*}
where we utilized $I + m \Sigma \Gamma = -w \Gamma$ again in the third line and (\ref{eqn:self-consistent-w}) in the last line. Now, recall the definition of the coefficient $\coefTwo (\bfSigma_M) \equiv \coefTwo_{w_1, w_2} (\bfSigma_M)$ specified in (\ref{def:coefficients-two-point}).
\begin{itemize}
    \item If $\abs{w_1 - w_2} > \tau^\prime$, then $\coefTwo (\bfSigma_M) = 0$ and thus $\abs{\tilde{\caQ}} = \abs{w_2}/\abs{1 - \fkt_{12} \fkb_{12}} \gtrsim \abs{w_1 - w_2} \gtrsim 1$. For the following, we assume $\abs{w_1 - w_2} \leq \tau^\prime$ so that the indicator function in (\ref{def:coefficients-two-point}) is nonzero.
    \item  If $\fks_1 \not= \fks_2$, then we have $\coefTwo (\bfSigma_M) = (w_2/w_1) \fkt_{12}$ and therefore $\abs{\tilde{\caQ}} = \abs{w_2} \gtrsim 1$.
    \item It remains to discuss the two cases in (\ref{def:coefficients-two-point}) corresponding to $\fks_1 = \fks_2$. Without loss of generality, let us assume $d_1 \geq d_2$. In this case, $\coefTwo (\bfSigma_M)$ is given by $(w_2/w_1) / \fkb_{1^*2}$ and we have
    \begin{equation*}
        \abs{\tilde{\caQ}} 
        \asymp \bigg | {\frac{\fkb_{12} / \fkb_{1^*2} - 1}{1 - \fkt_{12} \fkb_{12}}} \bigg |
        \asymp \bigg | {\frac{\Im m_1}{1 - \fkt_{12} \fkb_{12}}} \bigg |
        \gtrsim \frac{d_1^{1/2}}{(d_1 + d_2)^{1/2}} \gtrsim 1,
    \end{equation*}
    where we lower bounded $\abs{1 - \fkt_{12} \fkb_{12}}^{-1} \asymp \abs{\fkm [z_1, z_2]}$ by $(d_1 + d_2)^{-1/2}$ as remarked in (\ref{eqn:diff-lower-refined}). 
\end{itemize}
In conclusion, we have completed the proof of Lemma \ref{lemma:coef-Sigma-circ}.
\end{proof}

\subsection{Reduction inequalities} \label{subsec:reduction-inequalities}

The section is dedicated to proving Lemmas \ref{lemma:reduction-inequality}, \ref{lemma:sum-off-diagonal} and \ref{lemma:chains-with-Id}. In the proof of Lemma \ref{lemma:reduction-inequality}, we rely on the following estimates involving the absolute values of resolvents.

\begin{lemma}[Absolute values of resolvents] \label{lemma:abs-resolvent}
We have the following estimates uniformly in $\bbd_0$,
\begin{equation}
    \angles{\abs{G}} \prec 1
    \qand
    \angles{\fku, \abs{G} \fkv} \prec 1.
    \label{bound:single-abs-resolvent}
\end{equation}
In addition, suppose the $A_k$'s below are regular in the sense of Definition \ref{Def:regular-obs-in-chain}. Then, we have the following estimates uniformly in $\bbd_{\ell}$ provided (\ref{bound:inductive-hypothesis}),
\begin{equation}
    \abs{\angles{\abs{G_1} A_1 \abs{G_2} A_2}} \prec 1 + \frac{\paraAVE_2}{N \eta}
    \qand
    \abs{\angles{\fku, G_1 A_1 \abs{G_2} A_2 G_3 \fkv}} \prec \frac{1}{\eta} + \frac{\paraISO_2}{\sqrt{N} \eta^{3/2}}.
    \label{bound:multi-abs-resolvent}
\end{equation}
\end{lemma}

The proof of \ref{lemma:abs-resolvent} is grounded on the following integral representation for the absolute values of resolvents, which has previously appeared in \cite[Lemma 5.1]{cipolloniOptimalMultiresolventLocal2022}.

\begin{lemma} \label{lemma:integral-abs-resolvent}
Let $w = E + \mathrm{i} \eta \in \bbc \backslash \bbr$. Then the absolute value of the resolvent can be represented as
\begin{equation*}
    \abs{G(E + \mathrm{i} \eta)} = \frac{2}{\pi} \int_0^\infty \frac{\Im G (E + \mathrm{i} \sqrt{\eta^2 + t^2})}{\sqrt{\eta^2 + t^2}} \mathrm{d} t.
\end{equation*}
\end{lemma}

\begin{proof}[Proof of Lemma \ref{lemma:abs-resolvent}]
For (\ref{bound:single-abs-resolvent}), we activate Lemma \ref{lemma:integral-abs-resolvent} to get
\begin{align*}
    \angles{\abs{G(w)}}
    = \frac{2}{\pi} \bigg ( {\int_0^1 + \int_1^\infty} \bigg ) 
    \frac{\angles{\Im G (E + \mathrm{i} \sqrt{\eta^2 + t^2})}}
    {\sqrt{\eta^2 + t^2}} \mathrm{d} t
    \lesssim \int_0^1 \frac{\mathrm{d} t}{\sqrt{\eta^2 + t^2}} 
    + \int_1^\infty \frac{\mathrm{d} t}{\eta^2 + t^2} \prec 1,
\end{align*}
where we employed the single resolvent averaged law (\ref{eqn:single-resolvent-local-law}) for the local regime $0 \leq t \leq 1$ and utilized the trivial bound $\norm{G(w)} \leq 1/\abs{\Im w}$ for the global regime $t \geq 1$. Note that in the last step, we controlled the integral over $[0,1]$ by $\log (1/\eta) \leq \log N \prec 1$. The isotropic estimate in (\ref{bound:single-abs-resolvent}) can be derived completely analogously using the same approach.

Next, let us turn to the proof of (\ref{bound:multi-abs-resolvent}). In viewing of $\abs{G(w)} = \abs{G(w^*)}$ and Lemma \ref{lemma:regular-obs-properties} \ref{item:ppt-reflected-para}, we may assume, without loss of generality, that $w_k = E_k + \mathrm{i} \eta_k$, where $\eta_k > 0$. For simplicity, we also abbreviate $w_{k, t_k} := E_k + \mathrm{i} \sqrt{\eta_k^2 + t_k^2}$. Recall $\eta = \min_k \eta_k$. By invoking Lemma \ref{lemma:integral-abs-resolvent}, we obtain
\begin{align*}
    \abs{\angles{\abs{G_1} A_1 \abs{G_2} A_2}}
    \leq & \ \frac{4}{\pi^2} \int_0^\infty \int_0^\infty \frac{\abs{\angles{\Im G (w_{1, t_1}) A_1 \Im G (w_{2, t_2}) A_2}}}
    {\sqrt{\eta_1^2 + t_1^2} \sqrt{\eta_2^2 + t_2^2}} \mathrm{d} t_1 \mathrm{d} t_2 \\    
    \prec & \ \boundAVE_2 \int_0^1 \frac{\mathrm{d} t}{\sqrt{\eta^2 + t^2}} 
    \int_0^1 \frac{\mathrm{d} t}{\sqrt{\eta^2 + t^2}}
    + \int_0^1 \frac{\mathrm{d} t}{\sqrt{\eta^2 + t^2}} 
    \int_1^{N^{100}} \frac{\mathrm{d} t}{\eta^2 + t^2} \\
    & + \int_1^{N^{100}} \frac{\mathrm{d} t}{\eta^2 + t^2}
    \int_1^{N^{100}} \frac{\mathrm{d} t}{\eta^2 + t^2}
    + \int_0^\infty \frac{\mathrm{d} t}{\eta^2 + t^2} 
    \int_{N^{100}}^\infty \frac{\mathrm{d} t}{\eta^2 + t^2} 
    \lesssim \boundAVE_2 + 1.
\end{align*}
Here, in the second step, we utilized the decomposition $\Im G = (G - G^*)/(2 \mathrm{i})$ for both $\Im G$, and partition the integral into four parts:
\begin{itemize}
    \item For the local regime $t_1, t_2 \leq 1$, we note that the $A_k$'s are still regular w.r.t. their surrounding spectral parameters, thanks to Lemma \ref{lemma:increasing-scale}.
    \item For the case where $t_1 \leq 1 \leq t_2 \leq N^{100}$, we can apply the elementary inequality
    \begin{equation*}
        \abs{\angles{G(w_{1,t_1}) A_1 G(w_{2,t_2}) A_2}} 
        \leq \angles{\abs{G(w_{1,t_1})}} \| A_1 G(w_{2,t_2}) A_2\| 
        \prec 1/\sqrt{\eta_2^2 + t_2^2},
    \end{equation*}
    where we also utilized (\ref{bound:single-abs-resolvent}) to estimate $\angles{\abs{G(w_{1,t_1})}}$.
    \item When $1 \leq t_1, t_2 \leq N^{100}$, it suffices to employ the trivial bound $\| G (w) \| \leq 1/\abs{\Im w}$.
    \item Finally, for the case where $t_1 \vee t_2 \leq N^{100}$, we once again estimate the resolvents using the trivial bound. The key observation here is that the integral over $[N^{100}, \infty)$ exhibits a rapidly decreasing rate of ${O} (N^{-100})$, whereas the integral over $[0, \infty)$ can be controlled by $1/(\eta_1 \wedge \eta_2) \leq N$.
\end{itemize}

We proceed to prove the isotropic estimate in (\ref{bound:multi-abs-resolvent}), which is actually easier to handle. Again, we leverage the integral representation provided in Lemma \ref{lemma:integral-abs-resolvent} to get
\begin{align*}
    \abs{\angles{\fku, G_1 A_1 \abs{G_2} A_2 G_3 \fkv}}
    & \leq \frac{2}{\pi} \int_0^\infty \frac{\abs{\angles{\fku, G_1 A_1 \Im G (w_{2, t_2}) A_2 G_3 \fkv}}}
    {\sqrt{\eta_2^2 + t_2^2}} \mathrm{d} t_2 \\    
    & \prec \boundISO_2 \int_0^1 \frac{\mathrm{d} t}{\sqrt{\eta^2 + t^2}} 
    + \frac{1}{\eta} \int_1^{\infty} \frac{\mathrm{d} t}{\eta^2 + t^2}
    \lesssim \boundISO_2 + \frac{1}{\eta}.
\end{align*}
Here the integral is partitioned into two parts:
\begin{itemize}
    \item For the local regime $0 \leq t_2 \leq 1$, we again make use of the regularity of the $A_k$'s as guaranteed by Lemma \ref{lemma:increasing-scale}.
    \item While for the global regime $t_2 \geq 1$, we apply the elementary estimate
    \begin{equation*}
        \abs{\angles{\fku, G_1 A_1 \abs{G_2} A_2 G_3 \fkv}}
        \leq  \| G_1^* \fku \| 
        \| A_1 G (w_{2, t_2}) A_2 \| 
        \| G_3 \fkv \|
        \prec \eta^{-1} (\eta_2^2 + t_2^2)^{-1/2},
    \end{equation*}
    where we used estimates of the form
    \begin{equation*}
        \| G_1^* \fku \|^2
        = \angles{\fku, G_1 G_1^* \fku}
        = {\angles{\fku, \Im G_1 \fku}}/{\eta_1}
        \prec 1/\eta.
    \end{equation*}
\end{itemize}
This concludes our proof of Lemma \ref{lemma:abs-resolvent}.
\end{proof}

We are now ready to present the proof of Lemma \ref{lemma:reduction-inequality}. The proof we provide here closely parallels that of \cite[Lemma 3.6]{cipolloniOptimalMultiresolventLocal2022}, with the notable difference being that we avoid the need to centralize the resolvent chains, which is, in fact, unnecessary for deriving the self-improving inequalities (\ref{bound:self-improving-ineqs}). We have distilled the essential steps and compiled them here for the convenience of readers.


\begin{proof}[Proof of Lemma \ref{lemma:reduction-inequality}]
For the isotropic estimate involving $3$ regular matrices, we have
\begin{align*}
    & \abs{\angles{\fku, G_1 A_1 G_2 A_2 G_3 A_3 G_4 \fkv}} \\
    & = \bigg | {\sum_{i_2, i_3} \frac{
    \angles{\fku, G_1 A_1 \bfxi_{i_2}} \angles{\bfxi_{i_2}, A_2 \bfxi_{i_3}} 
    \angles{\bfxi_{i_3}, A_3 G_4 \fkv} }
    {(s_{i_2} - w_2) (s_{i_3} - w_3)} } \bigg | \\
    & \lesssim \bigg ( { \sum_{i_2, i_3}
    \frac{ \abs{\angles{\fku, G_1 A_1 \bfxi_{i_2}}}^2 \abs{\angles{\bfxi_{i_3}, A_3 G_4 \fkv}}^2}
    {\abs{s_{i_2} - w_2} \abs{s_{i_3} - w_3}} } \bigg )^{1/2}
    \bigg ( { \sum_{i_2, i_3} \frac{\abs{ \angles{\bfxi_{i_2}, A_2 \bfxi_{i_3}} }^2}
    {\abs{s_{i_2} - w_2} \abs{s_{i_3} - w_3}} } \bigg )^{1/2} \\
    & = \sqrt{N} \angles{\fku, G_1 A_1 \abs{G_2} A_1^* G_1^* \fku}^{1/2} 
    \angles{\fkv, G_4^* A_3^* \abs{G_3} A_3 G_4 \fkv}^{1/2} 
    \angles{\abs{G_2} A_2 \abs{G_3} A_2^*}^{1/2},
\end{align*}
where the summation indices $i_2, i_3$ range over the set $\bbj = \llbracket -M \wedge N, M \vee N \rrbracket \backslash \{ 0 \}$. Now, the first estimate in (\ref{bound:reduction-ISO-3-4}) is obtained by employing (\ref{bound:multi-abs-resolvent}). Similarly, for the isotropic estimate involving $4$ regular matrices, we have
\begin{align*}
    & \abs{\angles{\fku, G_1 A_1 G_2 A_2 G_3 A_3 G_4 A_4 G_5 \fkv}} \\
    & = \bigg | {\sum_{i_2, i_3, i_4} \frac{
    \angles{\fku, G_1 A_1 \bfxi_{i_2}} \angles{\bfxi_{i_2}, A_2 \bfxi_{i_3}} 
    \angles{\bfxi_{i_3}, A_3 \bfxi_{i_4}} \angles{\bfxi_{i_4}, A_4 G_5 \fkv} }
    {(s_{i_2} - w_2) (s_{i_3} - w_3) (s_{i_4} - w_4)} } \bigg | \\
    & \lesssim \sum_{i_2, i_3, i_4}
    \frac{ \abs{\angles{\fku, G_1 A_1 \bfxi_{i_2}}}^2 \abs{\angles{\bfxi_{i_3}, A_3 \bfxi_{i_4}}}^2
    + \abs{\angles{\bfxi_{i_4}, A_4 G_5 \fkv}}^2 \abs{\angles{\bfxi_{i_2}, A_2 \bfxi_{i_3}}}^2}
    {\abs{s_{i_2} - w_2} \abs{s_{i_3} - w_3} \abs{s_{i_4} - w_4}} \\
    & = N \angles{\fku, G_1 A_1 \abs{G_2} A_1^* G_1^* \fku} \angles{\abs{G_3} A_3 \abs{G_4} A_3^*}
    + N \angles{\fkv, G_5^* A_4^* \abs{G_4} A_4 G_5 \fkv} \angles{\abs{G_2} A_2 \abs{G_3} A_2^*}.
\end{align*}
This concludes the proof of (\ref{bound:reduction-ISO-3-4}). Regarding the estimates in (\ref{bound:vec-chain-with-Id-middle}), we only prove the last one to illustrate the idea. Note that
\begin{align*}
    \norm{G_3 A_2 G_2 A_1 G_1 \fku}^2 
    & = \angles{\fku, G_1^* A_1^* G_2^* A_2^* G_3^* G_3 A_2 G_2 A_1 G_1 \fku} \\
    & = \angles{\fku, G_1^* A_1^* G_2^* A_2^* (\Im G_3) A_2 G_2 A_1 G_1 \fku} / \eta_3
    = \oprec (\boundISO_4/\eta).
\end{align*}
Therefore, we have completed the proof of Lemma \ref{lemma:reduction-inequality}.
\end{proof}

\begin{proof}[Proof of Lemma \ref{lemma:sum-off-diagonal}]
Lemma \ref{lemma:sum-off-diagonal} follows as a straightforward corollary of (\ref{bound:vec-chain-with-Id-middle}). As an illustration, we estimate the first summation in (\ref{bound:sum-off-diagonal-1-regular}). By noting that $\bfId_M = \sum_i \bfe_i \bfe_i^\top$, we have
\begin{equation*}
    \sum_{i} \abs{\angles{\fku, G_1 A_1 G_2 D_2 \bfe_i}}^2
    = \angles{\fku, G_1 A_1 G_2 D_2 \bfId_M D_2^* G_2^* A_1^* G_1^* \fku}
    \lesssim \norm{G_2^* A_1^* G_1^* \fku}^2 
    \lesssim {\boundISO_2}/{\eta},
\end{equation*}
where we used (\ref{bound:vec-chain-with-Id-middle}) in the last step. The other estimates can be derived analogously.
\end{proof}

We now move forward to provide the proof of Lemma \ref{lemma:chains-with-Id}. By the restrictions imposed on $\bbd_{\ell}$ as stated in Lemma \ref{lemma:nested-domains} \ref{item:dist-boundary-domains}, for $w, w_1, w_2 \in \bbd_{\ell + 1}$, we have
\begin{equation}
    \int_{\partial \bbd_{\ell}} \frac{\mathrm{d} \zeta}{\abs{\zeta - w}} \lesssim \log \bigg ( {\frac{1}{\eta}} \bigg ) \prec 1
    \qand
    \int_{\partial \bbd_{\ell}} \frac{\mathrm{d} \zeta}{\abs{\zeta - w_1} \abs{\zeta - w_2}} \lesssim \frac{1}{\eta}.
    \label{bound:integral-abs-resolvent}
\end{equation}

\begin{proof}[Proof of Lemma \ref{lemma:chains-with-Id}]
We begin by deriving the expressions for $\Pi_{12} (\bfId^+)$ and $\Pi_{123} (\bfId^+, A_2)$. For brevity, we assume $w_1 \not= w_2$. The case of $w_1 = w_2$ can be obtained by letting $w_2 \to w_1$. Using (\ref{eqn:explicit-X12}), it is not difficult to verify that 
\begin{equation}
    \Pi_{12} (\bfId^+) = \Pi_1 \opX_{12} [\bfId^+] \Pi_2 = \frac{\Pi_1 - \Pi_2}{w_1 - w_2}.
    \label{eqn:opD-as-divi-diff-2}
\end{equation}
Therefore, as per definition (\ref{def:chain-deter-app-length3}), we have
\begin{equation}
    \Pi_{123} (\bfId^+, A_2)
    = \opB_{13}^{-1} [\Pi_{12} (\bfId^+) V_2 \Pi_3]
    = \frac{\opB_{13}^{-1} [(\Pi_1 - \Pi_2) V_2 \Pi_3]}{w_1 - w_2}
    = \frac{\Pi_{13} (A_2) - \Pi_{23} (A_2)}{w_1 - w_2},
    \label{eqn:opD-as-divi-diff-3}
\end{equation}
where in the last step we used (\ref{def:operator-B12}) and (\ref{def:operator-X12}) to get
\begin{align*}
    \opB_{13} [\Pi_{13} (A_2) - \Pi_{23} (A_2)]
    & = \Pi_1 A_2 \Pi_3 - \opB_{13} [\Pi_2 V_2 \Pi_3] \\
    & = \Pi_1 (A_2 + \opSd[\Pi_2 V_2 \Pi_3]) \Pi_3 - \Pi_2 V_2 \Pi_3 
    = (\Pi_1 - \Pi_2) V_2 \Pi_3.
\end{align*}

Consider $w_1, w_2 \in \bbd_{\ell + 1}$. Let us assume $\fks_1 = \fks_2$. The case where the spectral parameters reside in different half-planes is actually easier to handle, since in that case, we can always employ (\ref{eqn:resolvent-difference-product}) and lower bound the corresponding denominator $(w_1 - w_2)$ by $\eta$.

For the averaged forms in (\ref{bound:chains-with-Id-AVE-2}), we employ (\ref{eqn:contour-integral-product}) and (\ref{eqn:opD-as-divi-diff-2}) to rewrite it as
\begin{align*}
    \angles{\Upsilon_{12} (\bfId^+) A_2}
    & = \angles{[G_1 G_2 - \Pi_{12} (\bfId^+)] A_2}
    = \int_{\partial \bbd_{\ell}} \frac{\angles{\Upsilon(\zeta) A_2}}{(\zeta - w_1) (\zeta - w_2)} \mathrm{d} \zeta \\
    & = \int_{\partial \bbd_{\ell}} \frac{\angles{\Upsilon(\zeta) (A_2)^\circ_{\zeta}}}{(\zeta - w_1) (\zeta - w_2)} \mathrm{d} \zeta
    + \int_{\partial \bbd_{\ell}} \frac{\coefOne^{+}_{\zeta} (A_2) \angles{\Upsilon(\zeta)}}{(\zeta - w_1) (\zeta - w_2)} \mathrm{d} \zeta.
\end{align*}
With the assistance of (\ref{bound:integral-abs-resolvent}), the first integral above can be controlled as follows,
\begin{align*}
    (\abs{w_1-w_2}^{1/2} + \eta^{1/2}) &
    \int_{\partial \bbd_{\ell}} \bigg | {\frac{\angles{\Upsilon(\zeta) (A_2)^\circ_{\zeta}}}{(\zeta - w_1) (\zeta - w_2)}} \bigg | \mathrm{d} \zeta \\
    & \lesssim \frac{\paraAVE_1}{N \eta^{1/2}} \int_{\partial \bbd_{\ell}} 
    \frac{\abs{\zeta - w_1}^{1/2} + \abs{\zeta - w_2}^{1/2} + \eta^{1/2}}{\abs{\zeta - w_1} \abs{\zeta - w_2}} \mathrm{d} \zeta
    \prec \frac{\paraAVE_1}{N \eta}.
\end{align*}
As for the second integral, we utilize the single resolvent averaged law (\ref{eqn:single-resolvent-local-law}) alongside with the following perturbation bound,
\begin{align*}
    2 \abs{\coefOne^{+}_{\zeta} (A_2)}
    & = \abs{\angles{\Im \Gamma (\zeta) A_2} / {\angles{\Im \Gamma (\zeta)}}}  
    \lesssim \abs{\angles{\Gamma (\zeta)^* A_2 \Gamma (\zeta) \Sigma}} + \Im \zeta / \Im m(\zeta) \\
    & \lesssim \abs{\angles{\Gamma_1^* A_2 \Gamma_2 \Sigma}} + \abs{w_1 - \zeta} + \abs{w_2 - \zeta} + \abs{m_1 - m(\zeta)} + \abs{m_2 - m(\zeta)} + \Im \zeta / \Im m(\zeta) \\
    & \lesssim \beta_{1^* 2} + \abs{\zeta - w_1}^{1/2} + \abs{\zeta - w_2}^{1/2} + \eta^{1/2} \\
    & \lesssim \abs{\zeta - w_1}^{1/2} + \abs{\zeta - w_2}^{1/2} + \eta^{1/2}.
\end{align*}
It turns out that the second integral can be dominated as follows,
\begin{align*}
    (\abs{w_1-w_2}^{1/2} + \eta^{1/2}) &
    \int_{\partial \bbd_{\ell}} \bigg | {\frac{\coefOne^{+}_{\zeta} (A_2) \angles{\Upsilon(\zeta)}}{(\zeta - w_1) (\zeta - w_2)}} \bigg | \mathrm{d} \zeta \\
    & \lesssim \frac{1}{N \eta} \int_{\partial \bbd_{\ell}} 
    \frac{\abs{\zeta - w_1} + \abs{\zeta - w_2} + \eta}{\abs{\zeta - w_1} \abs{\zeta - w_2}} \mathrm{d} \zeta
    \prec \frac{1}{N \eta}.    
\end{align*}
Summarizing these estimates on the two integrals leads to (\ref{bound:chains-with-Id-AVE-2}). Next, let us turn to the isotropic form in (\ref{bound:chains-with-Id-ISO-2}). Again, we leverage (\ref{eqn:contour-integral-product}) and (\ref{eqn:opD-as-divi-diff-3}) to obtain
\begin{align*}
    \abs{\angles{\fku, \Upsilon_{123} (\bfId^+, A_2) \fkv}} 
    & = \angles{\fku, [G_1 G_2 A_2 G_3 - \Upsilon_{123} (\bfId^+, A_2)] \fkv}
    = \int \frac{\angles{\fku, \Upsilon(\zeta, A_2, w_3) \fkv}}{(\zeta - w_1) (\zeta - w_2)} \mathrm{d} \zeta \\
    & = \int \frac{\angles{\fku, \Upsilon(\zeta, (A_2)^\circ_{w_3}, w_3) \fkv}}{(\zeta - w_1)(\zeta - w_2)} \mathrm{d} \zeta
    + \sum_{\alpha = \pm} \coefOne^{\alpha}_{w_1} (A_1) \int \frac{\angles{\fku, \Upsilon(\zeta, \bfId^\alpha, w_3) \fkv}}{(\zeta - w_1)(\zeta - w_2)} \mathrm{d} \zeta.
\end{align*} 
We mention that one can alternatively regularize $A_2$ at $\zeta$. However, this does not make any essential difference to the argument presented here. The first integral above is relatively easy to estimate. Let us focus on the integral corresponding to $\alpha = +$. The perturbation bound we shall use here is
\begin{equation*}
    \abs{\coefOne^+_{w_3} (A_2)} 
    \lesssim \abs{w_2 - w_3}^{1/2} + \eta^{1/2}
    \lesssim \abs{\zeta - w_2}^{1/2} + \abs{\zeta - w_3}^{1/2} + \eta^{1/2}.
\end{equation*}
It turns out that, we can utilize (\ref{eqn:single-resolvent-local-law}) and (\ref{eqn:resolvent-difference-product}) to obtain
\begin{align*}
    (\abs{w_1-w_2}^{1/2} + \eta^{1/2}) & \int \bigg | { \frac{\coefOne^{+}_{w_3} (A_2) \angles{\fku, [\Upsilon (\zeta) - \Upsilon (w_3)] \fkv}}{(\zeta - w_1)(\zeta - w_2)(\zeta - w_3)}} \bigg |
    \mathrm{d} \zeta \\
    \prec & \ \frac{1}{\sqrt{N \eta}}
    \int \frac{\abs{\zeta - w_2} + \abs{\zeta - w_3} + \eta}{\abs{\zeta - w_1} \abs{\zeta - w_2} \abs{\zeta - w_3}} \mathrm{d} \zeta
    \prec \frac{1}{\sqrt{N \eta^3}}.
\end{align*}
As for the integral corresponding to $\alpha = -$, it can be handled straightforwardly using the chiral symmetry (\ref{eqn:chiral-symmetry}). This concludes our proof of (\ref{bound:chains-with-Id-ISO-2}). The derivation of (\ref{bound:chains-with-Id-ISO-1}) is omitted here as the argument follows the same structure. Therefore, we have completed the proof of Lemma \ref{lemma:chains-with-Id}.
\end{proof}


\printbibliography

@article{silversteinAnalysisLimitingSpectral1995,
  title = {Analysis of the {{Limiting Spectral Distribution}} of {{Large Dimensional Random Matrices}}},
  author = {Silverstein, Jack W. and Choi, S. I.},
  year = {1995},
  journal = {Journal of Multivariate Analysis},
  volume = {54},
  pages = {295--309},
  doi = {10.1006/jmva.1995.1058}
}

@misc{cipolloniOptimalLowerBound2023,
  title = {Optimal {{Lower Bound}} on {{Eigenvector Overlaps}} for Non-{{Hermitian Random Matrices}}},
  author = {Cipolloni, Giorgio and Erd{\H o}s, L{\'a}szl{\'o} and Henheik, Joscha and Schr{\"o}der, Dominik},
  year = {2023},
  eprint = {2301.03549},
  primaryclass = {math-ph},
  publisher = {{arXiv}},
  doi = {10.48550/arXiv.2301.03549},
  archiveprefix = {arxiv}
}

@article{cipolloniOptimalMultiresolventLocal2022,
  title = {Optimal Multi-Resolvent Local Laws for {{Wigner}} Matrices},
  author = {Cipolloni, Giorgio and Erd{\H o}s, L{\'a}szl{\'o} and Schr{\"o}der, Dominik},
  year = {2022},
  journal = {Electronic Journal of Probability},
  volume = {27},
  pages = {1--38},
  publisher = {{Institute of Mathematical Statistics and Bernoulli Society}},
  doi = {10.1214/22-EJP838}
}

@article{knowlesAnisotropicLocalLaws2017,
  title = {Anisotropic Local Laws for Random Matrices},
  author = {Knowles, Antti and Yin, Jun},
  year = {2017},
  journal = {Probability Theory and Related Fields},
  volume = {169},
  pages = {257--352},
  publisher = {{Springer Berlin Heidelberg}},
  doi = {10.1007/s00440-016-0730-4},
  copyright = {2016 Springer-Verlag Berlin Heidelberg}
}

@article{baoUniversalityLargestEigenvalue2015,
  title = {Universality for the Largest Eigenvalue of Sample Covariance Matrices with General Population},
  author = {Bao, Zhigang and Pan, Guangming and Zhou, Wang},
  year = {2015},
  journal = {The Annals of Statistics},
  volume = {43},
  pages = {382--421},
  publisher = {{Institute of Mathematical Statistics}},
  doi = {10.1214/14-AOS1281}
}

@Article{erdosLocalSemicircleLaw2013,
  author    = {Erd{\H o}s, L{\'a}szl{\'o} and Knowles, Antti and Yau, Horng-Tzer and Yin, Jun},
  journal   = {Electronic Journal of Probability},
  title     = {The Local Semicircle Law for a General Class of Random Matrices},
  year      = {2013},
  volume    = {18},
  pages     = {1--58},
  doi       = {10.1214/EJP.v18-2473},
  mrnumber  = {MR3068390},
  publisher = {{The Institute of Mathematical Statistics and the Bernoulli Society}},
  zmnumber  = {1373.15053},
}

@misc{benaych-georgesLecturesLocalSemicircle2018,
  title = {Lectures on the Local Semicircle Law for {{Wigner}} Matrices},
  author = {{Benaych-Georges}, Florent and Knowles, Antti},
  year = {2018},
  eprint = {1601.04055},
  primaryclass = {math-ph},
  publisher = {arXiv},
  doi = {10.48550/arXiv.1601.04055},
  archiveprefix = {arxiv}
}

@misc{erdosMatrixDysonEquation2019,
  title = {The Matrix {{Dyson}} Equation and Its Applications for Random Matrices},
  author = {Erd{\H o}s, L{\'a}szl{\'o}},
  year = {2019},
  eprint = {1903.10060},
  primaryclass = {math},
  publisher = {arXiv},
  doi = {10.48550/arXiv.1903.10060},
  archiveprefix = {arxiv}
}

@article{lytovaCentralLimitTheorem2009,
  title = {Central Limit Theorem for Linear Eigenvalue Statistics of Random Matrices with Independent Entries},
  author = {Lytova, Anna and Pastur, Leonid},
  year = {2009},
  journal = {The Annals of Probability},
  volume = {37},
  pages = {1778--1840},
  publisher = {Institute of Mathematical Statistics},
  doi = {10.1214/09-AOP452}
}

@Article{cipolloniEigenstateThermalizationHypothesis2021,
  author    = {Cipolloni, Giorgio and Erd{\H o}s, L{\'a}szl{\'o} and Schr{\"o}der, Dominik},
  journal   = {Communications in Mathematical Physics},
  title     = {Eigenstate {{Thermalization Hypothesis}} for {{Wigner Matrices}}},
  year      = {2021},
  pages     = {1005--1048},
  volume    = {388},
  copyright = {2021 The Author(s)},
  doi       = {10.1007/s00220-021-04239-z},
  publisher = {Springer Berlin Heidelberg},
}

@article{cipolloniThermalisationWignerMatrices2022,
  title = {Thermalisation for {{Wigner}} Matrices},
  author = {Cipolloni, Giorgio and Erd{\H o}s, L{\'a}szl{\'o} and Schr{\"o}der, Dominik},
  year = {2022},
  journal = {Journal of Functional Analysis},
  volume = {282},
  pages = {109394},
  doi = {10.1016/j.jfa.2022.109394}
}

@article{cipolloniRankuniformLocalLaw2022,
  title = {Rank-Uniform Local Law for {{Wigner}} Matrices},
  author = {Cipolloni, Giorgio and Erd{\H o}s, L{\'a}szl{\'o} and Schr{\"o}der, Dominik},
  year = {2022},
  journal = {Forum of Mathematics, Sigma},
  volume = {10},
  pages = {e96},
  publisher = {Cambridge University Press},
  doi = {10.1017/fms.2022.86}
}

@article{cipolloniGaussianFluctuationsEquipartition2023,
  title = {Gaussian Fluctuations in the Equipartition Principle for {{Wigner}} Matrices},
  author = {Cipolloni, Giorgio and Erd{\H o}s, L{\'a}szl{\'o} and Henheik, Joscha and Kolupaiev, Oleksii},
  year = {2023},
  journal = {Forum of Mathematics, Sigma},
  volume = {11},
  pages = {e74},
  doi = {10.1017/fms.2023.70}
}

@misc{cipolloniEigenstateThermalisationEdge2023,
  title = {Eigenstate Thermalisation at the Edge for {{Wigner}} Matrices},
  author = {Cipolloni, Giorgio and Erd{\H o}s, L{\'a}szl{\'o} and Henheik, Joscha},
  year = {2023},
  eprint = {2309.05488},
  primaryclass = {math-ph},
  publisher = {arXiv},
  doi = {10.48550/arXiv.2309.05488},
  archiveprefix = {arxiv}
}

@misc{erdosEigenstateThermalizationHypothesis2024,
  title = {Eigenstate {{Thermalization Hypothesis}} for {{Wigner-type Matrices}}},
  author = {Erd{\H o}s, L{\'a}szl{\'o} and Riabov, Volodymyr},
  year = {2024},
  eprint = {2403.10359},
  primaryclass = {math},
  publisher = {arXiv},
  doi = {10.48550/arXiv.2403.10359},
  archiveprefix = {arxiv}
}

@article{fanOverviewEstimationLarge2016,
  title = {An Overview of the Estimation of Large Covariance and Precision Matrices},
  author = {Fan, Jianqing and Liao, Yuan and Liu, Han},
  year = {2016},
  journal = {The Econometrics Journal},
  volume = {19},
  pages = {C1-C32},
  doi = {10.1111/ectj.12061},
  copyright = {{\copyright} 2016 Royal Economic Society.}
}

@article{caiEstimatingStructuredHighdimensional2016,
  title = {Estimating Structured High-Dimensional Covariance and Precision Matrices: {{Optimal}} Rates and Adaptive Estimation},
  shorttitle = {Estimating Structured High-Dimensional Covariance and Precision Matrices},
  author = {Cai, T. Tony and Ren, Zhao and Zhou, Harrison H.},
  year = {2016},
  journal = {Electronic Journal of Statistics},
  volume = {10},
  pages = {1--59},
  publisher = {{Institute of Mathematical Statistics and Bernoulli Society}},
  doi = {10.1214/15-EJS1081}
}

@book{pourahmadiHighDimensionalCovarianceEstimation2013,
  title = {High-{{Dimensional Covariance Estimation}}: {{With High-Dimensional Data}}},
  shorttitle = {High-{{Dimensional Covariance Estimation}}},
  author = {Pourahmadi, Mohsen},
  year = {2013},
  publisher = {John Wiley \& Sons}
}

@article{ledoitPowerNonLinear2022,
  title = {The {{Power}} of ({{Non-}}){{Linear Shrinking}}: {{A Review}} and {{Guide}} to {{Covariance Matrix Estimation}}},
  shorttitle = {The {{Power}} of ({{Non-}}){{Linear Shrinking}}},
  author = {Ledoit, Olivier and Wolf, Michael},
  year = {2022},
  journal = {Journal of Financial Econometrics},
  volume = {20},
  pages = {187--218},
  doi = {10.1093/jjfinec/nbaa007}
}

@article{lamHighdimensionalCovarianceMatrix2020,
  title = {High-Dimensional Covariance Matrix Estimation},
  author = {Lam, Clifford},
  year = {2020},
  journal = {WIREs Computational Statistics},
  volume = {12},
  pages = {e1485},
  doi = {10.1002/wics.1485},
  copyright = {{\copyright} 2019 Wiley Periodicals, Inc.}
}

@inproceedings{stein1975estimation,
  title={Estimation of a covariance matrix},
  author={Stein, Charles},
  booktitle={39th Annual Meeting IMS, Atlanta, Georgia},
  year={1975}
}

@article{steinLecturesTheoryEstimation1986,
  title = {Lectures on the Theory of Estimation of Many Parameters},
  author = {Stein, Charles},
  year = {1986},
  journal = {Journal of Soviet Mathematics},
  volume = {34},
  pages = {1373--1403},
  doi = {10.1007/BF01085007}
}

@article{ledoitWellconditionedEstimatorLargedimensional2004,
  title = {A Well-Conditioned Estimator for Large-Dimensional Covariance Matrices},
  author = {Ledoit, Olivier and Wolf, Michael},
  year = {2004},
  journal = {Journal of Multivariate Analysis},
  volume = {88},
  pages = {365--411},
  doi = {10.1016/S0047-259X(03)00096-4}
}

@article{ledoitAnalyticalNonlinearShrinkage2020,
  title = {Analytical Nonlinear Shrinkage of Large-Dimensional Covariance Matrices},
  author = {Ledoit, Olivier and Wolf, Michael},
  year = {2020},
  journal = {The Annals of Statistics},
  volume = {48},
  pages = {3043--3065},
  publisher = {Institute of Mathematical Statistics},
  doi = {10.1214/19-AOS1921}
}

@article{ledoitNonlinearShrinkageEstimation2012,
  title = {Nonlinear Shrinkage Estimation of Large-Dimensional Covariance Matrices},
  author = {Ledoit, Olivier and Wolf, Michael},
  year = {2012},
  journal = {The Annals of Statistics},
  volume = {40},
  pages = {1024--1060},
  publisher = {Institute of Mathematical Statistics},
  doi = {10.1214/12-AOS989}
}

@article{ledoitOptimalEstimationLargedimensional2018,
  title = {Optimal Estimation of a Large-Dimensional Covariance Matrix under {{Stein}}'s Loss},
  author = {Ledoit, Olivier and Wolf, Michael},
  year = {2018},
  journal = {Bernoulli},
  volume = {24},
  pages = {3791--3832},
  publisher = {{Bernoulli Society for Mathematical Statistics and Probability}},
  doi = {10.3150/17-BEJ979}
}

@article{ledoitQuadraticShrinkageLarge2022,
  title = {Quadratic Shrinkage for Large Covariance Matrices},
  author = {Ledoit, Olivier and Wolf, Michael},
  year = {2022},
  journal = {Bernoulli},
  volume = {28},
  pages = {1519--1547},
  publisher = {{Bernoulli Society for Mathematical Statistics and Probability}},
  doi = {10.3150/20-BEJ1315}
}

@article{ledoitShrinkageEstimationLarge2021,
  title = {Shrinkage Estimation of Large Covariance Matrices: {{Keep}} It Simple, Statistician?},
  shorttitle = {Shrinkage Estimation of Large Covariance Matrices},
  author = {Ledoit, Olivier and Wolf, Michael},
  year = {2021},
  journal = {Journal of Multivariate Analysis},
  volume = {186},
  pages = {104796},
  doi = {10.1016/j.jmva.2021.104796}
}

@article{ledoitEigenvectorsLargeSample2011,
  title = {Eigenvectors of Some Large Sample Covariance Matrix Ensembles},
  author = {Ledoit, Olivier and P{\'e}ch{\'e}, Sandrine},
  year = {2011},
  journal = {Probability Theory and Related Fields},
  volume = {151},
  pages = {233--264},
  publisher = {Springer Nature B.V.},
  address = {Heidelberg, Netherlands},
  doi = {10.1007/s00440-010-0298-3},
  copyright = {{\copyright} Springer-Verlag 2010.}
}

@article{marcenkoDistributionEigenvaluesSets1967,
  title = {Distribution of Eigenvalues for Some Sets of Random Matrices},
  author = {Mar{\v c}enko, Vladimir Alexandrovich and Pastur, Leonid Andreevich},
  year = {1967},
  journal = {Matematicheskii Sbornik},
  volume = {114},
  pages = {507--536},
  publisher = {IOP Publishing},
  doi = {10.1070/SM1967v001n04ABEH001994}
}

@article{ledoitSpectrumEstimationUnified2015,
  title = {Spectrum Estimation: {{A}} Unified Framework for Covariance Matrix Estimation and {{PCA}} in Large Dimensions},
  shorttitle = {Spectrum Estimation},
  author = {Ledoit, Olivier and Wolf, Michael},
  year = {2015},
  journal = {Journal of Multivariate Analysis},
  volume = {139},
  pages = {360--384},
  doi = {10.1016/j.jmva.2015.04.006}
}

@article{jingNonparametricEstimateSpectral2010,
  title = {Nonparametric Estimate of Spectral Density Functions of Sample Covariance Matrices: {{A}} First Step},
  shorttitle = {Nonparametric Estimate of Spectral Density Functions of Sample Covariance Matrices},
  author = {Jing, Bing-Yi and Pan, Guangming and Shao, Qi-Man and Zhou, Wang},
  year = {2010},
  journal = {The Annals of Statistics},
  volume = {38},
  pages = {3724--3750},
  publisher = {Institute of Mathematical Statistics},
  doi = {10.1214/10-AOS833}
}

@misc{benaych-georgesShortProofLedoitPeche2023,
  title = {A Short Proof of {{Ledoit-P{\'e}ch{\'e}}}'s {{RIE}} Formula for Covariance Matrices},
  author = {{Benaych-Georges}, Florent},
  year = {2023},
  eprint = {2201.05690},
  primaryclass = {math},
  publisher = {arXiv},
  doi = {10.48550/arXiv.2201.05690},
  archiveprefix = {arxiv}
}

@article{benaych-georgesOptimalCleaningSingular2023,
  title = {Optimal Cleaning for Singular Values of Cross-Covariance Matrices},
  author = {{Benaych-Georges}, Florent and Bouchaud, Jean-Philippe and Potters, Marc},
  year = {2023},
  journal = {The Annals of Applied Probability},
  volume = {33},
  pages = {1295--1326},
  publisher = {Institute of Mathematical Statistics},
  doi = {10.1214/22-AAP1842}
}

@Article{erdosRigidityEigenvaluesGeneralized2012,
  author   = {Erd{\H o}s, L{\'a}szl{\'o} and Yau, Horng-Tzer and Yin, Jun},
  journal  = {Advances in Mathematics},
  title    = {Rigidity of Eigenvalues of Generalized {{Wigner}} Matrices},
  year     = {2012},
  pages    = {1435--1515},
  volume   = {229},
  doi      = {10.1016/j.aim.2011.12.010},
}

@article{baoStatisticalInferencePrincipal2022,
  title = {Statistical Inference for Principal Components of Spiked Covariance Matrices},
  author = {Bao, Zhigang and Ding, Xiucai and Wang, Jingming and Wang, Ke},
  year = {2022},
  journal = {The Annals of Statistics},
  volume = {50},
  pages = {1144--1169},
  publisher = {Institute of Mathematical Statistics},
  doi = {10.1214/21-AOS2143}
}

@article{bloemendalPrincipalComponentsSample2016,
  title = {On the Principal Components of Sample Covariance Matrices},
  author = {Bloemendal, Alex and Knowles, Antti and Yau, Horng-Tzer and Yin, Jun},
  year = {2016},
  journal = {Probability Theory and Related Fields},
  volume = {164},
  pages = {459--552},
  publisher = {Springer Berlin Heidelberg},
  doi = {10.1007/s00440-015-0616-x},
  copyright = {2015 Springer-Verlag Berlin Heidelberg}
}

@article{paulAsymptoticsSampleEigenstructure2007,
  title = {Asymptotics of {{Sample Eigenstructure}} for a {{Large Dimensional Spiked Covariance Model}}},
  author = {Paul, Debashis},
  year = {2007},
  journal = {Statistica Sinica},
  volume = {17},
  eprint = {24307692},
  eprinttype = {jstor},
  pages = {1617--1642},
  publisher = {Institute of Statistical Science, Academia Sinica}
}

@article{silversteinEigenvectorsLargeDimensional1989,
  title = {On the Eigenvectors of Large Dimensional Sample Covariance Matrices},
  author = {Silverstein, Jack W.},
  year = {1989},
  journal = {Journal of Multivariate Analysis},
  volume = {30},
  pages = {1--16},
  doi = {10.1016/0047-259X(89)90084-5}
}

@article{silversteinLimitTheoremsEigenvectors1984,
  title = {Some Limit Theorems on the Eigenvectors of Large Dimensional Sample Covariance Matrices},
  author = {Silverstein, Jack W.},
  year = {1984},
  journal = {Journal of Multivariate Analysis},
  volume = {15},
  pages = {295--324},
  doi = {10.1016/0047-259X(84)90054-X}
}

@article{silversteinWeakConvergenceRandom1990,
  title = {Weak {{Convergence}} of {{Random Functions Defined}} by {{The Eigenvectors}} of {{Sample Covariance Matrices}}},
  author = {Silverstein, Jack W.},
  year = {1990},
  journal = {The Annals of Probability},
  volume = {18},
  pages = {1174--1194},
  publisher = {Institute of Mathematical Statistics},
  doi = {10.1214/aop/1176990741}
}

@article{baiAsymptoticsEigenvectorsLarge2007,
  title = {On Asymptotics of Eigenvectors of Large Sample Covariance Matrix},
  author = {Bai, Zhidong and Miao, Baiqi and Pan, Guangming},
  year = {2007},
  journal = {The Annals of Probability},
  volume = {35},
  pages = {1532--1572},
  publisher = {Institute of Mathematical Statistics},
  doi = {10.1214/009117906000001079}
}

@article{xiConvergenceEigenvectorEmpirical2020,
  title = {Convergence of Eigenvector Empirical Spectral Distribution of Sample Covariance Matrices},
  author = {Xi, Haokai and Yang, Fan and Yin, Jun},
  year = {2020},
  journal = {The Annals of Statistics},
  volume = {48},
  pages = {953--982},
  publisher = {Institute of Mathematical Statistics},
  doi = {10.1214/19-AOS1832}
}

@misc{yangLinearSpectralStatistics2020,
  title = {Linear Spectral Statistics of Eigenvectors of Anisotropic Sample Covariance Matrices},
  author = {Yang, Fan},
  year = {2020},
  eprint = {2005.00999},
  primaryclass = {math, stat},
  publisher = {arXiv},
  doi = {10.48550/arXiv.2005.00999},
  archiveprefix = {arxiv}
}

@Article{baiNoEigenvaluesOutside1998,
  author    = {Bai, Zhidong and Silverstein, Jack W.},
  journal   = {The Annals of Probability},
  title     = {No Eigenvalues Outside the Support of the Limiting Spectral Distribution of Large-Dimensional Sample Covariance Matrices},
  year      = {1998},
  pages     = {316--345},
  volume    = {26},
  doi       = {10.1214/aop/1022855421},
  mrnumber  = {MR1617051},
  publisher = {Institute of Mathematical Statistics},
  zmnumber  = {0937.60017},
}

@article{donohoOptimalShrinkageEigenvalues2018,
  title = {Optimal Shrinkage of Eigenvalues in the Spiked Covariance Model},
  author = {Donoho, David and Gavish, Matan and Johnstone, Iain},
  year = {2018},
  journal = {The Annals of Statistics},
  volume = {46},
  pages = {1742--1778},
  publisher = {Institute of Mathematical Statistics},
  doi = {10.1214/17-AOS1601}
}

@misc{dingEigenvectorDistributionsOptimal2024,
  title = {Eigenvector Distributions and Optimal Shrinkage Estimators for Large Covariance and Precision Matrices},
  author = {Ding, Xiucai and Li, Yun and Yang, Fan},
  year = {2024},
  month = apr,
  number = {arXiv:2404.14751},
  eprint = {2404.14751},
  primaryclass = {math, stat},
  publisher = {arXiv},
  doi = {10.48550/arXiv.2404.14751},
  urldate = {2024-04-26},
  archiveprefix = {arXiv}
}

\end{document}